%% file: main.tex
\documentclass[reqno,12pt]{amsart}
\usepackage[margin=2cm]{geometry}
\usepackage{
    mlmodern,
    amssymb,
    mathtools,
    amsthm,
    xcolor,
    hyperref,
    cleveref,
    bbm,
    dsfont,
    verbatim,
    aligned-overset,
    array,
}
\usepackage[shortlabels]{enumitem}
    \setlist{
        labelindent	=\parindent,
        leftmargin	=*,
        font		=\normalfont
    }
    \setlist[enumerate]{
        label		=(\roman*),
        ref			=(\roman*)
    }
\usepackage{tikz}
    \usetikzlibrary{positioning,arrows.meta}
    \tikzset{
        x=4\baselineskip,y=4\baselineskip,
        baseline=(current bounding box.center),
        mor/.style={
            -stealth,
            font=\scriptsize,
            line width=1pt
        },
        nat/.style={
            font=\scriptsize,
            arrows={-stealth},
            double equal sign distance,
            every edge/.append style={double}
        },
        igual/.style={
            font=\scriptsize,
            arrows={-},
            double equal sign distance,
            every edge/.append style={double}
        },
    }
\allowdisplaybreaks

\newtheorem{thrm}{Theorem}[section]
\newtheorem{Theorem}[thrm]{Theorem}

\newtheorem{Corollary}[thrm]{Corollary}

\newtheorem{Lemma}[thrm]{Lemma}

\newtheorem{Proposition}[thrm]{Proposition}
\theoremstyle{definition}
    
    \newtheorem{Definition}[thrm]{Definition}
    
    \newtheorem{Remark}[thrm]{Remark}
    
    \newtheorem{Example}[thrm]{Example}

\crefrangeformat{equation}{#3(#1)#4--#5(#2)#6}

\crefname{thrm}{Theorem}{Theorems}
\crefname{Theorem}{Theorem}{Theorems}
\crefname{lem}{Lemma}{Lemmas}
\crefname{Lemma}{Lemma}{Lemmas}
\crefname{cor}{Corollary}{Corollaries}
\crefname{Corollary}{Corollary}{Corollaries}
\crefname{prop}{Proposition}{Propositions}
\crefname{Proposition}{Proposition}{Propositions}
\crefname{defn}{Definition}{Definitions}
\crefname{Definition}{Definition}{Definitions}
\crefname{exm}{Example}{Examples}
\crefname{Example}{Example}{Examples}
\crefname{rem}{Remark}{Remarks}
\crefname{Remark}{Remark}{Remarks}
\crefname{section}{Section}{Sections}
\crefname{equation}{\unskip}{\unskip}
\crefname{enumi}{\unskip}{\unskip}

\newcommand{\id}{\mathrm{id}}
\newcommand{\Id}{\mathrm{Id}}

\newcommand{\End}{{\sf End}}
\newcommand{\Aut}{{\sf Aut}}

\DeclareMathOperator{\Hom}{Hom}

\newcommand{\m}{{}^{-1}}

\newcommand{\vf}{\varphi}

\newcommand{\g}{\gamma}

\newcommand{\cl}{\overline}

\newcommand{\sst}{\subseteq}

\newcommand{\C}{\ensuremath{\mathcal{C}}}
\newcommand{\ot}{\otimes}

\newcommand{\tto}{\Rightarrow}

\newcommand{\iiso}{\overset{\sim}{\Rightarrow}}
\newcommand{\Iso}{\ensuremath{\cong}}

\newcommand{\cI}{\mathcal I}
\newcommand{\cO}{\mathcal O}
\newcommand{\D}{\mathcal D}

\def \Eg   {{g}}
\def \Egi  {{g^{-1}}}
\def \Eh   {{h}}
\def \Ehi  {{h^{-1}}}
\def \Egh  {{gh}}
\def \Eghi {{(gh)^{-1}}}

\def \Cg   {{\C_{\Eg}}}
\def \Cgi  {{\C_{\Egi}}}
\def \Ch   {{\C_{\Eh}}}
\def \Chi  {{\C_{\Ehi}}}
\def \Cgh  {{\C_{\Egh}}}
\def \Cghi {{\C_{\Eghi}}}

\def \esp{{{}\mathop{\underline{\phantom{x}}}{}}}

\let\longto\longrightarrow
\def \1{\mathds {1}}
    \def \ug   {{\1_{\Eg}}}
    \def \ugi  {{\1_{\Egi}}}
    \def \uh   {{\1_{\Eh}}}
    \def \uhi  {{\1_{\Ehi}}}
    \def \ugh  {{\1_{\Egh}}}
    \def \ughi {{\1_{\Eghi}}}

\DeclareMathOperator{\Tr}{Tr} 

\def \Tg   {{T_{\Eg}}}

\def \Th   {{T_{\Eh}}}

\def \Tgh  {{T_{\Egh}}}

\def\smp{\mathbin{\underline{\rtimes}}}

\begin{document}\frenchspacing
\title[Partial actions of groups on monoidal categories]{Partial actions of groups on monoidal categories}

\author{Eliezer Batista}
\address{Departamento de Matem\'atica, Universidade Federal de Santa Catarina, Campus Trindade, Florian\'opolis, SC, CEP: 88040--900, Brazil}
\email{eliezer1968@gmail.com}

\author{Felipe Castro}
\address{Departamento de Matemática, Universidade Federal de Santa Catarina, Campus Trindade, Florianópolis, SC, CEP: 88040--900, Brazil}
\email{f.castro@ufsc.br}

\author{Mykola Khrypchenko}
\address{Departamento de Matem\'atica, Universidade Federal de Santa Catarina, Campus Trindade, Florian\'opolis, SC,  CEP: 88040--900, Brazil}
\email{nskhripchenko@gmail.com}

\subjclass[2020]{Primary 18M05; Secondary 18D25.}
\keywords{Partial action, strict monoidal category}

\begin{abstract}
	In this work, we introduce the notion of a partial action of a group on a strict monoidal category. We propose, in the context of Monoidal categories, new constructions analogous to those existing for partial group actions over an algebra such as the globalization, the subalgebra of partial invariants, and the partial smash product.
\end{abstract}

\maketitle

\tableofcontents

\section{Introduction}

Partial actions of groups were first introduced in the context of the theory of C*-algebras in order to describe some $\mathbb{Z}$-graded C*-algebras which are not isomorphic to a crossed product \cite{E-1}. Afterwards, a complete algebraic description of partial actions of groups on algebras and partial crossed products, including the Globalization Theorem, was given in \cite{DE}. Since then, the theory of partial group actions grew rapidly and gained a relevant place, both within pure algebra and in the theory of dynamical systems \cite{D}. 

In this work, we extend the existing boundaries in the theory of partial actions, proposing notions analogous to partial group action over an algebra, its globalization, the subalgebra of partial invariants and the partial smash product, to the context of Monoidal categories. For this purpose, we follow basically the constructions due to D. Tambara in \cite{Tambara} for group actions on Monoidal categories. 

There are several technical obstacles to be overcome in order to give meaning to partial group actions on monoidal categories. Firstly, a group can act partially on an algebra $A$ even when $A$ has no unity. Then it would be interesting to broaden our context to categories which admit a tensor product among its objects but not necessarily have a unit object. We call these categories semigroupal categories, among them there are those which are monoidal, that is, semigroupal categories which admit unit object. This notion is very general, including the poset of open subsets of a topological space, in which the tensor product is the intersection of open subsets can be viewed as a semigroupal category. 

In order to define a partial action of a group $G$ on an algebra $A$, we have to choose a family of ideals $\{ A_g \}_{g\in G}$ and a family of algebra isomorphisms $\{ \alpha_g :A_{g^{-1}}\rightarrow A_g \}_{g\in G}$ satisfying some compatibility conditions to be composed. We dedicate some effort to making sense of the notion of a tensor ideal on a semigroupal category. One specific example of a partial group action on an algebra  is given when the ideals $A_g$ are principal ideals generated by central idempotent elements $\1_g \in A$, in particular, the algebra $A$ itself must be unital. These are the partial actions to which the globalization theorem applies \cite{DE}. In semigroupal and monoidal categories, one can define central idempotent objects \cite{BD-idemp-2014} and then construct special cases of partial actions. From a partial action of a group $G$ on a monoidal category $\C $ generated by central idempotents one construct the globalization by embedding the category $\C $ into a subcategory of functors from $G$ (viewed as a monoidal category) to $\C $. This construction follows the ideas given in \cite{DE} for the globalization of a unital partial action of a group $G$ on a unital algebra $A$.

Beyond the globalization, one can define two new monoidal categories out of a partial action of a group $G$ on a monoidal category $\C $: the partial equivariantization, $\C ^{ \underline{G}}$, and the partial smash product $\C  \smp G$. These are generalizations of the equivariantization $\C ^G$ and the semi-direct product $\C  [G]$, introduced by D. Tambara in \cite{Tambara}. In particular, the smash product makes sense only when the category $\C $ also admits additive structure, then we assume that $\C $ is a $\Bbbk$-linear, abelian monoidal category\footnote{In \cite{Tambara}, the author calls these categories {\emph{Tensor Categories}}, but this is not the most common nomenclature used nowadays. For example, in \cite{EGNO}, the term {\emph{Tensor Category}} means something much more restrictive.  }.

This paper is divided as follows: In Section 2, we recall basic definitions of monoidal categories and functors and introduce the notions of a {\emph{Semigroupal category}}, consisting of a category with tensor products among its objects, but without imposing the  existence of a unit object, and of a {\emph{Semigroupal functor}}, which is a functor making compatible the respective tensor products of the semigroupal categories. Also in this preliminary section, we introduce the notion of tensor ideals in semigroupal categories, idempotent objects in semigroupal categories \cite{BD-idemp-2014}  and group actions on semigroupal categories \cite{EGNO,Tambara}. In Section 3, partial actions of groups on semigroupal categories and their respective morphisms are defined. We will focus mainly on a particular type of partial action of a group on a monoidal category, which is given when the tensor ideals of the partial actions are generated by central idempotent objects in the category. A small digression is made in Section 4 just to prove that any partial action of a group on a semigroupal category gives rise to a Hopf Polyad in the sense of A. Bruguières \cite{Bruguieres}. Section 5 is devoted to the construction of the standard globalization of a partial group action on a monoidal category in the case when the partial action is generated by central idempotents. Finally, in Section 6, we introduce and discuss some properties of two new monoidal categories associated to a partial action of a group on a semigroupal or monoidal category, namely, the partial $G$-equivariantization and the partial smash product, this last one being well defined only for partial group actions on $\Bbbk$-linear, abelian, monoidal categories. In Section 7, we give some directions for future investigations.

\section{Preliminaries}\label{sec-prelim}

Throughout this paper, we will adopt, for a category $\C $, the following convention:
\begin{enumerate}
    \item $\C ^{(0)}$ denotes the class of objects of $\C $.
    \item $\C ^{(1)}$ denotes the class of all morphisms of $\C $.
    \item Given two objects $X,Y\in \C ^{(0)}$, denote by $\Hom_{\C } (X,Y)$ the set of morphisms in $\C $ from $X$ to $Y$ (here, all categories will be locally small).
    \item $\C ^{(2)}$ denote the class of pairs of composable morphisms of $\C $.
\end{enumerate}

\subsection{Monoidal categories and functors}
\label{sec-mcat}
\phantom{|}\medskip
\begin{Definition}\label{mon-cat-defn}
	A \emph{semigroupal category} is a triple $(\C,\ot,a)$, where
	\begin{enumerate}
		\item $\C$ is a category,
		\item $\ot \colon \C\times\C\to\C$ is a (covariant) functor, called the \emph{tensor product},
		\item $a \colon (\esp \ot \esp ) \ot \esp \iiso \esp \ot ( \esp \ot \esp)$ is a natural isomorphism, called the \emph{associator},
	\end{enumerate}
	such that the following pentagon diagram commutes.
	\begin{equation}
        \label{pent-ax}
		\input{diagramas/cat-mon-1.tikz}
	\end{equation} 
	A semigroupal category $(\C,\ot,a)$ is said to be \emph{monoidal} if it admits a triple $(\1,l,r)$, where
	\begin{enumerate}
		\item $\1$ is an object of $\C$, called the \emph{unit} of $\C$,
		\item $l \colon \1\ot \esp \iiso \Id$ and $r \colon \esp \ot \1 \iiso \Id$ are natural isomorphisms, called the \emph{left and right unit isomorphisms},
	\end{enumerate}
	and the triangle diagram below commutes.
    \begin{equation}
        \label{tri-ax}
		\input {diagramas/cat-mon-2.tikz}
    \end{equation}
\end{Definition}
A semigroupal category is called \emph{strict}, if $a$ is the identity transformation. In a \emph{strict monoidal} category one additionally assumes that $l$ and $r$ are identity.

\begin{Example}\label{categorymonoid}
    Each semigroup (resp. monoid) $M$ can be seen as a strict semigroupal (resp. monoidal) category whose objects are elements of $M$ and for all $x,y\in M$
    \[
    \Hom(x,y) =
        \begin{cases}
            \{ \ast \}, & x=y,\\
            \emptyset, & x\ne y.
        \end{cases}
    \]
    The tensor product in this category is the multiplication in $M$, and $\id_x\ot\id_y=\id_{xy}$.
    
    In particular, when the monoid is a group $G$, we will denote this category as $\underline{G}$.
\end{Example}

\begin{Example}
    Given a topological space $X$, one can define a strict monoidal category $\mathcal{O}(X)$, whose objects are the open subsets of $X$ and the morphisms between two open subsets $A$ and $B$ are given by
    \[
    \Hom(A,B) =
        \begin{cases}
            \{ \ast \}, & A\subseteq B,\\
            \emptyset, & A\nsubseteq B.
        \end{cases}
    \]
    The tensor product in this category is given by the intersection between open subsets and the unit object is given by the whole space $X$. It is easy to see that this monoidal category is strict.
\end{Example}

\begin{Example}
    Consider a non unital associative ring $R$ and the category ${}_R\mathcal{M}^{firm}_R$ of firm bimodules over $R$. An $R$-bimodule $M$ is called firm if $R\otimes_R M \cong M \cong M \otimes_R R$. In this case, the category ${}_R\mathcal{M}^{firm}_R$ is semigroupal with the tensor product balanced over $R$, but $R$ is not necessarily the monoidal unit, unless $R$ is itself a firm $R$-bimodule. 
\end{Example}

\begin{Definition}\label{mon-funct-defn}
	Let $(\C,\ot,a)$ and $(\C',\ot',a')$ be semigroupal categories. A \emph{semigroupal functor} $\C\to\C'$ is a pair $(F,J)$, where
	\begin{enumerate}
		\item $F:\C\to\C'$ is a functor,
		\item $J:F(\esp)\ot'F(\esp)\iiso F(\esp \ot \esp)$ is a natural isomorphism,\label{F(-)-ot-F(-)-cong-F(-ot-)}
	\end{enumerate}
	such that the following hexagon diagram commutes.
	\begin{equation}
        \label{hex-J-a}
		\input{diagramas/cat-subg-1.tikz}
	\end{equation}
	A semigroupal functor between two monoidal categories is called a \emph{monoidal} functor if, additionally, it respects the unital structures $(\1,l,r)$ and $(\1',l',r')$ of the monoidal categories in the sense that there exists an isomorphism $J^0 \colon \1'\to F(\1)$ making the following diagrams commute.
    \begin{gather}
        \label{square-J-l}
		\input{diagramas/cat-subg-2.tikz}
        \\
        \label{square-J-r}
        \input{diagramas/cat-subg-3.tikz}
    \end{gather}
\end{Definition}

By a \emph{semigroupal} (resp. \emph{monoidal}) equivalence of semigroupal (resp. monoidal) categories we mean a semigroupal (resp. monoidal) functor $(F,J)$ (resp. $(F,J,J^0)$), such that $F$ is an equivalence of categories. 

\begin{Remark}
     The famous Mac Lane's strictness theorem states that every monoidal category is monoidally equivalent to a strict monoidal category \cite{MacLane63} (cf. also \cite[Theorem 2.8.5]{EGNO}). The proof of the coherence theorem, as done in \cite{Categories}, can also be adapted for semigroupal categories, then one can consider, without loss of generality, our semigroupal  categories as being strict. From now on, unless stated otherwise, we will be using strict semigroupal/monoidal categories.
\end{Remark}

\begin{Definition}\label{morph-mon-funct-defn}
	Let $(\C,\ot,a)$ and $(\C',\ot',a')$ be semigroupal categories, $(F,J)$ and $(\widetilde{F},\widetilde{J})$ semigroupal functors from $\C$ to $\C'$. A \emph{morphism} of semigroupal functors from $(F,J)$ to $(\widetilde{F},\widetilde{J})$ is a natural transformation $\eta \colon F\tto \widetilde{F}$, such that the following diagram commutes.
	\begin{equation}
        \label{square-J-eta}
        \input{diagramas/mor-subg-1.tikz}
	\end{equation}
	A \emph{morphism of monoidal functors} from $(F,J,J^0)$ to $(\widetilde{F},\widetilde{J},\widetilde{J}^0)$ should additionally make the following diagram commute.
	\begin{equation}
        \label{tri-vf-eta}
        \input{diagramas/mor-subg-2.tikz}
	\end{equation}
\end{Definition}

\subsection{Tensor ideals in semigroupal categories}
\label{sec-ideals}
\phantom{|}\medskip

Let $\C$ be a category and $\cO$ a class of objects of $\C$. We say that $\cO$ is \emph{replete} if it is closed under isomorphisms, that is, whenever for any $X\in\cO$ and $X'\in \C^{(0)}$ if $X'\Iso X$, then $X'\in\cO$. We define the \emph{isomorphism closure} of $\cO$ as
\[
	\cl\cO=\{X'\in\C\mid X'\Iso X\text{ for some }X\in\cO\}.
\]
Clearly, the class $\cl\cO$ is replete, and it is the minimal class with this property containing $\cO$. 

More generally~\cite[Exercise I.4B]{AHS}, let $\D$ be a subcategory of $\C$, then $\D$ is said to be \emph{replete} if for any $X\in \D^{(0)}$ and any isomorphism $\vf:X\to X'$ in $\C$ one has $X'\in \D^{(0)}$ and $\vf\in\D^{(1)}$. The minimum subcategory of $\C$ which is replete and contains $\D$ will be called \emph{the isomorphism closure} of $\D$ and denoted by $\cl\D$. Observe that $(\cl\D)^{(0)}=\cl{\D^{(0)}}$ and, for $X'$ and $Y'$ objects in $\overline{\D}$, a morphism $\vf'\in\Hom_{\C}(X',Y')$ belongs to $\Hom_{\overline{\mathcal{D}}}(X' ,Y')$ if and only if there are $X,Y\in \D^{(0)}$, isomorphisms $\mu:X\to X'$, $\nu:Y\to Y'$ in $\C$ and a morphism $\vf :X\to Y$ in $\D$, such that the following diagram commutes.
\begin{equation}
    \label{morphism-in-cl-D}
    \input{diagramas/replete.tikz}
\end{equation}
In particular, if $X$ and $Y$ are objects in $\mathcal{D}$, then $\Hom_{\D}(X,Y)=\Hom_{\overline{\D}}(X,Y)$.

Now let $\C$ be a semigroupal category. A subcategory $\D$ of $\C$ is said to be \emph{closed under left (resp. right) tensor product}, if for all $X\in \D^{(0)}$ and $Y\in \C^{(0)}$ one has $Y\ot X\in \D^{(0)}$ (resp. $X\ot Y\in \D^{(0)}$).

\begin{Definition}\label{ideal-defn}
	A subcategory $\cI$ of a semigroupal category $\C$ is a \emph{left (resp. right) tensor ideal} of $\C$ if it is replete and is closed under left (resp. right) tensor product. A \emph{tensor ideal} of $\C$ is a subcategory $\cI\sst\C$ which is both a left and right tensor ideal of $\C$.
\end{Definition}

Our notion of a left (right) tensor ideal is a generalization of the notion of a left (right) tensor ideal in the sense of~\cite{BD-idemp-2014}. Observe that (on the object level) our concept of a tensor ideal is slightly weaker than the notion of an ideal in a monoidal category given in~\cite{GPMV}, which also demanded the class of objects being closed under retractions. 

\begin{Remark}\label{intersect-ideal}
	It is clear that the intersection of two (left, right) tensor ideals of $\C$ is a (left, right) tensor ideal of $\C$.
\end{Remark}

\begin{Lemma}\label{closure-F(I)-ideal}
	Let $F$ be a semigroupal equivalence $\C\to\C'$. Then for any (left, right) tensor ideal $\cI$ of $\C$ the subcategory\footnote{Note that $F(\cI)$ is a subcategory of $\C'$, as $F$ is full (see \cite[Remark I.4.2 (3)]{AHS}).} $\cl{F(\cI)}$ is a (left, right) tensor ideal of $\C'$. Moreover, $F$ establishes an equivalence between $\cI$ and $\cl{F(\cI)}$.
\end{Lemma}
\begin{proof}
	We consider the case of a left tensor ideal. Let $X'\in\cl{F(\cI)}$ and $Y'\in\C'$. We need to prove that $Y'\ot X'\in\cl{F(\cI)}$. There is $X\in\cI$, such that $X'\Iso F(X)$. Moreover, since $F$ is a categorical equivalence, the functor $F$ is essentially surjective, that is, there is an object $Y\in \C$ such that $Y'\Iso F(Y)$. Then $Y'\ot X'\Iso F(Y)\ot F(X)$. But $F(Y)\ot F(X)\Iso F(Y\ot X)$ by \cref{F(-)-ot-F(-)-cong-F(-ot-)} of \cref{mon-funct-defn}. Hence, $Y'\ot X'\Iso F(Y\ot X)$. It remains to note that $Y\ot X\in\cI$, as $\cI$ is a left tensor ideal.
	
	In order to prove the second affirmation, note that, by definition, $X'\in \overline{F(\mathcal{I})}$ means that there is an object $X\in \mathcal{I}$ such that $X' \Iso F(X)$, therefore, the functor $F|_{\mathcal{I}} :\mathcal{I} \rightarrow \overline{F(\mathcal{I})}$ is essentially surjective. Moreover, the same functor is fully faithful: given two objects $X$ and $Y$ in $\mathcal{I}$, we have that $F(X)$ and $F(Y)$ are objects in $F(\mathcal{I})$, then
	\[
	\Hom_{\cl{F(\cI)}}
    (F(X),F(Y))
    = \Hom_{F(\mathcal{I})} (F(X),F(Y))
    = F(\Hom_{\cI}(X,Y)),
	\]
	where the last equality comes from the fact that the functor $F$ is an equivalence between $\C$ and $\C'$ and $\mathcal{I}$ is a subcategory of $\C$.
\end{proof}

\begin{Definition}\label{restr-F-to-I}
	By the \emph{restriction} of a semigroupal (resp. monoidal) equivalence $F:\C\to\C'$ to a tensor ideal $\cI$ of $\C$ we mean the induced equivalence between $\cI$ and $\cl{F(\cI)}$.
\end{Definition}

\subsection{Idempotent objects in semigroupal categories}
\label{sec:idemp-mcat}
\phantom{|}\medskip

\begin{Definition}\label{idempotentecentral}
	A \emph{central idempotent} in a semigroupal category \(\C\) is a triple \(\{e, \Phi_e, \sigma^e\}\), in which \(e\) is an object in \(\C\), \(\Phi_e \colon e \otimes e \to e\) is an isomorphism in \(\C\), henceforth called the \emph{fusion isomorphism}, and \(\sigma^e \colon e \otimes \esp \Rightarrow \esp \otimes e\) is a natural isomorphism in \(\C\), henceforth called the \emph{exchange isomorphism}, such that the following diagrams commute:
    \begin{gather}
        \label{idempotent1}
    	\input{diagramas/id-cent-1.tikz}
    \\
        \label{idempotent2}
    	\input{diagramas/id-cent-2.tikz}
    \\
        \label{idempotent3}
    	\input{diagramas/id-cent-3.tikz}
    \\
        \intertext{and}
        \label{idempotent4}
    	\input{diagramas/id-cent-4.tikz}
    \end{gather}
\end{Definition}

If the semigroupal category is monoidal, it is interesting to explore the relation between idempotent objects and the monoidal unit. That approach was done in \cite{BD-idemp-2014}, in which two classes of idempotents were introduced: the closed and open idempotents. The \emph{closed idempotents} are objects $e$ in $\C$ with a morphism $\pi :\1\rightarrow e$ such that $(\pi\otimes e)\circ l^{-1}_e :e\rightarrow e\otimes e$ and $(e\otimes \pi )\circ r^{-1}_e :e\rightarrow e\otimes e$ are isomorphisms. The \emph{open idempotents}, in their turn, are objects $e$ in $\C$ with a morphism $\imath :e\rightarrow \1$ such that $l_e \circ (\imath\otimes e) :e\otimes e\rightarrow e$ and $r_e \circ (e\otimes \imath ) :e\otimes e \rightarrow e$ are isomorphisms. Our approach allows one to define idempotent objects even when there is no unit object in the whole category.

\begin{Example}
    Consider the monoidal category $\mathcal{O} (X)$, consisting of the open sets of a topological space $X$. Any open subset $A\subseteq X$ is a central idempotent in this monoidal category, whose tensor product is given by the intersection. Here, the morphisms $\Phi_e$ and $\sigma^e$ are the identity.
\end{Example}

\begin{Example}
    Consider the monoidal category $({}_R \mathcal{M}, \otimes_R, R)$ of modules over a commutative ring $R$. Let $I\trianglelefteq R$ be an ideal of $R$ such that $I^2=I$. It is easy  to see that $I$ is an open idempotent object in ${}_R \mathcal{M}$ with the inclusion map $\imath : I \rightarrow R$. It is central because the category ${}_R \mathcal{M}$ is symmetric.
    
    On the other hand, the quotient $R/I$ is a closed idempotent with the canonical map $\pi: R \rightarrow R/I$. In order to verify that the map $R/I \otimes_R \pi :  R/I \otimes_R R  \rightarrow  R/I \otimes_R R/I $ is indeed an isomorphism, we only need to prove that it is injective, since it is already surjective. Consider an element
    \[
    \sum_{i=1}^n (a_i + I) \otimes_R b_i \in \mathop{\text{Ker}}(R/I \otimes_R \pi ) .
    \]
    Then we have
    \[
    \sum_{i=1}^n (a_i + I) \otimes_R (b_i +I)= 0.
    \]
    Therefore (see, for example, \cite[Subsection 2.12.10]{Wisbauer}), there exist elements $(x_j +I) \in R/I$, for $j\in \{ 1, \ldots , m \}$, and $r_{ji} \in R$ such that $\sum_{j=1}^m (x_j +I)\triangleleft r_{ji} =(a_i +I)$ for every $i\in \{ 1, \ldots , n \}$ and $\sum_{i=1}^n r_{ji} \triangleright (b_i +I) =0$ for every $j\in \{ 1, \ldots , m \}$. This implies that $\sum_{i=1}^n r_{ji} b_i \in I$ and, therefore, 
    \[
    \sum_{i=1}^n (a_i + I) \otimes_R b_i =\sum_{j=1}^m (x_j +I) \otimes_R \sum_{i=1}^n r_{ji}b_i =\sum_{j=1}^m \left( \sum_{i=1}^n x_j r_{ji}b_i  +I \right) \otimes_R 1_R =0.
    \]
\end{Example}

\begin{Example}
    Another source of idempotents consists of localizations of commutative rings. Given a multiplicative set $S$, which does not contain divisors of $0$, in a commutative ring $R$, the localization $S^{-1}R$ is an open central idempotent in the monoidal category ${}_R \mathcal{M}$ with the canonical map $\imath :R\rightarrow S^{-1}R$, sending an element $a$ into the fraction $\dfrac{a}{1}$. To prove this, we need to check that the map $S^{-1}R \otimes_R \imath :S^{-1}R \otimes_R R \rightarrow S^{-1}R \otimes_R S^{-1} R$ is an isomorphism. Let us create directly its inverse taking an element 
    \[
        \dfrac{a}{b} \otimes_R \dfrac{c}{d} \in S^{-1}R \otimes_R S^{-1} R .
    \]
    By an elementary manipulation, we have
    \[
        \dfrac{a}{b} \otimes_R \dfrac{c}{d} = \dfrac{ad}{bd} \otimes_R \dfrac{c}{d} = \dfrac{a}{bd} \otimes_R \dfrac{cd}{d} =\dfrac{a}{bd} \otimes_R \dfrac{c}{1} .
    \]
    Then, the obvious inverse to $S^{-1}R \otimes_R \imath$ is the map
    \begin{align*}
        \phi : S^{-1}R \otimes_R S^{-1} R & \longrightarrow S^{-1}R \otimes_R R \cong S^{-1}R, \\
        \dfrac{a}{b} \otimes_R \dfrac{c}{d} & \longmapsto \dfrac{a}{bd} \otimes_R c.
    \end{align*}
\end{Example}

\begin{Theorem}
    Let $\C$ be a semigroupal category and $e$ be a central idempotent in $\C$. Then $\overline{e\otimes \C}$ defines a tensor ideal in $\C$. Moreover, if $\C $ is a monoidal category, then $\overline{e\otimes \C }$ is a monoidal category
\end{Theorem}
\begin{proof}
    We have by construction that $\overline{e\otimes \C}$ is closed by isomorphisms. Now, given $X\in \C^{(0)}$ and $Y\in \overline{e\otimes \C}$, there is an object $Y' \in \C$ such that $Y\Iso e\otimes Y'$. It is obvious that $Y\otimes X \Iso e\otimes Y' \otimes X$, then $Y\otimes X\in \overline{e\otimes \C}$. On the other hand, $X\otimes Y \Iso X\otimes e \otimes Y' \Iso e\otimes X \otimes Y'$, in which the last isomorphism is performed by $(\sigma^e_X)^{-1} \otimes Y'$. Therefore $X\otimes Y \in \overline{e\otimes \C}$.
    
    For the case of $\C $ being a monoidal category, the object $e \cong e\otimes \1_{\C } \in \overline{e\otimes \C} $. Moreover, for any $X\in \overline{e\otimes \C}^{(0)}$, there is an object $X' \in \C ^{(0)}$ and an isomorphism $\varphi_X :X\rightarrow e\otimes X'$, then, we have the following sequence of isomorphisms
    \[
    e\otimes X
        \xrightarrow{e\otimes \varphi_X}
    e\otimes e \otimes X'
        \xrightarrow{\Phi_e \otimes X'}
    e\otimes X'
        \xrightarrow{\varphi_X^{-1}}
    X .
    \]
    Denote this sequence of isomorphisms by $L^e_X :e\otimes X \rightarrow X$, which is an isomorphism in the category $\overline{e\otimes \C }$.
    This definition should be independent of the choice of the particular isomorphism $\varphi_X :X\rightarrow e\otimes X'$, meaning that, for any other isomorphism $\psi_X :X\rightarrow e\otimes X''$, and denoting by $\theta :e\otimes X' \rightarrow e\otimes X''$ the isomorphism such that $\theta \circ \varphi_X =\psi_X$, the following diagram commutes:
    \begin{center}
    \input{diagramas/teo-idemp-mon-1.tikz}
    \end{center}
    The naturality of $L^e :e\otimes \esp\Rightarrow \Id_{\overline{e\otimes \C }}$ can be easily obtained. Indeed, consider $f\in \Hom_{\overline{e\otimes \C }}(X,Y)$, then there are isomorphisms $\varphi_X :X\rightarrow e\otimes X'$ and $\varphi_Y :Y \rightarrow e\otimes Y'$ and a morphism $f':X' \rightarrow Y'$ such that the following diagram commutes
    \begin{center}
    \input{diagramas/teo-idemp-mon-2.tikz}
    \end{center}
    Then
    \begin{align*}
    f\circ L^e_X
    & =
        f\circ \varphi_X^{-1}\circ (\Phi_e \otimes X')\circ (e\otimes \varphi_X) \\
    & =
        \varphi_Y^{-1} \circ (e\otimes f')\circ (\Phi_e \otimes X')\circ (e\otimes \varphi_X) \\
    & =
        \varphi_Y^{-1} \circ (\Phi_e \otimes Y')\circ (e\otimes e \otimes f')\circ (e\otimes \varphi_X) \\
    & =
        \varphi_Y^{-1} \circ (\Phi_e \otimes Y')\circ (e\otimes \varphi_Y) \circ (e\otimes f)\\
    & =
        L^e_Y \circ (e\otimes f). 
    \end{align*}
    Denote by $R^e_X :X\otimes e \rightarrow X$ the composition
    \[
    X \otimes e
        \xrightarrow{(\sigma^e_X)^{-1}}
    e \otimes X
        \xrightarrow{L^e_X}
    X
    \]
    As $R^e :\esp \otimes e \Rightarrow \Id_{\overline{e\otimes \C }}$ is a composition of two natural transformations, then it is itself a natural transformation.
    Therefore, $\overline{e\otimes \C}$ is a monoidal category with monoidal unit $e$. In fact, the triangle axiom can be easily proved using the compatibility relations in Definition \ref{idempotentecentral}. On the one hand we have
    \begin{align*}
        X\otimes L^e_Y
    & =
        (X\otimes \varphi_Y^{-1})\circ (X\otimes \Phi_e \otimes Y')\circ (X\otimes e \otimes \varphi_Y)\\
    & =
        (\varphi_X^{-1} \otimes \varphi_Y^{-1})\circ (e\otimes X' \otimes \Phi_e \otimes Y')\circ (\varphi_X \otimes e \otimes \varphi_Y).
    \end{align*}
    On the other hand, we have
    \begin{align*}
    &\hskip-1,2em
        R^e_X \otimes Y = (\varphi_X^{-1} \otimes Y)\circ (\Phi_e \otimes X' \otimes Y)\circ (e\otimes \varphi_X \otimes Y)\circ ( (\sigma^e_X)^{-1} \otimes Y) \\
    \overset{(I)}&=
        (\varphi_X^{-1} \otimes Y)\circ (\Phi_e \otimes X' \otimes Y) \circ ( (\sigma^e_{e\otimes X'})^{-1} \otimes Y)  \circ  (\varphi_X \otimes e \otimes Y) \\
    \overset{(II)}&=
        (\varphi_X^{-1} \otimes Y)\circ (\Phi_e \otimes X' \otimes Y) \circ ( (\sigma^e_{e})^{-1} \otimes X' \otimes Y)\circ ( e\otimes (\sigma^e_{X'})^{-1} \otimes Y)  \circ  (\varphi_X \otimes e \otimes Y) \\
    \overset{(III)}&=
        (\varphi_X^{-1} \otimes Y)\circ (\Phi_e \otimes X' \otimes Y) \circ ( e\otimes (\sigma^e_{X'})^{-1} \otimes Y)  \circ  (\varphi_X \otimes e \otimes Y) \\
    \overset{(IV)}&=
        (\varphi_X^{-1} \otimes Y)\circ ( (\sigma^e_{X'})^{-1} \otimes Y)\circ (X'\otimes  \Phi_e \otimes Y) \circ  ( \sigma^e_{X'} \otimes e \otimes  Y) \circ  (\varphi_X \otimes e \otimes Y) \\
    \overset{(V)}&=
        (\varphi_X^{-1} \otimes \varphi_Y^{-1})\circ ( (\sigma^e_{X'})^{-1} \otimes e \otimes Y')\circ (X'\otimes  \Phi_e \otimes e \otimes Y') \circ  ( \sigma^e_{X'} \otimes e \otimes e\otimes  Y') \circ  (\varphi_X \otimes e \otimes \varphi_Y) \\
    \overset{(VI)}&=
        (\varphi_X^{-1} \otimes \varphi_Y^{-1})\circ ( (\sigma^e_{X'})^{-1} \otimes e \otimes Y')\circ (X'\otimes e \otimes  \Phi_e \otimes Y') \circ  ( \sigma^e_{X'} \otimes e \otimes e\otimes  Y') \circ  (\varphi_X \otimes e \otimes \varphi_Y) \\
    \overset{(VII)}&=
    	(\varphi_X^{-1} \otimes \varphi_Y^{-1})\circ ( (\sigma^e_{X'})^{-1} \otimes e \otimes Y')\circ  ( \sigma^e_{X'} \otimes e \otimes  Y') \circ (e\otimes X' \otimes  \Phi_e \otimes Y') \circ  (\varphi_X \otimes e \otimes \varphi_Y) \\
    &=
    	(\varphi_X^{-1} \otimes \varphi_Y^{-1})\circ (e\otimes X' \otimes  \Phi_e \otimes Y') \circ  (\varphi_X \otimes e \otimes \varphi_Y) .
    \end{align*}
    Here, equality (I) follows from the naturality of $\sigma^e$. Equality (II) is a consequence of diagram (\ref{idempotent4}). Diagram (\ref{idempotent2}) is responsible for equality (III), and (IV) follows from (\ref{idempotent3}). In (V) we just open $Y$ into $e\otimes Y'$, and close it again by the isomorphism $\varphi_Y$. Equality (V) is simply a consequence of diagram (\ref{idempotent1}) taking into account that $\Phi_e$ is an isomorphism. Finally, equality (VII) comes again from the naturality of $\sigma^e$.
    
    Therefore, $\overline{e\otimes \C }$ is a monoidal category with monoidal unit given by $e$.
\end{proof}

\subsection{Group actions on semigroupal categories}
\label{sec-act-mcat}
\phantom{|}\medskip

Given a semigroupal (resp. monoidal) category $\C$, we denote by $\End_\ot(\C)$ the category of semigroupal (resp. monoidal) endofunctors of $\C$ (and morphisms between them as in \cref{morph-mon-funct-defn}). This is a strict monoidal category with tensor product of objects being the composition of functors and tensor product of morphisms being the horizontal composition of natural transformations. The full subcategory of $\End_\ot(\C)$ consisting of the semigroupal (resp. monoidal) auto-equivalences of $\C$ will be denoted by $\Aut_\ot(\C)$.

\begin{Definition}\label{act-of-G-on-C}
	An {\emph{action}} of a group $G$ on a semigroupal (resp. monoidal) category $\C$ is a monoidal functor from the monoidal category $\underline{G}$ of Example \ref{categorymonoid} to $\Aut_\ot(\C)$.
\end{Definition}

\begin{Remark}\

\begin{enumerate}[\bf a)]
\item
    Using \cref{mon-funct-defn}, we rewrite the definition of an action of $G$ on $\C$ as a triple:  
    \[
    (\{T_g\}_{g\in G},\{\g_{g,h}\}_{g,h\in G},u),
    \]
    in which
    
    \begin{enumerate}
    \item
        $T_g$ is a semigroupal (resp. monoidal) auto-equivalence of $\C$,
    \item
        $\g_{g,h}$ is a natural semigroupal (resp. monoidal) isomorphism $T_gT_h\iiso T_{gh}$,\label{T_gT_h-cong-T_gh}
    \item
        $u$ is a natural semigroupal (resp. monoidal) isomorphism $\Id_\C\iiso T_e$,\label{T_1-cong-Id_C}
    \end{enumerate}
    such that the following diagram commutes
    \begin{gather}
        \label{square-T-gamma}
        \input{diagramas/rem-acao-parcial-1.tikz}
    \\
        \intertext{and the following morphisms are mutually inverse.}
        \label{(gamma_g1)_X-u_(T_g(X))}
        \input{diagramas/rem-acao-parcial-2.tikz}
    \\
        \label{(gamma_1g)_X-T_g(u_X)}
        \input{diagramas/rem-acao-parcial-3.tikz}
    \end{gather}
\item
    Observe also that 
    \begin{align}
    u\m\g_{g,g\m}:T_gT_{g\m}&\iiso\Id_\C,
    \label{u-inv-gamma_(g_g-inv)}\\ 
    u\m\g_{g\m,g}:T_{g\m}T_g&\iiso\Id_\C.
    \label{u-inv-gamma_(g-inv_g)}
    \end{align}
\item
    The fact that the functors $T_g :\C \rightarrow \C$ are semigroupal (resp. monoidal) implies that they are defined as a pair $(T_g , J^g)$ (resp. a triple $(T_g , J^g , J^{g0})$ in the monoidal case) in which $J^g :T_g (\esp ) \otimes T_g (\esp ) \Rightarrow T_g (\esp \otimes \esp )$ is a natural isomorphism and, for the monoidal case, $J^{g0}:\1 \rightarrow T_g (\1 )$ is an isomorphism in the category $\C$, satisfying the standard hexagon diagram (\ref{hex-J-a}) and, for the monoidal case, the square diagrams (\ref{square-J-l}) and (\ref{square-J-r}).

\item
    The fact that \(\gamma_{g,h}\) and \(u\) are semigroupal (resp. monoidal) natural isomorphisms implies that they satisfy the commutativity of the square diagram (\ref{square-J-eta}) and, in the monoidal case, the triangle diagram (\ref{tri-vf-eta}).
\end{enumerate}
\end{Remark}

The following technical lemma will be used in the sequel.

\begin{Lemma}\label{closure-T_g(I)}
	Let $T$ be an action of a group $G$ on a semigroupal category $\C$ and $\cI$ a tensor  ideal in $\C$. Then
	\begin{enumerate}
	\item
        $\cl{T_e(\cI)}=\cI$;\label{cl-T_1(I)=I}
	\item
        $\cl{T_g(\cl{T_h(\cI)})}=\cl{T_{gh}(\cI)}$.\label{cl-T_g(cl-T_h(I))=cl-T_gh(I)}
	\end{enumerate}
\end{Lemma}
\begin{proof}
	For \cref{cl-T_1(I)=I} observe that $\cI^{(0)} \subseteq (\cl{T_e(\cI)})^{(0)}$, as any $X\in \cI^{(0)}$ is isomorphic to $T_e(X)$ by the isomorphism $u_X :X\rightarrow T_e (X)$. Now, consider any $Y\in (\cl{T_e(\cI)})^{(0)}$, then there is an object $X$ in $\cI$ such that $Y\Iso T_e(X)$, therefore $Y\Iso X$, which makes $Y \in \cI^{(0)}$. Therefore, $\cI^{(0)} = (\cl{T_e(\cI)})^{(0)}$. Now let $X',Y'\in (\cl{T_e (\cI)})^{(0)}$ and $f'\in\Hom_{\cl{T_e(\cI)}}(X',Y')$. By \cref{morphism-in-cl-D} there exist $X,Y\in \cI^{(0)}$, $f\in\Hom_{\cI}(X,Y)$ and isomorphisms $\mu:T_e(X)\to X'$ and $\nu:T_e(Y)\to Y'$ making the upper square of \cref{Mor(cl-T_1(I))=Mor(I)} commute.
	\begin{equation}
        \label{Mor(cl-T_1(I))=Mor(I)}
		\input{diagramas/lem-acao-parcial-1.tikz}
	\end{equation}
	The commutativity of the lower square of \cref{Mor(cl-T_1(I))=Mor(I)} is the naturality of $u$ from \cref{T_1-cong-Id_C} of the definition of a $G$-action \cite{Tambara}. Hence, the perimeter of \cref{Mor(cl-T_1(I))=Mor(I)} commutes, which means by \cref{morphism-in-cl-D} that $f'\in\Hom_{\cl{\cI}}(X' ,Y')=\Hom_{\cI}(X' ,Y')$. Thus $ \cl{T_e(\cI {}^{(1)}})\sst \cI^{(1)}$. The converse inclusion is explained by the commutativity of the lower square of \cref{Mor(cl-T_1(I))=Mor(I)}.
	
	To prove \cref{cl-T_g(cl-T_h(I))=cl-T_gh(I)}, take $X'\in (\cl{T_h(\cI)})^{(0)}$, i.\,e. $X'\Iso T_h(X)$ for some $X\in \cI^{(0)}$. Then $T_g(X')\Iso T_gT_h(X)$, and since $T_gT_h(X)\Iso T_{gh}(X)$, it follows that $T_g(X')\Iso T_{gh}(X)$, so $T_g(X')\in (\cl{T_{gh}(\cI)})^{(0)}$. Thus, $ (T_g(\cl{T_h(\cI)}))^{(0)} \sst (\cl{T_{gh}(\cI)})^{(0)} $ and consequently, $ (\cl{T_g(\cl{T_h(\cI)})})^{(0)} \sst (\cl{T_{gh}(\cI)})^{(0)} $. The converse inclusion is proved similarly. As to morphisms, consider $X',Y'\in (\cl{T_h(\cI)})^{(0)} $ and $f'\in\Hom_{\cl{T_h(\cI)}}(X',Y')$. There are $X,Y\in \cI^{(0)}$, $f\in \Hom_{I} (X,Y)$ and a pair of isomorphisms $\mu:T_h(X)\to X'$, $\nu:T_h(Y)\to Y'$ making the following square commute.
	
	\begin{align}
        \label{morph-in-cl-T_h(I)}
		\input{diagramas/lem-acao-parcial-2.tikz}
	\end{align}
	Applying $T_g$ to \cref{morph-in-cl-T_h(I)} and using the naturality of $\g_{g,h}$ (see \cref{T_gT_h-cong-T_gh} from the definition of a $G$-action \cite{Tambara}), we obtain the following diagram whose squares are commutative.
	
	\begin{equation}
        \label{morph-T_g(cl-T_h(I))is-morph-cl-T_gh(I)}
		\input{diagramas/lem-acao-parcial-3.tikz}
	\end{equation}
	The commutativity of the outside perimeter of \cref{morph-T_g(cl-T_h(I))is-morph-cl-T_gh(I)} implies that $T_g(f')\in (\cl{T_{gh}(\cI)})^{(1)}$, and thus $(T_g(\cl{T_h(\cI)}))^{(1)} \linebreak[0] \sst (\cl{T_{gh}(\cI)})^{(1)}$ yielding $(\cl{T_g(\cl{T_h(\cI)})})^{(1)} \sst (\cl{T_{gh}(\cI)})^{(1)}$. A similar argument proves the converse inclusion.
\end{proof}

\section{Partial actions}

Given an action of $G$ on $\C$ as above and a tensor ideal $\cI$ of $\C$, for any $g\in G$ we set
\begin{align}\label{D_g=I-cap-T_g(I)}
	\Cg=\cI\cap\cl{T_g(\cI)}.
\end{align}
By \cref{closure-F(I)-ideal,intersect-ideal} the subcategory $\Cg$ is a tensor ideal of $\C$ (and hence of $\cI$ as well). 

\begin{Lemma}\label{T_g(intesect)-in-intersect}
	Let $\Cg$ be defined by \cref{D_g=I-cap-T_g(I)}. Then\footnote{Of course, \cref{T_g(D_(g-inv)-cap-D_h)-sst-D_g-cap-D_gh} immediately follows from \cref{cl-T_g(D_(g-inv)-cap-D_h)=D_g-cap-D_gh}, but it will be convenient to us to put it as a separate statement. Moreover, we use \cref{T_g(D_(g-inv)-cap-D_h)-sst-D_g-cap-D_gh} in the proof of \cref{cl-T_g(D_(g-inv)-cap-D_h)=D_g-cap-D_gh}.}
	\begin{enumerate}
		\item $\C_{e}=\cI$;\label{D_1-is-I}
		\item $T_g(\Cgi\cap\Ch)\sst\Cg\cap\Cgh$ for all $g,h\in G$;\label{T_g(D_(g-inv)-cap-D_h)-sst-D_g-cap-D_gh}
		\item $\cl{T_g(\Cgi\cap\Ch)}=\Cg\cap\Cgh$ for all $g,h\in G$.\label{cl-T_g(D_(g-inv)-cap-D_h)=D_g-cap-D_gh}
	\end{enumerate}
\end{Lemma}
\begin{proof}
	Item \cref{D_1-is-I} immediately follows from \cref{cl-T_1(I)=I} of \cref{closure-T_g(I)}.

	Using both items of \cref{closure-T_g(I)}, we get
	\begin{align*}
		T_g(\Cgi\cap\Ch)=T_g(\cI\cap\cl{T_{g\m}(\cI)}\cap\cl{T_h(\cI)})\sst T_g(\cI)\cap \cl{T_e(\cI)}\cap\cl{T_{gh}(\cI)}\sst\Cg\cap\Cgh,
	\end{align*}
	proving \cref{T_g(D_(g-inv)-cap-D_h)-sst-D_g-cap-D_gh}.
	
	Item \cref{T_g(D_(g-inv)-cap-D_h)-sst-D_g-cap-D_gh} yields
	\begin{align}\label{cl-T_g(D_(g-inv)-cap-D_h)-sst-D_g-cap-D_gh}
		\cl{T_g(\Cgi\cap\Ch)}\sst\Cg\cap\Cgh.
	\end{align}
	Writing \cref{cl-T_g(D_(g-inv)-cap-D_h)-sst-D_g-cap-D_gh} for $g'=g\m$ and $h'=gh$, we obtain 
	\begin{align}\label{cl-T_(g-inv)(D_g-cap-D_gh)-sst-D_(g-inv)-cap-D_h}
		\cl{T_{g\m}(\Cg\cap\Cgh)}\sst\Cgi\cap\Ch.
	\end{align}
	Applying $T_g$ to the both sides of \cref{cl-T_(g-inv)(D_g-cap-D_gh)-sst-D_(g-inv)-cap-D_h}, we come to
	\begin{align}\label{T_g(cl-T_(g-inv)(D_g-cap-D_gh)-sst-T_g(D_(g-inv)-cap-D_h))}
		T_g(\cl{T_{g\m}(\Cg\cap\Cgh)})\sst T_g(\Cgi\cap\Ch).
	\end{align}
	Taking the closure of the both sides of \cref{T_g(cl-T_(g-inv)(D_g-cap-D_gh)-sst-T_g(D_(g-inv)-cap-D_h))} and using \cref{closure-T_g(I)} we get the converse inclusion of \cref{cl-T_g(D_(g-inv)-cap-D_h)-sst-D_g-cap-D_gh}, and hence \cref{cl-T_g(D_(g-inv)-cap-D_h)=D_g-cap-D_gh} holds.
\end{proof}

\begin{Corollary}\label{T_g(D_(g-inv))-sst-D_g}
	Under the conditions of \cref{T_g(intesect)-in-intersect} one has $\cl{T_g(\Cgi)}=\Cg$.
\end{Corollary}
\noindent For $\cl{T_g(\Cgi)}=\cl{T_g(\Cgi\cap\C_{e})}=\Cg\cap\C_{g\cdot e}=\Cg$ by \cref{cl-T_g(D_(g-inv)-cap-D_h)=D_g-cap-D_gh} of \cref{T_g(intesect)-in-intersect}.

\begin{Remark}\label{gamma-g-h}
    Denote by $S_g$ the restriction of $T_g$ to $\Cgi$. Under the conditions of \cref{T_g(intesect)-in-intersect}, there is an isomorphism (of the restrictions) $\g_{g,h}:S_gS_h\iiso S_{gh}$ on $\Chi\cap\C_{h\m g\m}$.
\end{Remark}
\noindent Indeed, $S_h$ is defined on $\Chi\cap\C_{h\m g\m}\sst\Chi$ and coincides with $T_h$ on this tensor ideal. Moreover, $S_h(\Chi\cap\C_{h\m g\m})\sst\Ch\cap\Cgi$ by \cref{T_g(D_(g-inv)-cap-D_h)-sst-D_g-cap-D_gh} of \cref{T_g(intesect)-in-intersect}. Hence, $S_gS_h$ is defined on $\Chi\cap\C_{h\m g\m}$ and coincides with $T_gT_h$. On the other hand, $\Chi\cap\C_{h\m g\m}\sst\C_{h\m g\m}$, so $S_{gh}$ is defined on $\Chi\cap\C_{h\m g\m}\sst\C_{h\m g\m}$ and coincides with $T_{gh}$. Since $\g_{g,h}:T_gT_h\iiso T_{gh}$, then, in particular, $\g_{g,h}:S_gS_h\iiso S_{gh}$.

This motivates us to give the following definition.
\begin{Definition}\label{pact-defn}
	A \emph{partial action} of a group $G$ on a semigroupal category $\C$ is a quadruple $T=(
    \{ \C _g \}_{g\in G} ,\{T_g\}_{g\in G}, \linebreak[0] \{\g_{g,h}\}_{g,h\in G}, \linebreak[0] u)$, where 
	\begin{enumerate}
        \item $\{ \C _g \}_{g\in G}$ is a family of tensor ideals in $\C $ indexed by the elements of the group,
		\item\label{T_g-semigr-equiv} $T_g:\Cgi\to\Cg$ is a semigroupal equivalence between tensor ideals $\C _{g^{-1}}$ and $\C _{g}$ of $\C$, that is, for each $g\in G$ there is a natural isomorphism $J^g :T_g (\esp) \otimes T_g (\esp) \Rightarrow T_g (\esp \otimes \esp )$,
		\item $\C_{e}=\C$ and $u:\Id_\C\iiso T_e$ is a natural isomorphism between semigroupal functors,
		\item The restriction of $T_g$ to $\Cgi\cap\Ch$ is a semigroupal functor between $\Cgi\cap\Ch$ and $\Cg\cap\Cgh$,
		\item $\g_{g,h}:T_gT_h\iiso T_{gh}$ is a semigroupal natural isomorphism of functors from $\Chi\cap\C_{h\m g\m}$ to $\C_g \cap \C_{gh}$,
	\end{enumerate}
satisfying the following compatibility conditions:
\begin{enumerate}[(a)]
\item
    Diagram \cref{square-T-gamma} commutes for $X\in\C_{k\m}\cap\C_{k\m h\m}\cap\C_{k\m h\m g\m}$,
\item
    Morphisms \cref{(gamma_g1)_X-u_(T_g(X)),(gamma_1g)_X-T_g(u_X)} are mutually inverse for $X\in\Cgi$,
\item
    For any $X, Y \in \Chi \cap \Cghi$, the following diagram commutes
    \begin{equation}
        \label{pent-gamma-jota}
        \input{diagramas/def-acao-parcial-1.tikz}
    \end{equation}
\item
    For any $X, Y \in \C$, the following diagram commutes
    \begin{equation}
        \label{triangulo_u_jota}
        \input{diagramas/def-acao-parcial-2.tikz}
    \end{equation}
\end{enumerate}
When $\Cg=\C$ for all $g\in G$, we say that $T$ is \emph{global}.
\end{Definition}

\begin{Proposition}
    Given a partial action $T=(
    \{ \C _g \}_{g\in G} ,\{T_g\}_{g\in G}, \linebreak[0] \{\g_{g,h}\}_{g,h\in G}, \linebreak[0] u)$ of a group $G$ on a semigroupal category $\C $, for each $g\in G$ the restriction of $T_g$ to the intersection $\C _{g^{-1}}\cap \C _h$ defines a semigroupal equivalence between $\C _{g^{-1}}\cap \C _h$ and $\C _g \cap \C _{gh}$.
\end{Proposition}

\begin{proof}
    Indeed, using item (iv) twice, we have the restriction functors:
    \[
    T_g|_{\C _{g^{-1}}\cap \C _h}:\C _{g^{-1}}\cap \C _h\rightarrow \C _g \cap \C _{gh} ,
    \]
    and
    \[
    T_{g^{-1}}|_{\C _g\cap \C _{gh}}: \C _g \cap \C _{gh} \rightarrow  \C _{g^{-1}}\cap \C _h.
    \]
    Then, for any $X\in \C _{g^{-1}}\cap \C _h$, we have 
    \[
    X =u^{-1}_X \circ \left( \gamma_{g^{-1},g} \right)_X (T_{g^{-1}} (T_g (X))) .
    \]
    As both $u$ and $\gamma_{g^{-1},g}$ are natural isomorphisms, we have $X\cong T_{g^{-1}} \circ T_g (X)$.
    Similarly, for $Y\in \C _g \cap \C _{gh} $, we have
    \[
    Y=u^{-1}_Y \circ \left( \gamma_{g,g^{-1}}\right)_Y (T_g (T_{g^{-1}}(Y))),
    \]
    leading to $Y\cong T_g \circ T_{g^{-1}} (Y)$. Therefore, these restricted functors implement a categorical equivalence. The fact that this categorical equivalence is semigroupal comes from diagrams (\ref{pent-gamma-jota}) and (\ref{triangulo_u_jota}).
\end{proof}

\begin{Example}\label{partialcentralidempotents}
    One important class of partial $G$-actions $T=\left( \{\C _g \}_{g\in G} , \{T_g \}_{g\in G} , \{\gamma_{g,h}\}_{g,h\in G} , u \right)$ is given when the tensor ideals $\C_g$ are defined by central idempotent objects in $\C $, that is $\C _g =\overline{\1_g \otimes \C}$. In this case, the category $\C$ itself is a monoidal category with the unit object being $\1 =\1_e$. Moreover, for each $g\in G$, the functors $T_g :\C_{g^{-1}} \rightarrow \C_g$ of the given partial action must be monoidal. Then, beyond the natural isomorphism $J^g :T_g (\esp)\otimes T_g (\esp) \Rightarrow T_g(\esp\otimes \esp)$, there should be an isomorphism in the category $\C_g$, $\varphi^{g} :\1_g \rightarrow T_g (\1_{g^{-1}})$, such that, for each $X\in \C_{g^{-1}}$ the following square commutes
    \begin{equation}
        \label{phi_T_g_jota}
        \input{diagramas/exe-acao-parcial.tikz}
    \end{equation}
\end{Example}

\begin{Lemma}\label{lemadasunidades}
    Let $G$ be a group, $\C$ be a monoidal category and $T$ be a partial action of $G$ on $\C$ defined by a family of central idempotent objects as explained in \cref{partialcentralidempotents}. Then
    \begin{enumerate} 
    \item
        For each elements $g,h\in G$ there is an isomorphism $\varphi^g_{gh} :\1_g \otimes \1_{gh} \rightarrow T_g (\1_{g^{-1}}\otimes \1_h)$.
    \item 
        For elements $g,h\in G$, the following diagram commutes:
        \begin{center}
            \input{diagramas/lem-acao-parcial-idemp.tikz}
        \end{center}
    \end{enumerate}
\end{Lemma}

\begin{proof}
    (i) First, take an arbitrary object $X\in \C_{g}\cap \C_{gh}$. As the functor $T_g|_{\C_{g^{-1}}\cap \C_h}$ defines a categorical equivalence between $\C_{g^{-1}}\cap \C_h$ and $\C_{g}\cap \C_{gh}$, then there is an object $Y\in \C_{g^{-1}}\cap \C_h$ such that $X\cong T_g (Y)$. Note that $Y\cong Y\otimes \1_{g^{-1}}$ and $Y\cong Y\otimes \1_h$. Therefore
    \begin{align*}
     X\otimes T_g (\1_{g^{-1}}\otimes \1_h)
        & \cong T_g (Y)\otimes T_g (\1_{g^{-1}}\otimes \1_h)  \\
        & \cong T_g (Y \otimes \1_{g^{-1}}\otimes \1_h )\\
        & \cong T_g (Y) \cong X.
    \end{align*}
    Taking, in particular, $X=\1_g \otimes \1_{gh}$, we have
    \[
    \1_g \otimes \1_{gh} \cong \1_g \otimes \1_{gh}\otimes T_g (\1_{g^{-1}}\otimes \1_h) \cong T_g (\1_{g^{-1}}\otimes \1_h) 
    \] 
    in which the last isomorphism is given by the fact that $\1_g \otimes \1_{gh}$ is the monoidal unit of $\C_g \cap \C_{gh}$.

More generally, $\forall i\in \{ 1, \ldots , n\}$, we have
\[
\1_{g_1} \otimes \1_{g_2} \otimes \cdots \otimes \1_{g_n} \cong T_{g_i} (\1_{g_i^{-1}g_1} \otimes \cdots \1_{g_i^{-1} g_{i-1}} \otimes \1_{g_i^{-1}}\otimes \1_{g_i^{-1} g_{i+1}} \otimes \cdots  \otimes \1_{g_i^{-1} g_n}) .
\]
For each $i\in \{ 1, \ldots , n\}$, denote these isomorphisms by $\varphi^{g_i}_{g_1, \ldots, g_{i-1}, g_{i+1}, \ldots , g_n}$.

    (ii) This comes from the fact both functors, $T_{gh}$ and $T_g T_h$ are monoidal functors between the monoidal categories $\C _{h^{-1}} \cap \C _{(gh)^{-1}}$ and $\C _g \cap \C _{gh}$. The isomorphisms $\varphi^{gh}_g$ and $T_g (\varphi^{h}_{g^{-1}}) \circ \varphi^{g}_{gh}$ are, respectively, the isomorphisms relating the monoidal units for the functors $T_{gh}$ and $T_g T_h$. As the natural transformation $\gamma_{g,h} :T_g T_h \Rightarrow T_{gh}$ is a monoidal natural transformation, the square diagram is nothing but the triangle diagram from the definition of a monoidal natural transformation.
\end{proof}

\begin{Definition}\label{morph-pact-defn}
    Let $T=(\{T_g\}_{g\in G},\{\g_{g,h}\}_{g,h\in G},u)$ and $T'=(\{T'_g\}_{g\in G},\{\g'_{g,h}\}_{g,h\in G},u')$ be partial actions of a group $G$ on semigroupal (resp. monoidal) categories $\C$ and $\C'$, respectively. A \emph{morphism of partial actions} $F: T\to T'$ is a pair $(F,\{\tau_g\}_{g\in G})$, where
    \begin{enumerate}
        \item $F:\C\to\C'$ is a semigroupal functor. Denote by $J^F:F(\esp)\otimes 'F(\esp)\Rightarrow F(\esp\otimes \esp)$ the natural isomorphism related to $F$.
        \item For each $g\in G$, $\tau_g:FT_g\iiso T'_gF$ is natural isomorphisms of semigroupal functors between $\C_{g^{-1}}$ and $\C'_{g}$,
    \end{enumerate}
    such that the following diagrams commute
    \begin{gather}
        \label{pent-tau-gamma}
        \input{diagramas/mor-acao-parcial-1.tikz}
    \\
        \intertext{for all $X\in \C_{h^{-1}} \cap \C_{{gh}^{-1}}$ and}
        \label{tau-Jota}
        \input{diagramas/mor-acao-parcial-2.tikz}
    \end{gather}  
    for all $X,Y\in \C_{g^{-1}}$.
\end{Definition}

\begin{Proposition}
    Let $T=(\{T_g\}_{g\in G},\{\g_{g,h}\}_{g,h\in G},u)$ and $T'=(\{T'_g\}_{g\in G},\{\g'_{g,h}\}_{g,h\in G},u')$ be partial actions of a group $G$ on semigroupal categories $\C$ and $\C'$, respectively and $F:T\rightarrow T'$ be a morphism of partial actions. Then, for all $X\in\C$, we have 
    \begin{equation}
        \label{tri-tau-u}
        \input{diagramas/pro-mor-acao-parcial-1.tikz}
    \end{equation}
    for all $X\in\C$.
\end{Proposition}
\begin{proof}
    Take diagram \cref{pent-tau-gamma} with $g=h=e$ and construct the following diagram:
    \begin{equation}
        \label{varpentagono}
        \input{diagramas/pro-mor-acao-parcial-2.tikz}
    \end{equation}
    The commutativity of the cell (I) in \cref{varpentagono} comes from the naturality of $\tau_e :FT_e \Rightarrow T'_e F$. From the expressions \cref{u-inv-gamma_(g-inv_g),u-inv-gamma_(g_g-inv)}, we have the commutativity of the cells (II) and (IV). Cell (III) commutes automatically, as a particular instance of the diagram  \cref{pent-tau-gamma}, and the square formed by the exterior edges commutes obviously because its vertical edges are identities and its horizontal lines are the same morphism. From these commutativities, one can conclude the commutativity of the cell (V) in the diagram \cref{varpentagono}.
    
    Now, we can construct a new diagram, as follows:
    \begin{equation}
        \label{vartriangulos}
        \input{diagramas/pro-mor-acao-parcial-3.tikz}
    \end{equation}
    The cell (I) in this diagram is exactly the cell (V) in \cref{varpentagono}. Cells (II) and (III) and the external square commute by the naturality  of $u': \Id_{\C'} \Rightarrow T'_e$. Then one can conclude the commutativity of the cell (IV), which is exactly the diagram \cref{tri-tau-u}.  
\end{proof}

\begin{Theorem}\label{restriction}
    Let $T= \left( \{ T_g \}_{g\in G} , \{ \gamma_{g,h} \}_{g,h\in G} , u \right)$ be an action of a group $G$ on a semigroupal category $\C$ and consider $\mathcal{I}$ a tensor ideal in $\C$. Then the data 
    \begin{itemize}
        \item $\C_g =\mathcal{I} \cap \overline{T_g (\mathcal{I})}$,
        \item $S_g =T_g \mid_{\C_{g^{-1}}} \colon \C_{g^{-1}} \rightarrow \C_g$,
        \item $\gamma_{g,h} \colon S_g S_h \Rightarrow S_{gh} \colon \overline{S_{h^{-1}} \left( \C_h \cap \C_g \right)} \rightarrow \overline{ S_{g} \left( \C_h \cap \C_g \right)}$
    \end{itemize}
    define a partial action of $G$ on the semigroupal category $\mathcal{I}$.
\end{Theorem}

\begin{proof}
    The proof of this theorem follows immediately from what was done in \cref{T_g(intesect)-in-intersect,T_g(D_(g-inv))-sst-D_g,gamma-g-h}.
\end{proof}

\begin{Example}
    Consider the monoidal category $\left( \mathcal{O} (X), \cap , X \right)$, in which $X$ is a topological space. Given a global action $\alpha$ of a group $G$ on $X$ by homeomorphisms, this induces naturally an action of $G$ on the monoidal category $\mathcal{O}(X)$. For each $g\in G$ the functor $T_g :\mathcal{O} (X) \rightarrow \mathcal{O}(X)$ sends an open subset $A\subseteq X$, to $T_g (A)=\alpha_g (A)$, for the morphisms, note that  if $A\subseteq B$, then $T_g (A) \subseteq T_g (B)$, therefore, $T_g$ is indeed an endofunctor on the category $\mathcal{O}(X)$.  Let $Y\subseteq X$ be an open subset. The restriction of the action $\alpha$ to $Y$ defines a partial action $\beta$ on the topological space $Y$ (with the induced topology). This partial action corresponds to a partial action on the monoidal category $\left( \mathcal{O}(Y), \cap, Y \right)$.
 \end{Example}

\begin{Example}\label{categorificationofA}
Recall that, for each group $G$, one can define the partial group algebra \cite{DEP}
\[
\Bbbk_{par}G =\left\langle [g] \mid g\in G, \; [e]=1, \; [g^{-1}][g][h]=[g^{-1}][gh], \; [g][h][h^{-1}]=[gh][h^{-1}] \right\rangle ,
\]
having a commutative subalgebra $A_{par} (G)$, henceforth denoted simply by $A$:
\[
A=\left\langle \varepsilon_g =[g][g^{-1}] \mid g\in G \right\rangle .
\]

We can construct a categorical analogue of the algebra $A$. Let $R$ be a commutative unital ring and $\alpha =\left(  R_g \trianglelefteq R, \alpha_g :R_{g^{-1}} \rightarrow R_g \right)_{g\in G}$ be a partial action of $G$ on $R$ such that the ideals $R_g$ are generated by idempotents $1_g \in R$, that is $R_g=1_g R$.
Consider the monoidal subcategory $\mathcal{A}_R (G)\subseteq {}_R\mathcal{M}$, whose objects are finite tensor products of the form
\[
R_{g_1} \otimes_R \cdots \otimes_R R_{g_n} \Iso R_{g_1}\ldots R_{g_n} \trianglelefteq R ,
\]
and whose morphisms are $R$-module maps between them. More specifically, an $R$-module morphism $\phi$ between $R_{g_1} \otimes_R \cdots \otimes_R R_{g_n}$ and  $R_{h_1} \otimes_R \cdots \otimes_R R_{h_m}$ is given by
\[
    \phi (a1_{g_1}\ldots 1_{g_n} )=a1_{h_1}\ldots 1_{h_m} .
\]
The partial action of $G$ on $R$ induces a partial action of $G$ on the monoidal category $\mathcal{A}_R (G)$. First note that the condition $R_g =1_g R$ implies that each ideal $R_g$ is an idempotent object. The categorical ideals $\overline{R_g \otimes_R \mathcal{A}_R (G)}$ are those whose objects are isomorphic to $R_g \otimes_R R_{h_1} \otimes_R \cdots \otimes_R R_{h_n}$ and the functors 
\[
    S_g \colon \overline{R_{g^{-1}} \otimes_R \mathcal{A}_R (G)} \rightarrow \overline{R_g \otimes_R \mathcal{A}_R (G)}
\]
are defined by 
\[
    S_g (R_{g^{-1}} \otimes_R R_{h_1} \otimes_R \cdots \otimes_R R_{h_n}) =R_g \otimes_R R_{gh_1} \otimes_R \cdots \otimes_R R_{gh_n} .
\]
This is compatible with the partial action $\alpha$ because 
\[
    \alpha_g ( R_{g^{-1}}R_{h_1}\ldots R_{h_n} ) =\alpha_g ( R_{g^{-1}} \cap R_{h_1} \cap \cdots \cap R_{h_n} ) =R_g \cap R_{gh_1} \cap \cdots \cap R_{gh_n} =R_g R_{gh_1} \ldots R_{gh_n} .
\]
Each $R$-module morphism $f: R_{g^{-1}} \otimes_R R_{h_1} \otimes_R \cdots \otimes_R R_{h_n} \rightarrow R_{g^{-1}} \otimes_R R_{k_1} \otimes_R \cdots \otimes_R R_{k_m}$ can be viewed as an $R$-module morphism $\widetilde{f}:R_{g^{-1}}R_{h_1}\ldots R_{h_n} \rightarrow R_{g^{-1}}R_{k_1}\ldots R_{k_n}  $. Define the new morphism $S_g (f) :R_g \otimes_R R_{gh_1} \otimes_R \cdots \otimes_R R_{gh_n} \rightarrow R_g \otimes_R R_{gk_1} \otimes_R \cdots \otimes_R R_{gk_m}$ whose associated morphism $\widetilde{S_g (f)}$ is given by $\widetilde{S_g (f)} =\alpha_g \circ f \circ \alpha_{g^{-1}}$.
\end{Example}

\begin{Example}\label{multifusion}
Consider a multifusion category $\C$ over the base field $\Bbbk$ whose simple objects 
$\{ M_1 , \ldots , M_n \}$  satisfy the fusion rule
\begin{equation}
\label{orthogonalsimple}
M_i \otimes M_j =\delta_{i,j} M_i ,
\end{equation}
that is, the simple objects are mutually orthogonal idempotents. In this category, we have $\mathbf{1}_{\C} =\bigoplus_{i=1}^n M_i$. There is a natural action of the symmetric group $S_n$ on 
$\C$: for each $\sigma \in S_n$, we have a functor $T_{\sigma} :\C \rightarrow \C$, sending the simple object $M_i$ to the simple object $M_{\sigma (i)}$ and extending this permutation to direct sums. Regarding morphisms, recall that for simple objects $M_i$ in a fusion category $\C$, we have $\End_{\C}(M_i) =\Bbbk$ and $\Hom_{\C}(M_i ,M_j)=0$, for $i\neq j$. Then define $T_{\sigma} : \End_{\C}(M_i) \rightarrow \End_{\C}(M_{\sigma (i)})$ as the identity on the base field $\Bbbk$. These functors $T_{\sigma}$ are strongly monoidal, because of the expression (\ref{orthogonalsimple}) for the tensor product in the category $\C$. Indeed 
\[
T_\sigma (M_i \otimes M_j) =T_\sigma (\delta_{i,j} M_i) =\delta_{i,j} M_{\sigma (i)} ,
\]
and
\[
T_\sigma (M_i) \otimes T_\sigma (M_j) =M_{\sigma (i)} \otimes M_{\sigma (j)} =\delta_{\sigma (i),\sigma (j)} M_{\sigma (i)}=\delta_{i,j} M_{\sigma (i)}.
\]
Let $X$ be a subset of $\{ 1, \ldots , n \}$, taking the idempotent object $E=\bigoplus_{i\in X} M_i$ and restricting the action of $S_n$ to the tensor ideal $\overline{E\otimes \C}$, we obtain a partial action of $S_n$ on $\C$ as stated in \cref{restriction}.
\end{Example}

\begin{Remark}
Actions of a symmetric group on multifusion categories primarily involve permuting the simple components. In the case of Example \ref{multifusion}, the simple objects in the category are mutually orthogonal, therefore their fusion rule is very trivial. More elaborated examples, with group actions on multifusion theory with nontrivial fusion rules can highlight structural aspects of the category. It could be interesting to explore the induced partial action on  the Grothendieck ring of the category $\C $. 

One interesting physical application for group actions on multifusion categories appears in the construction new topological phases and for describing Symmetry Protected Topological (SPT) phases with non invertible symmetries \cite{Levin_Wen}. More specifically, the simple objects in a fusion category label Topological Defect Lines (TDLs) in a 1+1D quantum field theory. An action maps TDLs to other TDLs, potentially changing their fusion rules. The introduction of partial actions enables us to describe symmetries localized only in a sector of the physical system, which are not globally defined.
\end{Remark}

The following Theorem shows that a partial action $T=( \{ \C _g \}_{g\in G},  \{T_g \}_{g\in G} , \{\gamma_{g,h}\}_{g,h\in G}, u )$ of a group $G$ on a monoidal category $\C $ defines a partial representation of the group $G$ in the category of endofunctors of $\C $. This relation between partial actions and partial representations will be further explored in the end of Section 6, where we discuss the universal property of the partial smash product $\C  \underline{\rtimes} G$, associated to the partial action.

\begin{Theorem}\label{teoremarepparcial}
    Let $T=( \{ \C _g =\overline{\1_g \otimes \C } \}_{g\in G},  \{T_g \}_{g\in G} , \{\gamma_{g,h}\}_{g,h\in G}, u )$ be a partial action of a group $G$ on a monoidal category $\C $ generated by central idempotent objects $\{ \1_g \}_{g\in G}$. Define the family of functors $\pi (g) :\C  \rightarrow \C $ on objects as $\pi (g)(X)=T_g (X\otimes \1_{g^{-1}})$ and on morphisms as $\pi (g)(f)=T_g (f\otimes \1_{g^{-1}})$, for $f:X\rightarrow Y$. Then, we have the following natural isomorphisms:
    \begin{enumerate}
        \item[(I)] $\pi (e) \cong \Id_{\C }$.
        \item[(II)] $\pi (g) \pi (h) \pi (h^{-1}) \cong \pi (gh) \pi (h^{-1})$, for all $g,h \in G$.
        \item[(III)] $\pi (g^{-1}) \pi (g) \pi (h) \cong \pi (g^{-1}) \pi (gh)$, for all $g,h \in G$.
    \end{enumerate}
\end{Theorem}

\begin{proof}
    (I) Take an object $X\in \C ^{(0)}$, then
    \[
    \pi (e)(X)=T_e (X\otimes \1_e) =T_e (X\otimes \1) .
    \] 
    Applying $u^{-1}_X \circ T_e (r_X)$ we obtain the natural isomorphism (I).

    (II) For $g,h \in G$ and $X\in \C ^{(0)}$, we have
    \begin{align*}
    \pi (g) \pi (h) \pi (h^{-1}) (X)
    &=
	   \pi (g) \left( T_h \left(T_{h^{-1}} \left(X\otimes \1_h\right) \otimes \1_{h^{-1}} \right)\right)\\
    &\cong
        \pi (g) \left( T_h \left(T_{h^{-1}} \left(X\otimes \1_h\right) \right)\right) \\
    &\cong 
        \pi (g) (X\otimes \1_h) \\
    &=
	   T_g (X\otimes \1_h \otimes \1_{g^{-1}}) .
    \end{align*}
    On the other hand, 
    \begin{align*}
    \pi (gh) \pi (h^{-1}) (X)
    &=
	   T_{gh} (T_{h^{-1}} (X\otimes \1_h )\otimes \1_{(gh)^{-1}}) \\
    &\cong
        T_{gh} (T_{h^{-1}} (X\otimes \1_h )\otimes  \1_{h^{-1}} \otimes \1_{(gh)^{-1}}) \\
    &\cong
        T_{gh} (T_{h^{-1}} (X\otimes \1_h )\otimes  T_{h^{-1}} (\1_h \otimes \1_{g^{-1}})) \\
    &\cong
        T_{gh} (T_{h^{-1}} (X\otimes \1_h \otimes \1_{g^{-1}})) \\
    &\cong
        T_g (X\otimes \1_h \otimes \1_{g^{-1}}).
    \end{align*}
    The naturality of isomorphisms is easily derived from the definition of functors $\pi (g)$.

    The natural isomorphism (III) is completely analogous.
\end{proof}

\section{Partial actions and Hopf Polyads}

Each partial action of a group $G$ on a semigroupal  category $\C $ defines a family of semigroupal monads, each one acting on the respective tensor ideal of the partial actions.

\begin{Proposition}
    \label{Pgmonad}
    Let $T=( \{ \C _g \}_{g\in G}, \; \{T_g\}_{g\in G}, \; \{\g_{g,h}\}_{g,h\in G},\; u)$ be a partial action of a group $G$ on a semigroupal category $\C$. For each element $g\in G$ one can define a monad $(P_g , \mu_g , \eta_g )$, where $P_g =T_g T_{g^{-1}} :\C_g \rightarrow \C_g$, $\mu_g =T_g \gamma_{e,g^{-1}} \circ T_g \gamma_{g^{-1},g}T_{g^{ -1}} :P_g P_g \Rightarrow P_g$ and $\eta_g =\gamma_{g,g^{-1}}^{-1} \circ u :\Id_{\C _g}\Rightarrow P_g$. Moreover, each monad $P_g$ is semigroupal.
\end{Proposition}
\begin{proof}
    Let us verify that the natural transformations $\mu_g$ and $\eta_g$ satisfy the associativity and the unit axioms, turning the functor $P_g$ into a monad in $\C _g$. 

    For the associativity,
    \[
    \mu_g \circ \left( \mu_g \ast \Id_{P_g} \right) =\mu_g \circ \left(\Id_{P_g} \ast \mu_g \right),
    \]
    consider the diagram
    
    \begin{equation}
    \label{assoc}
    \input{diagramas/pro-monad-1.tikz}
    \end{equation}
    The commutativity of the above diagram is easily verified from iterated uses of the diagram (\ref{square-T-gamma}). 
    
    In order to verify the unit axiom, 
    \[
    \mu_g \circ \left( \eta_g \ast \Id_{P_g} \right) =\Id_{P_g} =  \mu_g \circ \left(\Id_{P_g} \ast \eta_g \right),
    \]
    consider the following diagram.
    \begin{equation}
    \label{leftunit}
    \input{diagramas/pro-monad-2.tikz}
    \end{equation}
    The upper left triangle is an instance of diagram (\ref{(gamma_g1)_X-u_(T_g(X))}), while the lower right triangle is the inverse of (\ref{(gamma_1g)_X-T_g(u_X)}). The squares commute because of (\ref{square-T-gamma}). Therefore 
    \[
    \mu_g \circ \left( \eta_g \ast \Id_{P_g} \right) =\Id_{P_g}.
    \]
    The other equality is completely analogous.
    
    It is easy to see that each $P_g$ is a semigroupal functor, considering $X,Y \in \C _g$ then 
    \[
    T_g \left( J^{g^{-1}}_{X,Y}\right) \circ J^g_{T_{g^{-1}}(X), T_{g^{-1}}(Y)}: P_g (X) \otimes P_g (Y) \rightarrow P_g (X\otimes Y)
    \]
    defines its semigroupal structure. The multiplication $\mu_g =T_g \gamma_{e,g^{-1}} \circ T_g \gamma_{g^{-1},g} T_{g^{-1}}$ depends only of the natural transformation $\gamma$ which is a semigroupal. Also, the  unit $\eta_g =\gamma^{-1}_{g,g^{-1}} \circ u$ depends only on the natural transformations $\gamma$ and $u$, boh semigroupal, then $\eta_g$ is a semigroupal natural transformation.
\end{proof}

\begin{Remark}
    For the case of a partial action $T=\left( \{ \C _g =\overline{\1_g \otimes \C } \}_{g\in G} , \; \{ T_g \}_{g\in G} , \; \{ \gamma_{g,h}\}_{g,h \in G} , \; u \right)$, the associated family of monads $\{ (P_g ,\mu_g , \eta_g )\}_{g\in G}$ becomes a family of monoidal monads.
\end{Remark}

This family of monads defined from the partial action also gives an example of a structure known in the literature as a Hopf polyad.

\begin{Definition}[{\cite{Bruguieres}}]\label{Defpolyad}
    A \emph{polyad} consists of a set of data $(D, \C , T, \mu, \eta)$, in which
    \begin{itemize}
        \item $D$ is a category;
        \item $\C =\left( \C_i \right)_{i\in D^{(0)}}$ is a family of categories indexed by the objects of $D$;
        \item $T=\left( T_a \right)_{a\in D^{(1)}}$ is a family of functors indexed by the morphisms of $D$, with $T_a :\C_i \rightarrow \C_j$, for $a\in \Hom_D (i,j)$.
        \item $\mu =\left( \mu_{a,b} \right)_{(a,b)\in D^{(2)}}$ is a family of natural transformations indexed by the pairs of composable morphisms of $D$, with $\mu_{a,b} :T_a T_b \Rightarrow T_{ab}$;
        \item $\eta =\left( \eta_i \right)_{i\in D^{(0)}}$ is a family of natural transformations indexed by the objects of $D$, with $\eta_i :\Id_{\C_i} \Rightarrow T_i =T_{Id_i}$;
    \end{itemize}
    subject to the following axioms:
    \begin{enumerate}
    \item Given three sequential composable morphisms $a,b,c\in D^{(1)}$,
    \[
    \mu_{ab,c} \circ (\mu_{a,b} \ast \Id_{T_c})=\mu_{a,bc} \circ (\Id_{T_a} \ast \mu_{b,c}) ,
    \]
    \item Given a morphism $a: i\rightarrow j$ in $D$,
    \[
    \mu_{a, Id_i} \circ (\Id_{T_a} \ast \eta_i ) =\Id_{T_a}= \mu_{Id_j, a} \circ (\eta_j \ast \Id_{T_a} ) .
    \]
    \end{enumerate}
The category $D$ is called the source of the polyad $(D, \C  , T, \mu, \eta)$, or, equivalently, one can say that this polyad is a $D$-polyad.
\end{Definition}

\begin{Definition}
A \emph{comonoidal polyad} is a polyad $(D, \C  , T, \mu, \eta)$ endowed with a monoidal structure on each category $\C _i$, for $i\in D^{(0)}$, such that each functor $T_a$, for $a\in D^{(1)}$, is comonoidal and the natural transformations $\mu_{a,b}$, for $(a,b)\in D^{(2)}$ and $\eta_i$, for $i\in D^{(0)}$ are comonoidal natural transformations.
\end{Definition}

\begin{Definition}
A \emph{Hopf polyad} is a comonoidal polyad for which the left and right fusion operators,
\begin{gather*}
    H^l_{a,b}(X,Y) =\left( T_a (X) \otimes \left(\mu_{a,b}\right)_{Y}\right) \circ \left(\xi_a\right)_{X, T_b (Y)} :T_a (X \otimes T_b (Y)) \rightarrow T_a (X) \otimes T_{ab} (Y) ,\\
    H^r_{a,b}(Y,X) =\left( \left( \mu_{a,b}\right)_Y \otimes T_a (X) \right) \circ \left(\xi_a\right)_{T_b (Y), X} :T_a (T_b (Y) \otimes X) \rightarrow T_{ab} (Y) \otimes T_a (X) ,
\end{gather*}
are invertible for all composable pairs of morphisms $(a,b)\in D^{(2)}$, with $k\stackrel{b}{\rightarrow} j \stackrel{a}{\rightarrow} i$, and for all $X$ and $Y$ objects in $\C _j$ and $\C _k$, respectively. Here, $\xi_a: T_a (\esp \otimes \esp )\Rightarrow T_a (\esp ) \otimes T_a (\esp )$ is the comonoidal structure on the functor $T_a$.
\end{Definition}

For the case of a partial group action on a monoidal category, we can state the following Theorem.

\begin{Theorem}\label{Pgpolyad}
    Consider a partial action $T=\left( \{ \C _g =\overline{\1_g \otimes \C } \}_{g\in G}, \; \{ T_g \}_{g\in G} , \; \{ \gamma_{g,h} \}_{g,h \in G} , u \right)$ of a group $G$ on a monoidal category $\C $ with all tensor ideals $\C _g$ generated by central idempotent objects. Then the following data defines a Hopf Polyad:
    \begin{enumerate}
    \item The category $\underline{G}$, whose objects are the elements of the group and whose morphisms are identities.
    \item $C=\{ \C _g =\overline{\1_g \otimes \C }\}_{g\in G}$  is the family of tensor ideals in $\C $ indexed by elements of $G$.
    \item $P=\{ P_{\Id_g}= P_g =T_g \circ T_{g^{-1}} :\C _g \rightarrow \C _g \}_{g\in G}$ is a family of functors indexed by elements of $G$.
    \item $\mu= \{ \mu_{Id_g, Id_g}=\mu_g =T_g \gamma_{e,g^{-1}} \circ T_g \gamma_{g^{-1}, g} T_{g^{-1}}: P_g P_g \Rightarrow P_g\}_{g\in G}$ and $\{  \eta_g =\gamma_{g,g^{-1}}^{-1} \circ u :\Id_{\C _g} \Rightarrow P_g \}_{g\in G}$ are the families of natural transformations indexed by the elements of $G$.
    \end{enumerate}
\end{Theorem}

\begin{proof}
    Recall that the category $\underline{G}$ the  morphisms are,
    \[
    \Hom_{\underline{G}}(g,h) =
    \begin{cases}
        \emptyset, & \text{ if } g\neq h\\
        \{ \Id_g \}, & \text{ if } g=h,
    \end{cases}
    \] 
    then the only possible compositions are of the form $\Id_g \circ \Id_g =\Id_g$. Therefore, the only functors allowed for a polyad defined on $\underline{G}$ are $P_{\Id_g}$, choose them as the previously defined in \cref{Pgmonad}: $P_g =T_g \circ T_{g^{-1}} :\C _g \rightarrow \C _g$. Also, for composable morphisms $a,b\in \underline{G}^{(1)}$, the only multiplications $\mu_{a,b}$ allowed are of the form $\mu_{\Id_g, \Id_g}: P_g \circ P_g \Rightarrow P_g$. We can choose them to be exactly those natural transformations $\mu_g$ of Proposition \ref{Pgmonad}. Similarly, the unit natural transformations $\eta_g$, for $g\in \underline{G}^{(0)}$ can also be chosen to be those defined in \cref{Pgmonad}.

    The identities
    \[
    \mu_{\Id_g , \Id_g} \circ \left( \Id_{P_g} \ast \mu_{\Id_g , \Id_g} \right) =\mu_{\Id_g , \Id_g} \circ \left(  \mu_{\Id_g , \Id_g} \ast \Id_{P_g} \right)
    \]
    and 
    \[
    \mu_{\Id_g , \Id_g} \circ \left( \Id_{P_g} \ast \eta_g \right) =\Id_{P_g} =\mu_{\Id_g , \Id_g} \circ \left(  \eta_g \ast \Id_{P_g} \right)
    \]
    were already verified in (\ref{assoc}) and (\ref{leftunit}). 

    In the case of a partial action with tensor ideals $\C _g$ given by $\C _g =\overline{\1_g \otimes \C }$,  all the monads $P_g =T_g \circ T_{g^{-1}}$ are strongly monoidal, in particular, they are comonoidal with 
\[
\left(\xi_g \right)_{X,Y} =\left(J^g\right)^{{-1}}_{T_{g^{-1}} (X),T_{g^{-1}}(Y)} \circ T_g ( (J^{g^{-1}})^{-1}_{X,Y} ) : P_g (X\otimes Y) \rightarrow P_g (X) \otimes P_g (Y),
\]
for all $X,Y \in \C _g$. Also, it is easy to see that the multiplication morphisms $\mu_g$ and the units $\eta_g$ are comonoidal for all $g\in G$.

Finally, the left and right fusion operators can be written in our context as 
\begin{gather*}
  H^l_g(X,Y) =\left( P_g (X) \otimes \left(\mu_g\right)_{Y}\right) \circ \left(\xi_g \right)_{X, P_g (Y)} :P_g (X \otimes P_g (Y)) \rightarrow P_g (X) \otimes P_g (Y) ,\\
    H^r_g (Y,X) =\left( \left( \mu_g\right)_Y \otimes P_g (X) \right) \circ \left(\xi_g\right)_{P_g (Y), X} :P_g (P_g (Y) \otimes X) \rightarrow P_g (Y) \otimes P_g (X)  ,
\end{gather*}
for all $X,Y \in \C _g$, because both $\mu_g$ and $\xi_g$ are natural isomorphisms. 
\end{proof}

In particular, each monad $P_g$ of our Hopf Polyad is a Hopf monad, in the sense of Bruguières, Lack and Virelizier \cite{BruVir}.

\section{Globalization}

\begin{Definition}\label{glob-defn}
	Let $T=( \{ \C _g \}_{g\in G}, \{T_g\}_{g\in G},\{\g_{g,h}\}_{g,h\in G},u)$ be a partial action of a group $G$ on a semigroupal category $\C$. A \emph{globalization} of $T$ is a global action $T'=(\{T'_g\}_{g\in G},\{\g'_{g,h}\}_{g,h\in G},u')$ of $G$ on a semigroupal category $\C'$ together with a morphism $(F,\{\tau_g\}_{g\in G})$ from $T$ to $T'$, such that
\begin{enumerate}
	\item the subcategory $\cl{F(\C)}$ is a tensor ideal in $\C'$.
    \item $\bigcup_{g\in G}\cl{T_gF(\C)}=\C'$.
	\item for any $U\in\cl{F(\C)}$ if $T'_g(U)\in\cl{F(\C)}$, then $U\in\cl{F(\Cgi)}$,\label{T'(g)-in-cl-F(C)=>U-in-cl-F(D_(g-inv))}
    \item The functor $F$ establishes a semigroupal equivalence between $\C$ and $\cl{F(\C)}$ and an isomorphism between the partial actions 
    \[
    T=( \{ \C _g \}_{g\in G}, \{T_g\}_{g\in G},\{\g_{g,h}\}_{g,h\in G},u)
    \]
    and 
    \[
    \widetilde{T} =( \{ \overline{F(\C ) \cap T'_g (F(\C ))}\}_{g\in G} , 
    \{ \gamma'_{g,h} \}_{g,h\in G} , u').
    \]
\end{enumerate}
\end{Definition}

Item (ii) in the above definition means that a globalizations should be ``minimal'' in certain sense, that is, the category $\C '$ where the original category $\C $ is embedded is basically the orbit of the image of $\C $ by the functor $F$. 

One can deduce from \cref{T'(g)-in-cl-F(C)=>U-in-cl-F(D_(g-inv))} that, for $U\in \overline{F(\C )}$, if $T'_g (U)\in \overline{F(\C )}$, then $T'_g(U)\in\cl{F(\Cg)}$. Indeed, as $U\in \overline{F(\C )}$, there exists an object $V\in \C $ with $U\cong F(V)$, therefore, $T'_g (U) \cong T'_g(F(V))$, using the isomorphism $(\tau_g)_V$, we have 
\[
T'_g (U) \cong T'_g(F(V))=(\tau_g)_V (F(T_g (V)))\cong F(T_g (V)).
\]
Item \cref{T'(g)-in-cl-F(C)=>U-in-cl-F(D_(g-inv))} assures that $U\in \overline{F(\C _{g^{-1}})}$, then, by this line of isomorphisms and because each tensor ideal $\C _g$ is closed by isomorphisms, one can see that $V\in \C _{g^{-1}}$ and then $T'_g (U) \in \overline{F(\C _g)}$.


In what follows, we are going to see that the globalization of a partial action of $G$ on a monoidal category $\C $ exists, at least for the case of partial actions whose tensor ideals $\C _g$ are generated by central idempotent objects in $\C $.

Consider the category \([\underline{G}, \C]\) of functors from the monoidal category \( \underline{G}\) of Example \ref{categorymonoid} to \(\C\). This category can be endowed with a natural structure of a semigroupal category. Given \(F_1, F_2, F'_1, F'_2\) objects in \([G, \C]\) and \(\alpha_1 \colon F_1 \Rightarrow F'_1, \alpha_2 \colon F_2 \Rightarrow F'_2\) morphisms in \([G, \C]\) one can define their tensor products
\[
	\begin{split}
		F_1 \bullet F_2 \colon G & \longto \C\\
		g & \longmapsto F_1(g) \otimes F_2(g)
	\end{split}
	\quad\text{and}\quad
	\begin{split}
		\alpha_1 \bullet \alpha_2 \colon F_1 \bullet F_2 & \Longrightarrow  F'_1 \bullet F'_2\\
		(\alpha_1 \bullet \alpha_2)_g:  F_1 (g) \otimes  F_2 (g) & \rightarrow F'_1 (g) \otimes F'_2 (g)
	\end{split}
\]

Moreover, if \(\C\) is a monoidal category, then one can define a monoidal unit on \([\underline{G}, \C]\) via
\[
	\begin{split}
		\underline{\1} \colon G & \longto \C\\
		g & \longmapsto \underline{\1}(g) = \1
	\end{split}
\]

\begin{Proposition}\label{Globalaction}
    There is a global action of the group $G$ on the semigroupal (resp. monoidal) category $[\underline{G}, \C ]$
\end{Proposition}

\begin{proof} For each element $g\in G$ and for each functor $F\in [\underline{G}, \C ]$, define a new functor $\mathcal{T}_g(F)$, given by $\mathcal{T}_g(F)(h)=F(hg)$. One can easily verify that the natural transformation $\Gamma_{g,h}: \mathcal{T}_g(\mathcal{T}_h (F)) \Rightarrow \mathcal{T}_{gh}(F)$ is the identity. Indeed, for $k\in G$, we have
\[
\mathcal{T}_g(\mathcal{T}_h (F))(k)=\mathcal{T}_h (F)(kg)=F(kgh),
\]
and
\[
\mathcal{T}_{gh}(F)(k)=F(kgh).
\]
Also, for each $g\in G$ and functors $F_1 ,F_2 \in [\underline{G},\C ]$ the natural transformation $\mathcal{J}_{F_1 ,F_2}^g :\mathcal{T}_g (F_1) \bullet \mathcal{T}_g (F_2) \Rightarrow \mathcal{T}_g (F_1 \bullet F_2) $ corresponds to the identity: for $h\in G$, we have
\[
\left(\mathcal{T}_g (F_1) \bullet \mathcal{T}_g (F_2)\right)(h)=\mathcal{T}_g (F_1)(h)\otimes \mathcal{T}_g (F_2)(h)=F_1 (hg)\otimes F_2 (hg), 
\]
and
\[
\mathcal{T}_g (F_1 \bullet F_2)(h)=\left( F_1 \bullet F_2 \right) (hg)=F_1 (hg) \otimes F_2 (hg).
\]
Finally, $\mathcal{T}_e (F) =F$, for any functor $F\in [\underline{G}, \C ]$. Then, the functors $\mathcal{T}_g$ are semigroupal (resp. monoidal) endofunctors, acting on the monoidal category $[\underline{G} ,\C ]$.  Then the data
\[
\left( \{ \mathcal{T}_g \}_{g\in G} , \{ \Gamma_{g,h}\}_{g,h \in G} \right)
\]
define a global action of the group $G$ on the semigroupal (resp. monoidal) category $[\underline{G}, \C ]$.
\end{proof}

Consider now a monoidal category $(\C  , \otimes, \1 )$ with a partial action 
\[
T=\left( \left\{ \C _g =\overline{\1_g \otimes \C } \right\}_{g\in G} , \left\{ T_g  \right\}_{g\in G} , \left\{ \gamma_{g,h} \right\}_{g,h \in G} \right)
\]
in which $\{ \1_g \}_{g\in G}$ are central idempotent objects in $\C $. From this data, one can define a functor $\Phi :\C  \rightarrow [\underline{G},\C ]$ as
\begin{equation}\label{Phionobjects}
\Phi (X)(g)=T_{g} (\1_{g^{-1}} \otimes X) , 
\end{equation}
and, for a morphism $f:X\rightarrow Y$, the natural transformation 
$\Phi (f): \Phi (X) \Rightarrow \Phi (Y)$ localized at $g\in G$ is simply the morphism 
\begin{equation}\label{Phionmorphisms}
\Phi (f)_g = T_{g} (\1_{g^{-1}} \otimes f) 
\end{equation}
in the category $\C $.

The functor $\Phi$ is injective in the objects. Indeed, for $X,Y\in \C $ such that $\Phi (X) =\phi (Y)$ then 
\[
X\cong T_e (\1 \otimes X) =\Phi (X)(e) =\Phi (Y)(e) =T_e (\1 \otimes Y)\cong Y.
\]
With a similar argument, one can show that $\Phi$ is faithful. Moreover, for any objects $X,Y\in \C $ and any $g\in G$
\begin{align*}
\Phi (X\otimes Y)(g)
    & = T_g (\1_{g^{-1}} \otimes X \otimes Y)\\
    & \Iso T_g (\1_{g^{-1}} \otimes X \otimes \1_{g^{-1}} \otimes Y) \\
    & \Iso T_g (\1_{g^{-1}} \otimes X) \otimes T_g (\1_{g^{-1}} \otimes Y) \\
    & = \Phi (X) (g) \otimes \Phi (Y)(g) \\
    & = \left(\Phi (X) \bullet \Phi (Y)\right) (g) .
\end{align*}

Let us construct a new monoidal category generated by objects $\mathcal{T}_g (\Phi (X))$, for $X\in \C $ and $g\in G$ with the product $\bullet$ previously defined.

\begin{Lemma} \label{lemaglobalizacao}\

\begin{enumerate}
\item For each objects $X,Y \in \C $ and for each $g\in G$, we have the natural isomorphisms
\[
    \Phi (X) \bullet \mathcal{T}_g (\Phi (Y)) \Iso \Phi (X\otimes T_g (\1_{g^{-1}} \otimes Y)) 
\]
and
\[
\mathcal{T}_g (\Phi (X)) \bullet \Phi (Y) \Iso \Phi (T_g (\1_{g^{-1}} \otimes X) \otimes Y) .
\]
\item  For each objects $X,Y \in \C $ and elements $g, h\in G$, we have the natural isomorphism 
\[
\mathcal{T}_g (\Phi (X)) \bullet \mathcal{T}_h (\Phi (Y)) \Iso \mathcal{T}_g \Phi (X\otimes T_{g^{-1}h} (\1_{h^{-1}g} \otimes Y)).
\]
\end{enumerate}
\end{Lemma}

\begin{proof}
    (i) For any object $h$ in the category $G$,
    \begin{align*}
    \Phi (X) \bullet \mathcal{T}_g (\Phi (Y)) (h) 
    & = \Phi (X) (h) \otimes \mathcal{T}_g (\Phi (Y)) (h) =\Phi (X) (h) \otimes \Phi (Y) (hg) \\
     & = T_h (\1_{h^{-1}} \otimes X) \otimes T_{hg} (\1_{(hg)^{-1}} \otimes Y) \\
     & \Iso T_h (\1_{h^{-1}} \otimes X) \otimes \1_h \otimes \1_{hg} \otimes  T_{hg} (\1_{(hg)^{-1}} \otimes Y) \\
     \overset{(I)}&{\Iso} T_h (\1_{h^{-1}} \otimes X) \otimes T_{hg} ( \1_{g^{-1}} \otimes \1_{(hg)^{-1}}) \otimes  T_{hg} (\1_{(hg)^{-1}} \otimes Y) \\
     \overset{(II)}&{\Iso} T_h (\1_{h^{-1}} \otimes X) \otimes T_{hg} ( \1_{g^{-1}} \otimes \1_{(hg)^{-1}}\otimes Y) \\
     \overset{(III)}&{\Iso} T_h (\1_{h^{-1}} \otimes X) \otimes T_h( T_g ( \1_{g^{-1}} \otimes \1_{(hg)^{-1}}\otimes Y)) \\
     \overset{(IV)}&{\Iso} T_h (\1_{h^{-1}} \otimes X) \otimes T_h( T_g ( \1_{g^{-1}} \otimes \1_{(hg)^{-1}})\otimes  T_g( \1_{g^{-1}}\otimes Y)) \\
     \overset{(V)}&{\Iso} T_h (\1_{h^{-1}} \otimes X) \otimes T_h(  \1_{h^{-1}} \otimes  T_g( \1_{g^{-1}}\otimes Y)) \\
     \overset{(VI)}&{\Iso} T_h (\1_{h^{-1}} \otimes X\otimes  T_g( \1_{g^{-1}}\otimes Y)) \\
     &= \Phi (X\otimes  T_g( \1_{g^{-1}}\otimes Y))(h) .
    \end{align*}
The first three equalities are, respectively, the definition of the $\bullet$ tensor product in the category $[\underline{G}, \C ]$, the definition of the functor $\mathcal{T}_g(\Phi (Y))$ and the definition of the functor $\Phi :\C \rightarrow [\underline{G}, \C ]$. The first isomorphism (not numbered) comes from the fact that the units $\1_h$ and $\1_{hg}$ are the monoidal units in the respective categories $\C _h$ and $\C _{hg}$. For the sake of simplicity, we are omitting the explicit left and right unit morphisms where their omission does not cause any problem in following the reasoning. The natural isomorphism (I) comes from the fact that the functor $T_{hg}|_{\C _{g^{-1}} \cap\C _{(hg)^{-1}}}$ is monoidal and takes the monoidal unit $\1_{g^{-1}}\otimes \1_{(hg)^{-1}}$ of $\C _{g^{-1}} \cap\C _{(hg)^{-1}}$  into the monoidal unit $\1_h \otimes \1_{hg}$ of $\C _h \cap \C _{hg}$. The natural isomorphism (II) is nothing else but $\Id_{T_h (\1_{h^{-1}} \otimes X)} \otimes J^{hg}_{\1_{g^{-1}} \otimes \1_{(hg)^{-1}} ,\1_{(hg)^{-1}} \otimes Y}$. The natural isomorphism (III) corresponds to $\Id_{T_h (\1_{h^{-1}} \otimes X)} \otimes (\gamma_{h,g})^{-1}_{\1_{g^{-1}} \otimes \1_{(hg)^{-1}} \otimes Y}$. In (IV), we have the natural isomorphism $\Id_{T_h (\1_{h^{-1}} \otimes X)} \otimes T_h ((J^g)^{-1}_{\1_{g^{-1}}\otimes \1_{(hg)^{-1}} , \1_{g^{-1}} \otimes Y})$, while in (V) is again the fact that $T_g|_{\C _{g^{-1}} \cap \C _{(hg)^{-1}}}$ is a monoidal functor, associating the monoidal units of the respective categories, similar to what happened in (I). Finally, but not less important, in (VI), we have the natural isomorphism $J^h_{\1_{h^{-1}} \otimes X ,\1_{h^{-1}} \otimes T_g (\1_{g^{-1}} \otimes Y)}$. The last equality is the definition of the functor $\Phi$.

The second isomorphism of item (i) is obtained similarly.

(ii) Take an object $k$ in the category $G$, then:
\begin{align*}
    \mathcal{T}_g (\Phi (X)) \bullet \mathcal{T}_h (\Phi (Y))(k)
    & =\Phi (X)(kg) \otimes \Phi (Y) (kh) \\
    & = T_{kg}(\1_{(kg)^{-1}} \otimes X) \otimes T_{kh} (\1_{(kh)^{-1}} \otimes Y )\\
    & \Iso T_{kg}(\1_{(kg)^{-1}} \otimes X) \otimes \1_{kg} \otimes \1_{kh} \otimes  T_{kh} (\1_{(kh)^{-1}} \otimes Y )\\
    & \Iso T_{kg}(\1_{(kg)^{-1}} \otimes X) \otimes T_{kh} (\1_{h^{-1}g} \otimes \1_{(kh)^{-1}}) \otimes  T_{kh} (\1_{(kh)^{-1}} \otimes Y )\\
    \overset{(I)}&{\Iso} T_{kg}(\1_{(kg)^{-1}} \otimes X) \otimes T_{kh} (\1_{h^{-1}g} \otimes \1_{(kh)^{-1}} \otimes Y )\\
    & = T_{kg}(\1_{(kg)^{-1}} \otimes X) \otimes T_{kgg^{-1}h} (\1_{h^{-1}g} \otimes \1_{(kh)^{-1}} \otimes Y )\\
    \overset{(II)}&{\Iso} T_{kg}(\1_{(kg)^{-1}} \otimes X) \otimes T_{kg}( T_{g^{-1}h} (\1_{h^{-1}g} \otimes \1_{(kh)^{-1}} \otimes Y ))\\
    & \Iso T_{kg}(\1_{(kg)^{-1}} \otimes X) \otimes T_{kg}( \1_{(kg)^{-1}} \otimes T_{g^{-1}h} (\1_{h^{-1}g} \otimes Y ))\\
    \overset{(III)}&{\Iso} T_{kg}(\1_{(kg)^{-1}} \otimes X \otimes T_{g^{-1}h} (\1_{h^{-1}g} \otimes Y ))\\
    & = \Phi (X \otimes T_{g^{-1}h} (\1_{h^{-1}g}  \otimes Y ))(kg)\\
    & = \mathcal{T}_g (\Phi (X \otimes T_{g^{-1}h} (\1_{h^{-1}g} \otimes Y )))(k).
\end{align*}
The first two and the last two equalities come from the immediate application of the definitions. The unadorned isomorphisms are related to the fact that the functors defining the partial action are monoidal, and then they associate the respective monoidal unities. The natural isomorphism (I) is given by $\Id_{T_{kg}(\1_{(kg)^{-1}} \otimes X)} \otimes J^{kh}_{\1_{h^{-1}g}\otimes \1_{(kh)^{-1}} ,\1_{(kh)^{-1}} \otimes Y}$. The natural isomorphism (II) is nothing but $\Id_{T_{kg}(\1_{(kg)^{-1}} \otimes X)} \otimes (\gamma_{kg,g^{-1}h})^{-1}_{\1_{h^{-1}g} \otimes \1_{(kh)^{-1}} \otimes Y}$. The natural isomorphism (III), in its turn, corresponds to $J^{kg}_{\1_{(kg)^{-1}} \otimes X , \1_{(kg)^{-1}} \otimes T_{g^{-1}h}(\1_{h^{-1}g} \otimes Y)}$.
\end{proof}

\begin{Definition}\label{Cchapeu} Define $\widehat{\C }$ as the subcategory of $[\underline{G}, \C ]$ whose objects are of the form $\mathcal{T}_g (\Phi (X))$, for $g\in G$ and $X\in \C $ and whose morphisms are the natural transformations of the form $\alpha_f :\mathcal{T}_g\left( \Phi (X)\right) \Rightarrow \mathcal{T}_g (\Phi (Y))$, for each morphism $f:X\rightarrow Y$ in $\C $. 
\end{Definition}

Lemma \ref{lemaglobalizacao} shows that the category $\widehat{\C }$ is semigroupal as a subcategory of $[\underline{G} , \C  ]$. Furthermore, since $\Phi (X)=\mathcal{T}_e (\Phi (X))$, the image of the functor $\Phi :\C  \rightarrow [\underline{G}, \C ]$ lies in $\widehat{\C }$ and, from item (i)  of Lemma \ref{lemaglobalizacao}, we have that the isomorphism closure $\overline{\Phi (\C )}$ defines a tensor ideal in $\widehat{\C }$.

\begin{Theorem}
    Consider a partial action 
    \[
    T=\left( \left\{ \C _g =\overline{\1_g \otimes \C } \right\}_{g\in G} , \left\{ T_g :\C _{g^{-1}} \rightarrow \C _g \right\}_{g\in G} \left\{ \gamma_{g,h} \right\}_{g,h \in G} \right)
    \]
    of a group $G$ on $\C $, the functor $\Phi :\C  \rightarrow [\underline{G},\C ]$, defined by (\ref{Phionobjects}) and (\ref{Phionmorphisms}) and the subcategory $\widehat{C}\subseteq [\underline{G},\C ]$ of Definition \ref{Cchapeu}. Then
    \begin{enumerate}[(1)]
        \item The global action $\left( \{ \mathcal{T}_g \}_{g\in G} , \{\Gamma_{g,h} \}_{g,h \in G} \right)$ of $G$ on $[\underline{G}, \C ]$ of Proposition \ref{Globalaction} restricts to the subcategory $\widehat{\C }$ as a global action.
        \item The data 
        \[
    \mathcal{T}=\left( \left\{ \widehat{\C }_g =\overline{\Phi (\C ) \cap \mathcal{T}_g (\Phi (\C )) }\right\}_{g\in G} , \left\{ \mathcal{T}_g|_{\widehat{\C }_{g^{-1}}} :\widehat{\C }_{g^{-1}} \rightarrow \widehat{\C }_g \right\}_{g\in G} \left\{ \Gamma_{g,h} \right\}_{g,h \in G} \right)
    \]
     defines a partial action on $\widehat{\C }$.
     \item For each $g\in G$, there is a natural isomorphism $\tau_g :\Phi \circ T_g \Rightarrow \mathcal{T}_g \circ \Phi :\C _{g^{-1}} \rightarrow \widehat{\C }_g$  commuting diagrams (\ref{pent-tau-gamma}) and (\ref{tau-Jota}).
     \item The data $\left( \{ \mathcal{T}_g \}_{g\in G} , \{\Gamma_{g,h} \}_{g,h \in G} \right)$ and $\left( \Phi , \{ \tau_g \}_{g\in G} \right)$ defines a globalization of the partial action $T$ on $\C $.
    \end{enumerate}
\end{Theorem}

\begin{proof}
    (1) The objects of the category $\widehat{C}$ are of the form $\mathcal{T}_g (\Phi (X))$, for $g\in G$ and $X\in \C $, then for any $g\in G$
    \[
    \mathcal{T}_g (\mathcal{T}_h(\Phi (X))\cong \mathcal{T}_{gh}(\Phi (X)),
    \]
    with the isomorphism given by $\left(\Gamma_{g,h}\right)_{\Phi (X)}$. The action of $\mathcal{T}_g$ alongside the morphisms is also coherent. Then, we have a well defined restriction of functors  $\mathcal{T}_g|_{\widehat{C}} : \widehat{C} \rightarrow \widehat{C}$ for every $g\in G$, restricting the global action $\left( \{ \mathcal{T}_g \}_{g\in G} , \{\Gamma_{g,h} \}_{g,h \in G} \right)$ to the subcategory $\widehat{C}$.

    (2) From Lemma \ref{lemaglobalizacao}, it follows that, for every $g\in G$, the subcategories $\widehat{C}_g =\overline{\Phi (\C ) \cap \mathcal{T}_g (\Phi (\C ))}$ are tensor ideals of $\widehat{C}$. Then the result follows from Theorem \ref{restriction}.

    (3) Consider an object $X\in \C _{g^{-1}}$, then
    \begin{align*}
        \Phi (T_g (X))(k)
            & = T_k (T_g (X) \otimes \1_{k^{-1}}) \\
            & \cong T_k (T_g (X) \otimes \1_g \otimes \1_{k^{-1}}) \\
            & \cong T_k (T_g (X) \otimes T_g( \1_{g^{-1}} \otimes \1_{(kg)^{-1}})) \\
            & \cong T_k (T_g (X \otimes  \1_{g^{-1}} \otimes \1_{(kg)^{-1}})) \\
            & \cong T_{kg} (X \otimes  \1_{g^{-1}} \otimes \1_{(kg)^{-1}}) \\
            & \cong T_{kg} (X \otimes \1_{(kg)^{-1}}) \\
            & = \Phi (X)(kg) \\
            & = \mathcal{T}_g (\Phi (X))(k) .
    \end{align*}
    From this sequence of isomorphisms, one can define, then, the natural isomorphism of functors in $[\underline{G},\C ]$ 
    \[
    \left(\tau_g\right)_X :\Phi (T_g (X))\Rightarrow \mathcal{T}_g (\Phi (X)) :\C _{g^{-1}} \rightarrow \widehat{\C }_g, 
    \]
    which, localized in $k\in G$, is given by
    \begin{equation}\label{tau}
    \left( \left(\tau_g\right)_X \right)_k = \left( \gamma_{k,g}\right)_{X\otimes \1_{g^{-1}} \otimes \1_{(kg)^{-1}}} \circ T_k (J^g_{X,\1_{g^{-1}} \otimes \1_{(kg)^{-1}}}) \circ T_k (T_g (X) \otimes \varphi^g_{k^{-1}} ) , 
    \end{equation}
    where the morphism $\varphi^g_{k^{-1}} :\1_{g} \otimes \1_{k^{-1}} \rightarrow T_g (\1_{g^{-1}}\otimes \1_{(kg)^{-1}})$ is the morphism defined in Lemma \ref{lemadasunidades} (i). It is easy to see that $\tau_g :\Phi \circ T_g \Rightarrow \mathcal{T}_g \circ \Phi$ is natural in $X$. It remains to show that the family of natural transformations $\tau_g$ satisfies the diagram (\ref{pent-tau-gamma}). Consider elements $g,h,k\in G$ and an object $X\in \C _{h^{-1}} \cap \C _{(gh)^{-1}}$, then (\ref{pent-tau-gamma}) writes
    \begin{center}
    \input{diagramas/teo-glob-acao-parcial-1.tikz}
    \end{center}

    Translating carefully the functor $\Phi (X)$ and considering that the natural transformations $\Gamma_{g,h}$ are the identity in $[\underline{G}, \C ]$, we end up with the following square:
    \begin{center}
        \input{diagramas/teo-glob-acao-parcial-2.tikz}
    \end{center}
    Let us keep in mind that $X\in \C _{h^{-1}} \cap \C _{(gh)^{-1}}$, then $X\cong X\otimes \1_{h^{-1}} \otimes \1_{(gh)^{-1}}$, and $T_g T_h(X) \cong T_g T_h (X) \otimes \1_g \otimes \1_{gh}$. Then, we have
    \begin{align*}
    &\hskip-1,2em
        \left( \left( \tau_{gh}\right)_X \right)_k \circ T_k \left( \left( \gamma_{g,h} \right)_X \otimes \1_{k^{-1}} \right)
    \cong
        \left( \left( \tau_{gh}\right)_X \right)_k \circ T_k \left( \left( \gamma_{g,h} \right)_X \otimes \1_g \otimes \1_{gh} \otimes \1_{k^{-1}} \right)\\
    & = \begin{multlined}
        \left( \gamma_{k,gh} \right)_{X\otimes \1_{h^{-1}}\otimes \1_{(gh)^{-1}} \otimes \1_{(kgh)^{-1}}} \circ T_k \left[ J^{gh}_{X,\1_{h^{-1}}\otimes \1_{(gh)^{-1}} \otimes \1_{(kgh)^{-1}}}  \right.  \left. \circ \left( T_{gh}(X) \otimes \varphi^{gh}_{g,k^{-1}} \right)  \circ \left( \left( \gamma_{g,h} \right)_X \otimes \1_g \otimes \1_{gh} \otimes \1_{k^{-1}} \right) \right]
    \end{multlined}\\
    \overset{(I)}&= \left( \gamma_{k,gh} \right)_{X\otimes \1_{h^{-1}}\otimes \1_{(gh)^{-1}} \otimes \1_{(kgh)^{-1}}} \circ T_k \left[ J^{gh}_{X,\1_{h^{-1}}\otimes \1_{(gh)^{-1}} \otimes \1_{(kgh)^{-1}}} \right. \\
    & \qquad \left. \circ \left( \left( \gamma_{g,h} \right)_X \otimes T_{gh}( \1_{h^{-1}} \otimes \1_{(gh)^{-1}} \otimes \1_{(kgh)^{-1}}) \right) \circ \left( T_g T_h (X) \otimes \varphi^{gh}_{g,k^{-1}} \right) \right] \\
    \overset{(II)}&= \left( \gamma_{k,gh} \right)_{X\otimes \1_{h^{-1}}\otimes \1_{(gh)^{-1}} \otimes \1_{(kgh)^{-1}}} \circ T_k \left[ J^{gh}_{X,\1_{h^{-1}}\otimes \1_{(gh)^{-1}} \otimes \1_{(kgh)^{-1}}} \circ \left( \left( \gamma_{g,h} \right)_X \otimes \left( \gamma_{g,h} \right)_{\1_{h^{-1}} \otimes \1_{(gh)^{-1}} \otimes \1_{(kgh)^{-1}}} \right) 
    \right. \\
    & \qquad \left.  \left( T_g T_h (X) \otimes T_g (\varphi^h_{g^{-1}, (kg)^{-1}})\right) \circ \left( T_g T_h (X) \otimes \varphi^g_{gh,k^{-1}} \right) \right] \\
    \overset{(III)}&= \left( \gamma_{k,gh} \right)_{X\otimes \1_{h^{-1}}\otimes \1_{(gh)^{-1}} \otimes \1_{(kgh)^{-1}}} \circ T_k \left( \left( \gamma_{g,h} \right)_{X\otimes \1_{h^{-1}}\otimes \1_{(gh)^{-1}} \otimes \1_{(kgh)^{-1}}} \right)  \circ T_k \left[ T_g \left(  J^h_{X,\1_{h^{-1}}\otimes \1_{(gh)^{-1}} \otimes \1_{(kgh)^{-1}}} \right) \right. \\
    & \qquad \left. \circ J^g_{T_h(X), T_h (\1_{h^{-1}}\otimes \1_{(gh)^{-1}} \otimes \1_{(kgh)^{-1}})} \circ \left( T_g T_h (X) \otimes T_g (\varphi^h_{g^{-1}, (kg)^{-1}})\right) \circ \left( T_g T_h (X) \otimes \varphi^g_{gh,k^{-1}} \right) 
    \right] \\
    \overset{(IV)}&= \left( \gamma_{kg,h} \right)_{X\otimes \1_{h^{-1}}\otimes \1_{(gh)^{-1}} \otimes \1_{(kgh)^{-1}}} \circ  \left( \gamma_{k,g} \right)_{T_h (X\otimes \1_{h^{-1}}\otimes \1_{(gh)^{-1}} \otimes \1_{(kgh)^{-1}})} \circ T_k T_g \left( J^h_{X,\1_{h^{-1}}\otimes \1_{(gh)^{-1}} \otimes \1_{(kgh)^{-1}}}  \right) \\
    &\qquad \circ T_k \left[  J^g_{T_h(X), T_h (\1_{h^{-1}}\otimes \1_{(gh)^{-1}}  \otimes \1_{(kgh)^{-1}})}  \circ  \left( T_g T_h (X) \otimes T_g (\varphi^h_{g^{-1}, (kg)^{-1}})\right) \circ \left( T_g T_h (X) \otimes \varphi^g_{gh,k^{-1}} \right)  \right] \\
    \overset{(V)}&= \left( \gamma_{kg,h} \right)_{X\otimes \1_{h^{-1}}\otimes \1_{(gh)^{-1}} \otimes \1_{(kgh)^{-1}}} \circ
    T_{kg}\left( J^h_{X,\1_{h^{-1}}\otimes \1_{(gh)^{-1}} \otimes \1_{(kgh)^{-1}}} \right) \circ  \left( \gamma_{k,g} \right)_{T_h (X)\otimes T_h(\1_{h^{-1}}\otimes \1_{(gh)^{-1}} \otimes \1_{(kgh)^{-1}})} \\
    & \qquad \circ T_k \left[  T_g \left(   T_h (X) \otimes \varphi^h_{g^{-1}, (kg)^{-1}}\right)  \circ J^g_{T_h(X), \1_{h}\otimes \1_{g^{-1}} \otimes \1_{(kg)^{-1}}} \circ  \left( T_g T_h (X) \otimes \varphi^g_{gh,k^{-1}} \right) \right] \\
    \overset{(VI)}&= \left( \gamma_{kg,h} \right)_{X\otimes \1_{h^{-1}}\otimes \1_{(gh)^{-1}} \otimes \1_{(kgh)^{-1}}} \circ
    T_{kg}\left( J^h_{X,\1_{h^{-1}}\otimes \1_{(gh)^{-1}} \otimes \1_{(kgh)^{-1}}} \right) \circ T_{kg} \left( T_h (X) \otimes \varphi^h_{g^{-1}, (kg)^{-1}} \right) \\
    & \qquad \circ \left( \gamma_{k,g} \right)_{T_h (X)\otimes \1_{h}\otimes \1_{g^{-1}} \otimes \1_{(kg)^{-1}}} \circ T_k \left( J^g_{T_h(X), \1_{h}\otimes \1_{g^{-1}} \otimes \1_{(kg)^{-1}}} \right) \circ T_k \left( T_g T_h (X) \otimes \varphi^g_{gh,k^{-1}} \right) \\
    \overset{(VII)}&= \left( \gamma_{kg,h} \right)_{X\otimes \1_{h^{-1}} \otimes \1_{(kgh)^{-1}}} \circ
    T_{kg}\left( J^h_{X,\1_{h^{-1}} \otimes \1_{(kgh)^{-1}}} \right) \circ T_{kg} \left( T_h (X) \otimes 
    \varphi^h_{(kg)^{-1}} \right) \\
    & \qquad \circ \left( \gamma_{k,g} \right)_{T_h (X)\otimes \1_{g^{-1}} \otimes \1_{(kg)^{-1}}} \circ T_k \left( J^g_{T_h(X), \1_{g^{-1}} \otimes \1_{(kg)^{-1}}} \right) \circ T_k \left( T_g T_h (X) \otimes \varphi^g_{gh,k^{-1}} \right) \\
    & = \left( \left( \tau_h \right)_X \right)_{kg} \circ \left( \left( \tau_g \right)_{T_h(X)} \right)_k .
    \end{align*}
    Equality (I) follows from the naturality of $\gamma_{g,h}$. For equality (II), we used Lemma \ref{lemadasunidades} (ii) and for (III) an instance of diagram (\ref{pent-gamma-jota}). Equality (IV) follows from (\ref{square-T-gamma}), (V) and (VI) are nothing else but the naturality of $\gamma_{k,g}$. Finally (VII) comes from absorption of the monoidal units, recalling that $X\cong X\otimes \1_{(gh)^{-1}}$ and $T_h (X) \cong T_h (X) \otimes \1_h$. 

    In order to verify (\ref{tau-Jota}), consider $g,h\in G$ and $X,Y \in \C _{g^{-1}}$, then the natural transformation $J^{\Phi}_{T_g(X),T_g(Y)}$ localized in $h$,
    \[
    \left( J^{\Phi}_{T_g(X),T_g(Y)} \right)_h: \left( \Phi (T_g (X))\bullet \Phi (T_g(Y))\right)(h) \rightarrow \Phi \left( T_g (X) \otimes T_g (Y)\right) (h) ,
    \]
    is nothing but
    \[
    * \colon T_h \left( T_g (X) \otimes \1_g \otimes\1_{h^{-1}}\right) \otimes T_h \left( T_g (Y) \otimes \1_g \otimes \1_{h^{-1}}\right)
    \rightarrow
    T_h \left( T_g (X)\otimes T_g (Y) \otimes \1_g \otimes \1_{h^{-1}} \right),
    \]
    where \(*\) denotes \(J^{h}_{T_g(X) \otimes \1_g \otimes \1_{h^{-1}},T_g(Y)\otimes \1_g \otimes \1_{h^{-1}} }\).
    Similarly, 
    \[
    \left( \Phi \left( J^g_{X,Y}\right)\right)_h : \Phi \left(T_g (X) \otimes T_g (Y)\right)(h) \rightarrow \Phi \left( T_g (X\otimes Y)\right) 
    \]
    is written explicitly as
    \[
    T_h \left( J^g_{X ,Y} \otimes \1_g \otimes  \1_{h^{-1}} \right): T_h \left( T_g (X)\otimes T_g (Y) \otimes \1_g \otimes \1_{h^{-1}} \right) \rightarrow T_h \left( T_g (X\otimes Y) \otimes \1_g \otimes \1_{h^{-1}} \right).
    \]
    Finally,
    \[
    \left(\mathcal{T}_g \left( J^{\Phi}_{X,Y}\right)\right)_h :\mathcal{T}_g \left( \Phi (X) \bullet \Phi (Y)\right)(h) \rightarrow \mathcal{T}_g \left( \Phi (X\otimes Y) \right)(h)
    \]
    reads
    \[
    * \colon
    T_{hg}\left( X \otimes \1_{g^{-1}} \otimes \1_{(hg)^{-1}} \right) \otimes T_{hg}\left( Y \otimes \1_{g^{-1}} \otimes \1_{(hg)^{-1}} \right)
    \rightarrow T_{hg} \left( X\otimes Y \otimes \1_{g^{-1}} \otimes \1_{(hg)^{-1}} \right),
    \]
   where \(*\) denotes \(J^{hg}_{X\otimes \1_{g^{-1}} \otimes \1_{(hg)^{-1}} ,Y\otimes  \1_{g^{-1}} \otimes \1_{(hg)^{-1}}}\).

    Then, on the one hand, we have
    \begin{align*}
    &\hskip-1,2em
        \left( \mathcal{T}_g \left( J^{\Phi}_{X,Y} \right) \right)_h \circ \left( \mathcal{J}^g_{\Phi (X) ,\Phi (Y)}\right)_h \circ \left( \left( \tau_g \right)_X \bullet \left( \tau_g \right)_Y \right)_h =\\
    &=
        J^{hg}_{X\otimes \1_{g^{-1}}\otimes \1_{(hg)^{-1}} ,Y\otimes \1_{g^{-1}}\otimes \1_{(hg)^{-1}}} \circ \, \Id \, \circ \left( \left( \left( \tau_g \right)_X\right)_h \otimes \left( \left( \tau_g \right)_Y\right)_h\right) \\
    \overset{(\ref{tau})}&=
        J^{hg}_{X\otimes \1_{(hg)^{-1}} ,Y\otimes \1_{(hg)^{-1}}} \circ \left( \left( \gamma_{h,g} \right)_{X\otimes \1_{g^{-1}} \otimes \1_{(hg)^{-1}}} \otimes \left( \gamma_{h,g} \right)_{Y\otimes \1_{g^{-1}} \otimes \1_{(hg)^{-1}}} \right) \\
    &\quad
        \circ \left( T_h \left( J^g_{X,\1_{g^{-1}}\otimes \1_{(hg)^{-1}}}\right) \otimes T_h \left( J^g_{Y,\1_{g^{-1}}\otimes \1_{(hg)^{-1}}}\right)\right) \circ \left( T_h\left( T_g (X) \otimes \varphi^g_{h^{-1}}\right) \otimes T_h\left( T_g (Y) \otimes \varphi^g_{h^{-1}}\right)\right) \\
    \overset{(\ref{pent-gamma-jota})}&=
        \left( \gamma_{h,g} \right)_{X\otimes Y \otimes\1_{g^{-1}} \otimes \1_{(hg)^{-1}}} \circ T_h \left( J^g_{X \otimes\1_{g^{-1}} \otimes \1_{(hg)^{-1}}, Y \otimes\1_{g^{-1}} \otimes \1_{(hg)^{-1}} }\right) \\
    &\quad 
        \circ J^h_{T_g (X \otimes\1_{g^{-1}} \otimes \1_{(hg)^{-1}}), T_g (Y \otimes\1_{g^{-1}} \otimes \1_{(hg)^{-1}})} \circ \left( T_h \left( J^g_{X,\1_{g^{-1}}\otimes \1_{(hg)^{-1}}}\right) \otimes T_h \left( J^g_{Y,\1_{g^{-1}}\otimes \1_{(hg)^{-1}}}\right)\right) \\
    &\quad
        \circ \left( T_h\left( T_g (X) \otimes \varphi^g_{h^{-1}}\right) \otimes T_h\left( T_g (Y) \otimes \varphi^g_{h^{-1}}\right)\right) \\
    \overset{(I)}&=
        \left( \gamma_{h,g} \right)_{X\otimes Y \otimes\1_{g^{-1}} \otimes \1_{(hg)^{-1}}} \circ T_h \left( J^g_{X \otimes\1_{g^{-1}} \otimes \1_{(hg)^{-1}}, Y \otimes\1_{g^{-1}} \otimes \1_{(hg)^{-1}} }\right) \\
    &\quad
        \circ  T_h \left( J^g_{X, \1_{g^{-1}} \otimes \1_{(hg)^{-1}}} \otimes J^g_{Y, \1_{g^{-1}} \otimes \1_{(hg)^{-1}}}\right) \circ J^h_{T_g(X) \otimes T_g (\1_{g^{-1}} \otimes \1_{(hg)^{-1}}) ,T_g(Y) \otimes T_g (\1_{g^{-1}} \otimes \1_{(hg)^{-1}}) } \\
     &\quad
        \circ \left( T_h\left( T_g (X) \otimes \varphi^g_{h^{-1}}\right) \otimes T_h\left( T_g (Y) \otimes \varphi^g_{h^{-1}}\right)\right) \\
     \overset{(II)}&=
        \left( \gamma_{h,g} \right)_{X\otimes Y \otimes\1_{g^{-1}} \otimes \1_{(hg)^{-1}}} \circ T_h \left( J^g_{X \otimes\1_{g^{-1}} \otimes \1_{(hg)^{-1}}, Y \otimes\1_{g^{-1}} \otimes \1_{(hg)^{-1}} }\right) \\
     &\quad
        \circ T_h \left( J^g_{X, \1_{g^{-1}} \otimes \1_{(hg)^{-1}}} \otimes J^g_{Y, \1_{g^{-1}} \otimes \1_{(hg)^{-1}}}\right)
     \circ T_h \left( T_g (X) \otimes T_g (Y) \otimes \varphi^g_{h^{-1}}\right) \\
     &\quad
        \circ J^{h}_{T_g(X) \otimes \1_g \otimes \1_{h^{-1}},T_g(Y)\otimes \1_g \otimes \1_{h^{-1}} },
    \end{align*}
    where the equalities (I) and (II) follow from the naturality of $J^h$. On the other hand, we have
    \begin{align*}
    &\hskip-1,2em
        \left( \left( \tau_g\right)_{X\otimes Y}\right)_h \circ \left( \Phi \left( T_g (X\otimes Y)\right)\right)_h \circ \left( J^\Phi_{T_g (X),T_g(Y)}\right)_h =\\
    \overset{(\ref{tau})}&=
	   \left( \gamma_{h,g} \right)_{X\otimes Y \otimes \1_{g^{-1}} \otimes \1_{(hg)^{-1}}} \circ T_h \left( J^g_{X\otimes Y , \1_{g^{-1}} \otimes \1_{(hg)^{-1}}} \right) \circ T_h \left( T_g (X\otimes Y) \otimes \varphi^g_{h^{-1}} \right) \\
    &\quad
        \circ T_h \left( J^g_{X ,Y} \otimes \1_g \otimes  \1_{h^{-1}} \right) \circ J^{h}_{T_g(X) \otimes \1_g \otimes \1_{h^{-1}},T_g(Y)\otimes \1_g \otimes \1_{h^{-1}} } \\
    \overset{(III)}&=
	   \left( \gamma_{h,g} \right)_{X\otimes Y \otimes \1_{g^{-1}} \otimes \1_{(hg)^{-1}}} \circ T_h \left( J^g_{X\otimes Y , \1_{g^{-1}} \otimes \1_{(hg)^{-1}}} \right) \circ T_h \left( J^g_{X,Y} \otimes T_g (\1_{g^{-1}} \otimes \1_{(hg)^{-1}}) \right) \\
    &\quad
        \circ T_h \left( T_g (X) \otimes T_g (Y) \otimes \varphi^g_{h^{-1}}\right) \circ J^{h}_{T_g(X) \otimes \1_g \otimes \1_{h^{-1}},T_g(Y)\otimes \1_g \otimes \1_{h^{-1}} },
    \end{align*}
    where the equality (III) follows from the naturality of $J^g$. The composition of the middle maps are equal, as an isomorphism from $T_h(T_g (X)\otimes T_g (Y) \otimes T_g (\1_{g^{-1}} \otimes \1_{(hg)^{-1}})$ and $T_h(T_g(X\otimes Y \otimes \1_{g^{-1}} \otimes \1_{(hg)^{-1}}))$.

    (4) Now, let us verify that the global action $\left(  \{ \mathcal{T}_g \}_{g\in G} , \{ \Gamma_{g,h} \}_{g,h\in G} \right)$ of $G$ on $\widehat{\C }$ and the semigroupal functor $\left( \Phi , \{ \tau_g \}_{g\in G} \right)$ defines a globalization for the partial action of $G$ on $\C $. Item (i) of the definition follows from Lemma \ref{lemaglobalizacao}. One can see that item (ii) is easily satisfied by the very definition of the category $\widehat{\C }$. 
    For item (iii), let $X\in \C _g$, then, for each $h\in G$
    \begin{align*}
        \Phi (X)(h)
        & = T_h (X\otimes \1_{h^{-1}}) \\
        & \Iso T_h (X\otimes \1_g \otimes \1_{h^{-1}}) \\
        & \Iso T_h (T_g (T_{g^{-1}}(X)) \otimes \1_g \otimes \1_{h^{-1}}) \\
        & \Iso T_h (T_g (T_{g^{-1}}(X) \otimes \1_{g^{-1}}\otimes \1_{(hg)^{-1}})) \\
        & \Iso T_{hg} (T_{g^{-1}} (X) \otimes \1_{(hg)^{-1}})\\
        & = \Phi (T_{g^{-1}}(X) )(hg) \\ 
        & =\mathcal{T}_g (\Phi (T_{g^{-1}}(X)))(h).
    \end{align*}
    Therefore, $\Phi (X) \in \overline{\Phi (\C )\cap \mathcal{T}_g (\Phi (\C ))}$.

    On the other hand, let $\Phi (X) \in \overline{\Phi (\C ) \cap \mathcal{T}_g (\Phi (\C ))}$, then there is an object $Y\in \C _{g^{-1}}$ such that $\Phi (X) \Iso \mathcal{T}_g (\Phi (Y))$. For every element $h\in G$, we have
    \[
    T_h (X\otimes \1_{h^{-1}}) \Iso T_{hg} (Y \otimes \1_{(hg)^{-1}}) ,
    \]
    then
    \begin{align*}
    T_{hg} (Y \otimes \1_{(hg)^{-1}})
        & \Iso  T_h (X\otimes \1_{h^{-1}}) \otimes \1_h \otimes \1_{hg} \\
        & \Iso T_h (X \otimes \1_g \otimes \1_{h^{-1}}) \\
        & \Iso T_h (T_g (T_{g^{-1}}(X\otimes \1_g)) \otimes \1_g \otimes \1_{h^{-1}}) \\
        & \Iso T_h (T_g (T_{g^{-1}}(X\otimes \1_g)) \otimes T_g( \1_{g^{-1}} \otimes \1_{(hg)^{-1}})) \\
        & \Iso T_h (T_g (T_{g^{-1}}(X\otimes \1_g) \otimes  \1_{g^{-1}} \otimes \1_{(hg)^{-1}}) )\\
        & \Iso T_{hg} (T_{g^{-1}}(X\otimes \1_g) \otimes \1_{(hg)^{-1}}) .
    \end{align*}
    As $T_{hg}:\C _{(hg)^{-1}}\rightarrow \C _{hg}$ is a monoidal categorical equivalence, we conclude that 
    \[
    Y\otimes \1_{(hg)^{-1}} \cong T_{g^{-1}}(X\otimes \1_g) \otimes 1_{(hg)^{-1}},
    \]
    and then,  using the left unit isomorphism in $\C _{(hg)^{-1}}$, we obtain the isomorphism 
    $Y\Iso T_{g^{-1}}(X\otimes \1_g)$, that is,
    \[
    \Phi (X) \Iso \mathcal{T}_g (\Phi (T_{g^{-1}}(X\otimes \1_g))) .
    \]
    Applying this natural isomorphism at $e\in G$, we have
    \[
    X\Iso T_g (T_{g^{-1}}(X\otimes \1_g)) ,
    \]
    implying that $X\Iso X\otimes \1_g$. Therefore $X\in \C _g$.

    Finally, from (3) we proved that the functor $\Phi$ is a morphism between the partial action $T$ and the global action $\left( \{ \mathcal{T}_g \}_{g\in G} , \{\Gamma_{g,h}\}_{g,h\in G} \right)$. Also, $\Phi$ is a semigroupal equivalence between $\C $ and $\overline{\Phi (\C )}$.  And for each $g\in G$, $\tau_g$ is a natural isomorphism between $\Phi T_g $ and $\mathcal{T}_g \Phi$. 
    
    Therefore, the functor $\Phi :\C  \rightarrow \widehat{\C }$ defines an isomorphism between partial actions  
    \[
    T= \left( \{ \C _g =\overline{\1_g \otimes \C } \}_{g\in G} , \{ T_g :\C _{g^{-1}} \rightarrow \C _g \}_{g\in G} , \{ \gamma_{g,h} \}_{g,h\in G}  \right)
    \]
    and 
    \[
    \mathcal{T} =
    \left(
        \{ \widehat{\C }_g =\widehat{\C } \cap \mathcal{T}_g (\widehat{\C }) \}_{g\in G},
        \{ \mathcal{T}_g|_{\widehat{\C }_{g^{-1}}} :\widehat{\C }_{g^{-1}} \rightarrow \widehat{\C }_g \}_{g\in G} ,
        \{ \Gamma_{g,h} \}_{g,h\in G}
    \right).
    \qedhere
    \]
\end{proof}

\begin{Definition}
Given a partial action $T=\left( \{ \C _g =\overline{\1_g \otimes \C } \}_{g\in G} , \{ T_g :\C _{g^{-1}} \rightarrow \C _g \}_{g\in G} , \{ \gamma_{g,h} \}_{g,h\in G} , u \right)$ of a group $G$ on a monoidal category $\C $ defined by a family of central idempotent objects $\{ \1_g \}_{g\in G}$, we call the category $\widehat{\C } \subseteq [\underline{G}, \C ]$ with the global action $\mathcal{T} :G \rightarrow \Aut_{\otimes} (\widehat{\C })$ previously defined, the \emph{standard globalization of} $T$.
\end{Definition}

\begin{Remark}
    In \cite{DE}, the authors construct the standard globalization for a partial action of a group $G$ on a unital algebra $A$ in which all the ideals $A_g \trianglelefteq A$ are defined by central idempotents $\1_g$. Also, it is proved that this globalization, is admissible, that is, it satisfies a minimality condition $B=\bigcup_{g\in G} \beta_g (\varphi (A))$. Then this globalization is unique up to equivalence. In Definition \ref{glob-defn}, item (ii) also establishes an admissibility condition for a globalization. In what follows, we show that our standard globalization is also unique up to a monoidal categorical equivalence.
\end{Remark}

\begin{Proposition}
    Let $T= \left( \{ \C _g =\overline{\1_g \otimes \C } \}_{g\in G} , \{ T_g :\C _{g^{-1}} \rightarrow \C _g \}_{g\in G} , \{ \gamma_{g,h} \}_{g,h\in G}  \right)$ be a partial action of a group $G$ on a monoidal category $\C $ defined by a family of idempotent objects $\{ \1_g \}_{g\in G}$. If the data \[
    \left( \widetilde{\C } , \; \left( \{ \widetilde{T}_g \}_{g\in G} , \{ \widetilde{\gamma}_{g,h} \}_{g,h\in G} \right), \; \left( \widetilde{\Phi} , \{\widetilde{\tau}_g \}_{g\in G} \right)\right)
    \]
    is a globalization for the partial action $T$, then there is a monoidal equivalence $\Psi :\widehat{\C } \rightarrow \widetilde{\C }$ such that:
    \begin{enumerate}
        \item $\Psi \circ \Phi =\widetilde{\Phi}$;
        \item $\Psi$ is an isomorphism between the global actions $\left(\{ \mathcal{T}_g \}_{g\in G} , \{ \Gamma_{g,h} \}_{g,h \in G} \right)$ on $\widehat{\C }$ and\\
        $\left( \{ \widetilde{T}_g \}_{g\in G} , \{ \widetilde{\gamma}_{g,h} \}_{g,h\in G} \right)$ on $\widetilde{\C }$.
    \end{enumerate}
\end{Proposition}

\begin{proof}
    Define a functor $\Psi :\widehat{\C }\rightarrow \widetilde{\C }$ on objects by
    \[
    \Psi \left( \mathcal{T}_g \left( \Phi (X) \right) \right) =\widetilde{T}_g \left( \widetilde{\Phi}(X) \right) ,
    \]
    and, for each $f:X\rightarrow Y$ in $\C $, denote by $\alpha_f :\mathcal{T}_g (X) \rightarrow \mathcal{T}_g (Y)$ the respective natural transformation. Define
    \[
    \Psi \left( \alpha_f \right) =\widetilde{T}_g (\widetilde{\Phi}(f)).
    \]
    By construction, we have that $\Psi \circ \Phi =\widetilde{\Phi}$. As $\widetilde{\C } =\bigcup_{g\in G} \overline{\widetilde{T}_g (\widetilde{\Phi}(\C )}$, for each object $M\in \widetilde{\C }$, there exist an object $X\in \C $ and an element $g\in G$ such that $M \cong \widetilde{T}_g (\widetilde{\Phi}(X))$. Then the functor $\Psi$ is essentially surjective. Moreover, the functor $\Phi$ defines a categorical equivalence between $\C $ and and $\overline{\Phi (\C )}$, whereas $\widetilde{\Phi}$ defines a categorical equivalence between $\C $ and and $\overline{\widetilde{\Phi} (\C )}$. Also, for each $g\in G$, $\mathcal{T}_g$ and $\widetilde{T}_g$ are categorical equivalences in their respective categories. Therefore, we have a sequence of bijections
    \begin{align*}
        \Hom_{\widetilde{\C }} \left( \Psi \left( \mathcal{T}_g (\Phi (X))\right) , \; \Psi \left( \mathcal{T}_g (\Phi (Y))\right)\right)
    &=
	   \Hom_{\widetilde{\C }} \left(  \widetilde{T}_g (\widetilde{\Phi} (X)) , \;  \widetilde{T}_g (\widetilde{\Phi} (Y))\right) \\
    &\cong
        \Hom_{\widetilde{\C }} \left(   \widetilde{\Phi} (X) , \;  \widetilde{\Phi} (Y)\right) \\
    &\cong
        \Hom_{{\C }} \left(   X,Y\right) \\
    &\cong
        \Hom_{\widehat{\C }} \left(   {\Phi} (X) , \;  {\Phi} (Y)\right) \\
    &\cong 
        \Hom_{\widehat{\C }} \left(   \mathcal{T}_g \left({\Phi} (X)\right) , \;  \mathcal{T}_g \left( {\Phi} (Y)\right)\right),
    \end{align*}
    proving that the functor $\Psi$ is fully faithful. 

    Take now $X, Y\in \C $ and $g,h \in G$, let us evaluate 
    \[
    \Psi \left( \mathcal{T}_g (\Phi(X)) \bullet \mathcal{T}_h (\Phi (Y)) \right) .
    \]
    From this expression, $\Phi (X)\in \overline{\Phi (\C ) \cap \mathcal{T}_{g^{-1}} (\Phi((\C ))} =\widehat{\C }_{g^{-1}}$, then there is an object $X' \in \C _{g^{-1}}$ such that $\Phi (X) \cong \Phi (X')$. Similarly, $\Phi (Y)\in \overline{\Phi (\C ) \cap \mathcal{T}_{h^{-1}} (\Phi((\C ))} =\widehat{\C }_{h^{-1}}$, then there is an object $Y' \in \C _{h^{-1}}$ such that $\Phi (Y) \cong \Phi (Y')$. Then
    \begin{align*}
        \Psi \left( \mathcal{T}_g (\Phi(X)) \bullet \mathcal{T}_h (\Phi (Y)) \right)
    & \cong
        \Psi \left( \mathcal{T}_g (\Phi(X')) \bullet \mathcal{T}_h (\Phi (Y')) \right) \\
    & \cong 
        \Psi \left( \mathcal{T}_g \left( \Phi \left( X' \otimes T_{g^{-1}h} (\1_{h^{-1}g} \otimes Y') \right)\right)\right) \\
    & \cong
        \widetilde{T}_g \left( \widetilde{\Phi} \left( X' \otimes T_{g^{-1}h} (\1_{h^{-1}g} \otimes Y')\right)\right)\\
    & \cong 
        \widetilde{\Phi} \left( T_g ( X' \otimes T_{g^{-1}h} (\1_{h^{-1}g} \otimes Y'))\right). 
    \end{align*}
    On the other hand,
    \begin{align*}
        \Psi \left( \mathcal{T}_g (\Phi (X))\right) \otimes \Psi \left( \mathcal{T}_h (\Phi (Y))\right)
    & \cong
        \Psi \left( \mathcal{T}_g (\Phi (X'))\right) \otimes \Psi \left( \mathcal{T}_h (\Phi (Y'))\right) \\
    & \cong
        \widetilde{T}_g \left( \widetilde{\Phi} (X')\right) \otimes \widetilde{T}_h \left( \widetilde{\Phi} (Y')\right) \\
    & \cong
        \widetilde{\Phi} \left(T_g (X') \right) \otimes \widetilde{\Phi} \left(T_h (Y') \right) \\
    & \cong
        \widetilde{\Phi} \left(T_g (X') \otimes T_h (Y') \right) \\
    & \cong
        \widetilde{\Phi} \left(T_g (X') \otimes \1_g \otimes \1_h \otimes T_h (Y') \right) \\
    & \cong
        \widetilde{\Phi} \left(T_g (X') \otimes T_h(\1_{h^{-1}g} \otimes \1_{h^{-1}}) \otimes T_h (Y') \right) \\
    & \cong
        \widetilde{\Phi} \left(T_g (X') \otimes T_h(\1_{h^{-1}g} \otimes \1_{h^{-1}} \otimes  Y') \right) \\
    & \cong
        \widetilde{\Phi} \left(T_g (X') \otimes T_g \circ T_{g^{-1}h}(\1_{h^{-1}g} \otimes \1_{h^{-1}} \otimes  Y') \right) \\
    & \cong
        \widetilde{\Phi} \left(T_g (X' \otimes T_{g^{-1}h}(\1_{h^{-1}g} \otimes \1_{h^{-1}} \otimes  Y') \right) \\
    & \cong
        \widetilde{\Phi} \left(T_g (X' \otimes \1_{g^{-1}} \otimes  T_{g^{-1}h}(\1_{h^{-1}g}  \otimes  Y') \right) \\
    & \cong
        \widetilde{\Phi} \left(T_g (X'  \otimes  T_{g^{-1}h}(\1_{h^{-1}g}  \otimes  Y') \right).
    \end{align*}
    Therefore, the functor $\Psi$ is a monoidal equivalence. Moreover, for $g,h \in G$ and $X\in \C $, we have
    \[
    \Psi (\mathcal{T}_g (\mathcal{T}_h (\Phi (x))))\cong \Psi (\mathcal{T}_{gh}  (\Phi (x))) =\widetilde{T}_{gh} (\widetilde{\Phi}(X)) \cong \widetilde{T}_g ( \widetilde{T}_h (\widetilde{\Phi}(X))) =\widetilde{T}_g \Psi (\mathcal{T}_h (\Phi (X))). 
    \]
    Then the functor $\Psi$ defines an isomorphism between the action $\left( 
    \{ \mathcal{T}_g \}_{g\in G} , \; \{ \Gamma_{g,h} \}_{g,h \in G} \right)$ on $\widehat{\C }$ and $\left( 
    \{ \widetilde{T}_g \}_{g\in G} , \; \{ \widetilde{\gamma}_{g,h} \}_{g,h \in G} \right)$ on $\widetilde{\C }$.
\end{proof}

\section{Partial \emph{G}-equivariantization and partial smash product}

Recall the definition of \(G\)-invariants by an action from \cite{Tambara}:

\begin{Definition}
    Given \((\{T_g\}_{g\in G},\{\g_{g,h}\}_{g,h\in G},u)\) a (global) action of \(G\) on a monoidal category \(\C\), the \emph{\(G\)-invariants} in \(\C\) (or \emph{$G$-equivariantization}), denoted by \(\C^G\), is the category given as follows:
    \begin{itemize}
        \item the objects are pairs \((X, \theta)\), where \(X \in  (\C)^{(0)}\) and \(\theta\) is a family of isomorphisms \(\theta_g \colon T_g(X) \longto X\) in \(\C\), such that the following diagrams commute: 
        \begin{gather}
            \label{quadradoequivariant}
            \input{diagramas/def-equivar-global-1.tikz}
        \\
            \intertext{for every $g,h\in G$, and}
            \label{trianguloequivariant}
            \input{diagramas/def-equivar-global-2.tikz}
        \end{gather}
        \item a morphism \( \varphi: (X,\theta) \longto (X',\theta')\) in \(\C^G\) is a morphism \(\varphi \colon X \longto X'\) in \(\C\) such that, for all \(g \in G\), the following diagram commutes:
        \begin{equation}
            \label{morphismglobalequivariant}
            \input{diagramas/def-equivar-global-3.tikz}
        \end{equation}
    \end{itemize}
    The category \(\C^G\) is a monoidal category with product given by
    \[
        (X, \theta)\widehat{\otimes}(X', \theta') = ( X \otimes X', \widetilde{\theta}),
    \]
    where
    \[
        \widetilde{\theta}_g = ( \theta_g \otimes \theta'_g ) \circ (J^g_{X,X'})^{-1}
    \]
    and the unity object is \((\1, J^0)\), and the associativity and unity isomorphism are inherited from \(\C\).
\end{Definition}

\subsection{Partial \emph{G}-equivariantization}
\phantom{|}\medskip

\def\cG{{\ensuremath{\C^{G}}}}

For a unital partial action of a group $G$ on a unital algebra $A$, one can define the subalgebra of partial invariants
\[
A^{\underline{G}} =\{ a\in A \; | \; \alpha_g (a \1_{g^{-1}})=a\1_g \} .
\]
Motivated by the notion of the subalgebra of partial invariants of a partial action, one can define the notion of a partial $G$-equivariantization for a partial action of a group $G$ on a semigroupal category $\C $. 

\begin{Definition}
    Let $T=( \{ \C _g \}_{g\in G} ,    \{T_g\}_{g\in G},\{\g_{g,h}\}_{g,h\in G},u)$ be a partial action of a group $G$ on a semigroupal category $\C$.
    \begin{enumerate}
    \item
        A \emph{partially $G$-equivariant object} in $\C$ is a pair $ ( X, \{\sigma_g^X\}_{g \in G} ) $, where 
        \[\sigma_g^X \colon T_g(X \otimes \esp) \implies X \otimes T_g(\esp)\colon \C_{g^{-1}} \longrightarrow \C_g\]
        is a family of natural isomorphisms such that the following diagrams commute:
        \begin{gather}
            \label{pent-sigma}
            \input{diagramas/def-equivar-parcial-1.tikz}
        \\
            \intertext{for all \(Y \in \Chi \cap \Cghi\),}
            \label{triangle-sigma}
            \input{diagramas/def-equivar-parcial-2.tikz}
        \\
            \intertext{for all $Y\in \C $, and}
            \label{Jota-sigma}
            \input{diagramas/def-equivar-parcial-3.tikz}
        \end{gather}
        for all $Y,Z \in \C _{g^{-1}}$.
    \item
        A \emph{morphism} $ ( X, \sigma^X ) \longto ( Y, \sigma^Y ) $ between two partially $G$-equivariant objects in $\C$ is a morphism $\varphi\colon X \longto Y$ in $\C$ such that the following diagram commutes
        \begin{equation}
            \label{eq:G-equiv-mor}
            \input{diagramas/def-equivar-parcial-4.tikz}
        \end{equation}
    \end{enumerate}
    Then, we have the category $\C ^{\underline{G}}$ of partial \(G\)-equivariantizations. 
\end{Definition}

For the case of a partial action generated by central idempotents, we have a simpler way to obtain partially equivariant objects.

\begin{Lemma}
    Let $T= \left( \{ \C _g =\overline{\1_g \otimes \C } \}_{g\in G} , \{ T_g :\C _{g^{-1}} \rightarrow \C _g \}_{g\in G} , \{ \gamma_{g,h} \}_{g,h\in G} , u \right)$ be a partial action of a group $G$ on a monoidal category $\C $ defined by a family of central idempotent objects $\{ \1_g \}_{g\in G}$. Then it is equivalent to define a partially $G$-equivariant object $ ( X, \{\sigma_g^X\}_{g\in G} ) $ by a family of morphisms  
    \[
    \left\{ \widetilde{\sigma}_{g, g_1 ,\ldots , g_n}^X \colon T_g \left( X \otimes \1_{g^{-1}} \otimes \1_{g_1} \otimes \cdots \otimes \1_{g_n}\right) \rightarrow X \otimes \1_g \otimes \1_{gg_1} \otimes \cdots \otimes \1_{gg_n} \vert n\in \mathbb{N} , g,g_1, \ldots g_n \in G \right\} ,
    \]
    with $\widetilde{\sigma}^X_e =u_X^{-1}:T_e (X) \rightarrow X$, 
    and such that, for all $g,h\in G$, the following diagram commutes:
    \begin{equation}
        \label{tildesigma}
        \input{diagramas/lem-equivar-idemp-1.tikz}
    \end{equation}
\end{Lemma}
\begin{proof}
    Given a partially equivariant object $ ( X, \{\sigma_g^X\}_{g\in G} ) $, define 
    \[
        \widetilde{\sigma}^X_{g, g_1, \ldots , g_n} =\left( X\otimes (\varphi^g_{g g_1, \ldots , g g_n})^{-1} \right) \circ \left( \sigma^X_{g} \right)_{\1_{g^{-1}}\otimes \1_{g_1} \otimes \cdots \otimes \1_{g_n}} .
    \]
    
    One can verify the commutativity of diagram \eqref{tildesigma}, as follows:
    
    For $g,h\in G$, we have
    \begin{center}
        \small
        \input{diagramas/lem-equivar-idemp-2.tikz}
    \end{center}
    
    Here, the commutativity of (I) comes from the commutativity of the diagram (\ref{pent-sigma}), the square (II) is commutative from Lemma \ref{lemadasunidades} (ii) and (III) is nothing else but the naturality of the transformation $\sigma^X_g$.
    
    On the other hand, consider a family of morphisms 
    \[
       \left\{ \widetilde{\sigma}_{g, g_1 ,\ldots , g_n}^X \colon T_g \left( X \otimes \1_{g^{-1}} \otimes \1_{g_1} \otimes \cdots \otimes \1_{g_n}\right) \rightarrow X \otimes \1_g \otimes \1_{gg_1} \otimes \cdots \otimes \1_{gg_n} \vert n\in \mathbb{N} , g,g_1, \ldots g_n \in G \right\}
       \]
    such that, for each $g,h \in G$, the diagram (\ref{tildesigma}) commutes. Then, for $g\in G$ and $Y\in \C _g^{-1}$,  define
    $\left( \sigma^X_g \right)_{Y} :T_g (X\otimes Y )\rightarrow X\otimes T_g (Y) $ as the composition:
    \[
    \left( \sigma^X_g \right)_{Y} = \left(\widetilde{\sigma}^X_g  \otimes T_g (Y) \right) \circ \left( J^g_{X\otimes \1_{g^{-1}}, Y }\right)^{-1} .
    \]
    Here, for simplicity, we omitted the unit isomorphisms of the monoidal categories $\C _{g^{-1}}$ and $\C _g$, respectively. It is easy to see that $\sigma^X_g$ is natural with respect to $Y$.
    
    In order to verify that the family of natural transformations $\left\{ \sigma^X_g \right\}_{g\in G}$ satisfies the commutativity of the diagram (\ref{pent-sigma}), consider $g,h\in G$ and $Y\in \C _{h^{-1}}\cap \C _{(gh)^{-1}}$, then
    
    \begin{align*}
    &\hskip-1,2em
        \left( \sigma^X_{gh} \right)_Y \circ \left( \gamma_{g,h} \right)_{X\otimes Y} =\\
    & = 
        \left( \widetilde{\sigma}^X_{gh, h^{-1}} \otimes T_{gh}(Y) \right) \circ \left( J^{gh}_{X\otimes \1_{h^{-1}}\otimes \1_{(gh)^{-1}}} \right)^{-1} \circ \left( \gamma_{g,h} \right)_{X\otimes Y} \\
    \overset{(I)}&=
        \left( \widetilde{\sigma}^X_{gh, h^{-1}} \otimes T_{gh}(Y) \right) \circ \left(  \left( \gamma_{g,h }\right)_{X\otimes \1_{h^{-1}} \otimes \1_{(gh)^{-1}}}  \otimes \left( \gamma_{g,h }\right)_{Y} \right) \\
    & \quad
        \circ \left( J^g_{ T_h (X\otimes \1_{h^{-1}} \otimes \1_{(gh)^{-1}}) ,T_h(Y)} \right)^{-1}  \circ T_g \left(\left( J^h_{X\otimes \1_{h^{-1}} \otimes \1_{(gh)^{-1}}, Y} \right)^{-1} \right) \\
    \overset{(II)}&=
        \left( X\otimes \1_g \otimes \1_{gh} \otimes \left( \gamma_{g,h} \right)_Y \right)  \circ \left(  \left( \widetilde{\sigma}^X_{gh, h^{-1}} \circ \left( \gamma_{g,h} \right)_{X\otimes \1_{h^{-1}} \otimes \1_{(gh)^{-1}}} \right) \otimes T_g T_h (Y) \right) \\
    & \quad
        \circ \left( J^g_{ T_h (X\otimes \1_{h^{-1}} \otimes \1_{(gh)^{-1}}) ,T_h(Y)} \right)^{-1}  \circ T_g \left(\left( J^h_{X\otimes \1_{h^{-1}} \otimes \1_{(gh)^{-1}}, Y} \right)^{-1} \right) \\
    \overset{(III)}&=
        \left( X\otimes \left( \gamma_{g,h} \right)_Y \right)  \circ  \left(  \left( \widetilde{\sigma}^X_{g,h} \circ T_g ( \widetilde{\sigma}^X_{h,(gh)^{-1}}) \right) \otimes T_g T_h (Y) \right) \\
    & \quad
        \circ \left( J^g_{ T_h (X\otimes \1_{h^{-1}} \otimes \1_{(gh)^{-1}}) ,T_h(Y)} \right)^{-1}  \circ T_g \left(\left( J^h_{X\otimes \1_{h^{-1}} \otimes \1_{(gh)^{-1}}, Y} \right)^{-1} \right) \\
    \overset{(IV)}&=
        \left( X\otimes \left( \gamma_{g,h} \right)_Y \right)  \circ \left( \widetilde{\sigma}^X_{g,h} \otimes T_g T_h (Y) \right) \circ \left( J^g_{ X\otimes \1_h \otimes \1_{g^{-1}} ,T_h(Y)} \right)^{-1} \\
    & \quad
        \circ T_g \left( \left( 
        \widetilde{\sigma}^X_{h,(gh)^{-1}} \otimes T_h (Y) \right) \circ \left( J^h_{X\otimes \1_{h^{-1}} \otimes \1_{(gh)^{-1}}, Y} \right)^{-1} \right) \\
    & =
        \left( X\otimes \left( \gamma_{g,h} \right)_Y \right)  \circ \left(\sigma^X_{g} \right)_{T_h(Y)} \circ \left( \sigma^X_h \right)_Y .
    \end{align*}
    Here, to obtain equality (I), we used the commutativity of the diagram (\ref{pent-gamma-jota}), for equality (II), we used the naturality of $\gamma_{g,h}$ and for (III) the commutativity of the diagram of the Lemma \ref{lemadasunidades} (ii). Finally, equality (IV) follows from the naturality of $J^g$.
    
    By definition, $\widetilde{\sigma}^X_e =u^{-1}_X$, then the commutativity of the triangle diagram (\ref{triangle-sigma}) and the square diagram (\ref{Jota-sigma}) follow trivially.
\end{proof}

In the next result we see that, when starting from a global action of a group \(G\) on a monoidal category, a partial \(G\)-equivariant object generates a global one, and vice-versa.

\begin{Theorem}
    \label{thm:g<->p-equivariant-object}
    Let \((      \{T_g\}_{g\in G},             \{\g_{g,h}\}_{g,h\in G},u)\) be a (global) action of \(G\) on a monoidal category \(\C\). Then there exists a categorical isomorphism between the category of partial $G$-equivariantizations $\mathcal{G}^{\underline{G}}$ and the category of $G$-equivariantizations $\C ^{G}$.
    
\end{Theorem}
\begin{proof}
    Given a partially $G$-equivariant object \((X, \{\sigma^X_g\}_{g \in G})\), we define a (globally) $G$-equivariant object $(X, \theta^X)$ in the following way:
    
    \[
    \theta [\sigma ]_g^X \colon T_g(X) 
        \xrightarrow{~T_g(\Iso)~}
    T_g(X \otimes \1)
        \xrightarrow{~(\sigma_g^X)_{\1}~}
    X \otimes T_g(\1)
        \xrightarrow{~X \otimes (J^{g0})^{-1}}
    X \otimes \1
        \xrightarrow{~\Iso~}
    X.
    \]

    To verify that the family $\{ \theta [\sigma ]^X_g \}_{g\in G}$ satisfies the commutativity of the diagram (\ref{quadradoequivariant}), take elements $g,h\in G$, then
    \begin{center}
        \input{diagramas/teo-iso-CG-1.tikz}
    \end{center}
    The commutativity of (I) comes from the diagram (\ref{pent-sigma}), the square (II) is nothing else but the naturality of $\sigma^X_g$, and (III) follows from the fact that $\gamma_{g,h}$ is a monoidal natural transformation.
    
    Now, let us check the commutativity of the triangle diagram (\ref{trianguloequivariant}):
    \begin{center}
        \input{diagramas/teo-iso-CG-2.tikz}
    \end{center}
    Here, the commutativity of (I) follows from the naturality of $u$, (II) is exactly the triangle diagram (\ref{triangle-sigma}) and (III) comes from the fact that $u$ is a monoidal natural transformation.

    For a morphism $\varphi :(X, \sigma^X) \rightarrow (Y, \sigma^Y)$ in $\C ^{\underline{G}}$, let us verify that $\varphi$ is also a morphism in $\C ^G$. For this, we need to verify the commutativity of diagram (\ref{morphismglobalequivariant}). Indeed

    \begin{equation}
        \input{diagramas/teo-iso-CG-3.tikz}
    \end{equation} 
    The commutativity of the first and the last squares comes from the naturality of $r: \underline{\;} \otimes \1 \Rightarrow \Id_{\C }$ and of its inverse. The second square commutes because (\ref{eq:G-equiv-mor}) and the third square commutes because the morphisms applied to both entries of the tensor product are independent.

    Therefore, we can define a functor
    \begin{center}
        \input{diagramas/teo-iso-CG-4.tikz}
    \end{center}

    Reciprocally, given a (global) $G$-equivariant object \((X, \{\theta^X_g\}_{g \in G})\), we can also define a partial \(G\)-equivariant object \((X, \{\sigma [\theta ]^X_g\}_{g \in G})\) out of this $G$-equivariant object. Since \(T_g\) is a monoidal functor, so there exists the natural isomorphism \(J^g \colon T_g(\esp) \otimes T_g(\esp) \implies T_g ( \esp \otimes \esp )\), then one can define a natural transformation \(\sigma [\theta ]^X_g \colon T_g(X \otimes \esp) \implies X \otimes T_g(\esp)\) in the following way
    \[
        \sigma [\theta ]^X_g \colon T_g ( X \otimes \esp )   \xRightarrow{\ \left( J^g_{X,\esp} \right)^{-1}\ }
        T_g(X) \otimes T_g(\esp) \xRightarrow{\ \theta^X_g \otimes T_g(\esp) \ }
        X \otimes T_g(\esp) .
    \]
    To verify that this family of natural transformations satisfies the commutativity of the diagram (\ref{pent-sigma}), consider $g,h\in G$, then
    \begin{align*}
        \left( \sigma [\theta ]^X_{gh} \right)_Y \circ \left( \gamma_{g,h} \right)_{X\otimes Y}
    &= 
        \left( \theta^X_{gh} \otimes T_{gh} (Y) \right) \circ \left( J^{gh}_{X,Y}\right)^{-1} \circ \left( \gamma_{g,h} \right)_{X\otimes Y} \\
    \overset{(I)}&=
        \left( \theta^X_{gh} \otimes T_{gh} (Y) \right) \circ \left( \left( \gamma_{g,h} \right)_{X} \otimes \left( \gamma_{g,h} \right)_{Y}\right) \circ \left( J^{g}_{T_h(X),T_h(Y)}\right)^{-1} \circ T_g \left( \left( J^{h}_{X,Y}\right)^{-1} \right) \\
    \overset{(II)}&=
        \left( X \otimes \left( \gamma_{g,h} \right)_{Y} \right) \circ \left( \theta^X_{gh} \otimes T_g T_h (Y) \right) \circ \left( \left( \gamma_{g,h} \right)_{X} \otimes T_g T_h (Y) \right) \\
    & \quad
        \circ \left( J^{g}_{T_h(X),T_h(Y)}\right)^{-1} \circ T_g \left( \left( J^{h}_{X,Y}\right)^{-1} \right) \\
        \overset{(III)}&= \left( X \otimes \left( \gamma_{g,h} \right)_{Y} \right) \circ \left( \theta^X_{g} \otimes T_g T_h (Y) \right) \circ \left( T_g (\theta^X_{h}) \otimes T_g T_h (Y) \right) \\
    & \quad
        \circ \left( J^{g}_{T_h(X),T_h(Y)}\right)^{-1} \circ T_g \left( \left( J^{h}_{X,Y}\right)^{-1} \right) \\
    \overset{(IV)}&=
        \left( X \otimes \left( \gamma_{g,h} \right)_{Y} \right) \circ \left( \theta^X_{g} \otimes T_g T_h (Y) \right) \circ \left( J^{g}_{X,T_h(Y)}\right)^{-1} \\
    & \quad
        \circ T_g \left( \left( \theta^X_{h} \otimes T_h (Y) \right) \circ \left( J^{h}_{X,Y}\right)^{-1} \right) \\
    & =
        \left( X \otimes \left( \gamma_{g,h} \right)_{Y} \right) \circ \left( \sigma^X_{g} \right)_{T_h(Y)} \circ T_g \left( \left( \sigma [\theta ]^X_{h} \right)_{Y} \right) .
    \end{align*}
    Equality (I) follows from (\ref{pent-gamma-jota}), (II) is the naturality of $\gamma_{g,h}$, (III) comes from (\ref{quadradoequivariant}) and (IV) from the naturality of $J^g$.

    For diagram (\ref{triangle-sigma}), we have
    \begin{align*}
    \left( \sigma [\theta ]^X_{e} \right)_{Y} \circ u_{X\otimes Y}
        & = \left( \theta^X_{e} \otimes T_{e} (Y) \right) \circ \left( J^{e}_{X,Y}\right)^{-1} \circ u_{X\otimes Y} \\
        & = \left( \theta^X_{e} \otimes T_{e} (Y) \right) \circ \left( u_X \otimes u_Y \right) \\
        & = X\otimes u_Y .
    \end{align*}

    Finally, for (\ref{Jota-sigma})
    \begin{align*}
        \left( \sigma[\theta]^X_g \right)_{Y\otimes Z} \circ J^g_{X\otimes Y, Z}
    & =
        \left( \theta^X_g \otimes T_g (Y\otimes Z) \right) \circ \left( J^g_{X,Y\otimes Z}\right)^{-1} \circ J^g_{X\otimes Y, Z} \\
    & =
        \left( \theta^X_g \otimes T_g (Y\otimes Z) \right) \circ \left( T_g (X) \otimes J^g_{Y,Z}\right) \circ \left( \left( J^g_{X,Y}\right)^{-1} \otimes T_g (Z)\right) \\
    & =
        \left( X \otimes J^g_{Y,Z}\right) \circ \left( \theta^X_g \otimes T_g (Y) \otimes T_g (Z)\right)    \circ \left( \left( J^g_{X,Y}\right)^{-1} \otimes T_g (Z)\right) \\
    & =
        \left( X \otimes J^g_{Y,Z}\right) \circ \left( \left( \sigma [\theta]^X_g\right)_Y \otimes T_g (Z) \right) .
    \end{align*}

    For a morphism $\varphi :(X , \theta^X) \rightarrow (Y, \theta^Y)$ in $\C ^G$, let us verify that the same $\varphi$ is a morphism in $\C ^{\underline{G}}$. For that, one needs to verify the commutativity of (\ref{eq:G-equiv-mor}). Indeed,

    \begin{equation}
        \input{diagramas/teo-iso-CG-5.tikz}
    \end{equation} 
    The first square commutes by the naturality of (the inverse of) $J^g$. The commutativity of the second square is a consequence of the commutativity of (\ref{morphismglobalequivariant}).

    Then, we can define a second functor
    \begin{center}
        \input{diagramas/teo-iso-CG-6.tikz}
    \end{center}

    In order to verify that $\Psi \circ \Phi =\text{Id}_{\C ^{\underline{G}}}$ and $\Phi \circ \Psi =\text{Id}_{\C ^{{G}}}$, we just need to check that, for any object $(X, \theta^X )\in \C ^G$, we have $\theta[\sigma [\theta]]^X =\theta^X$ and, for any object $(X, \sigma^X)\in \C ^{\underline{G}}$, we have 
    $\sigma [\theta [\sigma]]^X =\sigma^X$.
    
    Consider the following diagram for $(X, \theta^X )\in \C ^G$ and $g\in G$:
    \begin{center}
        \input{diagramas/teo-iso-CG-7.tikz}
    \end{center}
    Here, the commutativity of the triangle (II) comes from (\ref{square-J-l}). The square (III) commutes because of the independence of order of morphisms and (IV) commutes by the naturality of the right unit isomorphism.
    
    Now, take an object $(X, \sigma^X)\in \C ^{\underline{G}}$, and element $g\in G$ and an object $Y\in \C $. Then
    \begin{center}
        \input{diagramas/teo-iso-CG-8.tikz}
    \end{center}
    The commutativity of the square (I) comes from the naturality of $J^g$. Square (II) is an instance of (\ref{Jota-sigma}), while square (III) is an instance of (\ref{square-J-l}). Commutativity of triangle (IV) is the triangle axiom for a monoidal category and, finally, square (V) commutes by the naturality of $\sigma^X_g$.

    Therefore, for a global action of a group $G$ on a monoidal category $\C $, we have a categorical isomorphism between $\C ^G$ and $\C ^{\underline{G}}$.
\end{proof}

\begin{Example}
    Consider a unital partial action of a finite group $G$ on a commutative ring $R$ and its extension to the partial action of $G$ on the sub-monoidal category $\mathcal{A}_R (G) \subseteq {}_R \mathcal{M}$, generated by tensor products of the ideals $R_g$ as introduced in Example \ref{categorificationofA}. If $G=\{ g_1 ,g_2, \ldots , g_n \}$, then the object $X=R_{g_1} \otimes_R R_{g_2} \otimes_R \cdots \otimes_R R_{g_n}$, together with the morphisms 
    $\sigma^{X}_g :S_g (X \otimes_R M)\rightarrow X\otimes S_g (M)$, which are the identity morphism in $X\otimes_R S_g (M)$, for $M\in \overline{R_g \otimes_R \mathcal{A}_R (G)}$, constitutes an equivariant object in $\mathcal{A}_R (G)^{\underline{G}}$
\end{Example}

\begin{Theorem}
    Let $T=\left(  \{ \C _g \}_{g\in G} , \; \{ T_g \}_{g\in G} , \; \{ \gamma_{g,h}\}_{g,h\in G} , u \right)$ be a partial action of a group $G$ on a semigroupal category $\C $. Then the category $\C ^{\underline{G}}$ of partial equivariant objects is semigroupal with the forgetful functor $U :\C ^{\underline{G}} \rightarrow \C $ being strictly semigroupal. Moreover, if the partial action $T$ is defined by a family of central idempotent objects $\{ \mathbbm{1}_g \} _{g\in G}$, then the category $\C ^{\underline{G}}$ is monoidal and the forgetful functor is strictly monoidal.
\end{Theorem}

\begin{proof}
    Let \( ( X , \sigma^X ) \) and \( ( Y , \sigma^Y ) \) be two partial \(G\)-equivariant objects. Define the product of these objects in the following way
    \[
        ( X , \sigma^X ) \widehat{\otimes} ( Y , \sigma^Y )= \left( X \otimes Y , \sigma^{X \otimes Y} \right),
    \]
    where \( \{ \sigma_g^{X \otimes Y} \}_{g \in G} \) is the family of natural isomorphisms given by
    \[
    T_g ( X \otimes Y \otimes \esp ) 
        \xRightarrow{~\sigma_g^X~}
    X \otimes T_g ( Y \otimes \esp ) 
        \xRightarrow{~X \otimes \sigma_g^Y~}
    X \otimes Y \otimes T_g ( \esp ).
    \]

    similarly, for two morphisms \(\alpha \colon (X, \sigma^X) \longto (X' , \sigma^{X'})\) and \(\beta \colon (Y, \sigma^Y) \longto (Y' , \sigma^{Y'}) \) in $\C ^{\underline{G}}$, define the product of these morphisms as the tensor product \( \alpha \otimes \beta \) in \( \C \), i.e.
    \[
        \alpha \otimes \beta \colon X \otimes Y \longto X' \otimes Y'.
    \]

    We need to verify that \(\alpha \otimes \beta \) is a morphism in $\C ^{\underline{G}}$. So we need to show that the diagram in \eqref{eq:G-equiv-mor} commutes for \(\alpha \otimes \beta \), so we have the following diagram
    \begin{center}
        \input{diagramas/teo-CG-mon-1.tikz}
    \end{center}
    which becomes
    \begin{center}
        \input{diagramas/teo-CG-mon-2.tikz}
    \end{center}
    The commutativity of (II) and (II) follows from the definition of tensor product of \(G\)-equivariant objects and the commutativity of (I) and (IV) follows from the naturality of \(T_g\).
    
    It is clear from the definitions that the forgetful functor $U :\C ^{\underline{G}} \rightarrow \C $ is semigroupal and strict.
    
    Consider now $\C $ being a monoidal category and the unital partial action 
    \[
    T= \left( \{ \C _g =\overline{\mathbbm{1}_g \otimes \C } \}_{g\in G}, \; \{ T_g \}_{g\in G}, \; \{ \gamma_{g,h} \} _{g,h\in G} , u \right).
    \]
    We want to give a structure of monoidal category to \(\C ^{\underline{G}}\) with the tensor product given, so we need to define the unity for such tensor product. For each $g\in G$, we have the monoidal functor \(T_g = (T_g, J^g, J^{g0})\), in which $J^{g0}=\left(\varphi^g\right)^{-1}$. Then, define the unity \((\1, \sigma^{\1})\), where \(\1\) is the unity of \(\C\) and \(\sigma^{\1}\) is defined by
    \begin{center}
        \input{diagramas/teo-CG-mon-3.tikz}
    \end{center}
    where the isomorphisms denoted by $\cong$ are given by \(l\) and \(l^{-1}\). More precisely, for $Y\in \C _{g^{-1}}$, we have
    \[
    T_g (\1 \otimes Y) \cong T_g (\1 \otimes \1_{{g^{-1}}} \otimes Y) \cong T_g (\1_{g^{-1}} \otimes Y) 
    \]
    and 
    \[
    \1_g \otimes T_g (Y) \cong \1 \otimes \1_g \otimes T_g (Y) \cong \1 \otimes T_g (Y) .
    \]
    
    The verification that 
    \(\C ^{\underline{G}}\) is indeed a monoidal category is routine and left to the reader.
\end{proof}

\subsection{The partial trace}
\phantom{|}\medskip

Throughout this subsection, we shall deal only with partial actions of a finite group $G$ upon a monoidal and abelian category $\C$ in which each functor $T_g$ is monoidal and additive and the tensor ideals are defined by central idempotent objects.

\begin{Definition}\label{Definitionoftrace}
    Consider $\C $ being an abelian monoidal category and let
    \[ 
    T= \left( \{ \C _g =\overline{\1_g \otimes \C } \}_{g\in G} , \{ T_g :\C _{g^{-1}} \rightarrow \C _g \}_{g\in G} , \{ \gamma_{g,h} \}_{g,h\in G} , u \right)
    \]
    be a partial action of a finite group $G$ on $\C $ defined by a family of central idempotent objects $\{ \1_g \}_{g\in G}$ in which the functors $T_g$ are additive. For an object $X\in \C $, define its partial trace as the object
    \[
    \Tr (X) =\bigoplus_{g\in G} T_g (X\otimes \1_{g^{-1}}) .
    \]
\end{Definition}

\begin{Lemma}
    Under the conditions of the Definition \ref{Definitionoftrace}, for each list of elements $g, g_1, \ldots , g_n\in G$ there is an isomorphism 
    \[
    \widetilde{\sigma}^{TrX}_{g, g_1, \ldots g_n} :T_g \left( \Tr (X) \otimes \1_{g^{-1}} \otimes \1_{g_1} \otimes \cdots \otimes \1_{g_n} \right) \rightarrow \Tr(X)\otimes \1_g \otimes \1_{gg_1} \otimes \cdots \otimes \1_{gg_n}
    \]
    such that $\left( \Tr(X) , \{ \widetilde{\sigma}^{TrX}_{g, g_1, \ldots, g_n} \}_{g,g_1, \ldots , g_n\in G} \right) $ defines a partially equivariant object in $\C ^{\underline{G}}$.
\end{Lemma}
\begin{proof}
    Indeed, for $g, g_1, \ldots , g_n\in G$ and $X\in \C $, we have
    \begin{align*}
    &\hskip-1,2em
        T_g \left( \Tr (X) \otimes \1_{g^{-1}} \otimes \1_{g_1} \otimes \cdots \otimes \1_{g_n} \right)
        = T_g \left( \left( \bigoplus_{h\in G} T_h (X\otimes \1_{h^{-1}})\right) \otimes \1_{g^{-1}} \otimes \1_{g_1} \otimes \cdots \otimes \1_{g_n} \right)\\
    & \cong
        T_g \left( \bigoplus_{h\in G} \left( T_h (X\otimes \1_{h^{-1}}) \otimes \1_{g^{-1}} \otimes \1_{g_1} \otimes \cdots \otimes \1_{g_n} \right) \right)  \\
    & \cong
        \bigoplus_{h\in G} T_g \left( T_h (X\otimes \1_{h^{-1}}) \otimes \1_{g^{-1}} \otimes \1_{g_1} \otimes \cdots \otimes \1_{g_n} \right)  \\
    & \cong
        \bigoplus_{h\in G} T_g \left( T_h (X\otimes \1_{h^{-1}}) \otimes \1_h \otimes \1_{g^{-1}} \otimes \1_{g_1} \otimes \cdots \otimes \1_{g_n} \right)  \\
    & \cong
        \bigoplus_{h\in G} T_g \left( T_h (X\otimes \1_{h^{-1}}) \otimes T_h( \1_{h^{-1}} \otimes \1_{(gh)^{-1}} \otimes \1_{h^{-1}g_1} \otimes \cdots \otimes \1_{h^{-1}g_n}) \right)  \\
    & \cong
        \bigoplus_{h\in G} T_g \left( T_h (X\otimes \1_{h^{-1}} \otimes \1_{(gh)^{-1}} \otimes \1_{h^{-1}g_1} \otimes \cdots \otimes \1_{h^{-1}g_n} ) \right)  \\
    & \cong
        \bigoplus_{h\in G} T_{gh} \left( X\otimes \1_{h^{-1}} \otimes \1_{(gh)^{-1}} \otimes \1_{h^{-1}g_1} \otimes \cdots \otimes \1_{h^{-1}g_n}  \right)  \\
    & \cong
        \bigoplus_{h\in G} T_{gh} \left( X\otimes \1_{(gh)^{-1}}\right)  \otimes T_{gh} \left( \1_{h^{-1}} \otimes \1_{(gh)^{-1}} \otimes \1_{h^{-1}g_1} \otimes \cdots \otimes \1_{h^{-1}g_n}  \right)  \\
    & \cong
        \bigoplus_{h\in G} T_{gh} \left( X\otimes \1_{(gh)^{-1}}\right)  \otimes  \1_g \otimes \1_{gh} \otimes \1_{gg_1} \otimes \cdots \otimes \1_{gg_n}   \\
    & \cong
        \Tr (X) \otimes \1_g \otimes \1_{gh} \otimes \1_{gg_1} \otimes \cdots \otimes \1_{gg_n}.
    \end{align*}
    Then, ignoring the obvious isomorphism coming from the distributivity of the tensor product with respect to the direct sum, the additivity of the functor $T_g$ and the absorptions of the respective units, one can write the isomorphism $\widetilde{\sigma}^{\Tr(X)}_{g, g_1, \ldots g_n}$ as the composition:
    \begin{equation}\label{sigmatildotraco}
    \begin{split}
    \widetilde{\sigma}^{\Tr(X)}_{g, g_1, \ldots g_n}
        & = \left( \bigoplus_{h\in G} \left( T_{gh}(X\otimes \1_{(gh)^{-1}})\otimes \left( \varphi^{gh}_{g,gg_1, \ldots , gg_n} \right)^{-1} \right) \right) \\
        & \quad \circ \left( \bigoplus_{h\in G} \left( J^{gh}_{X\otimes \1_{(gh)^{-1}} ,\1_{h^{-1}} \otimes \1_{(gh)^{-1}} \otimes \1_{h^{-1}g_1} \otimes \cdots \otimes \1_{h^{-1}g_n}} \right)^{-1} \right) \\
        & \quad \circ \left( \bigoplus_{h\in G} \left( \gamma_{g,h}\right)_{X\otimes \1_{h^{-1}} \otimes \1_{(gh)^{-1}} \otimes \1_{h^{-1}g_1} \otimes \cdots \otimes \1_{h^{-1}g_n}} \right) \\
        & \quad \circ  \left( \bigoplus_{h\in G} T_g \left( T_{h}(X\otimes \1_{h^{-1}})\otimes J^h_{X\otimes \1_{h^{-1}} ,\1_{h^{-1}} \otimes \1_{(gh)^{-1}} \otimes \1_{h^{-1}g_1} \otimes \cdots \otimes \1_{h^{-1}g_n}} \right) \right) \\
        & \quad \circ \left( \bigoplus_{h\in G} T_g \left( T_{h}(X\otimes \1_{h^{-1}})\otimes  \varphi^{h}_{g^{-1},g_1, \ldots , g_n}  \right) \right).
    \end{split}
    \end{equation}
    We shall omit here the long, but straightforward chain of calculations needed to check that the family $\{ \widetilde{\sigma}^{\Tr(X)}_{g, g_1, \ldots g_n} \}_{g,g_1, \ldots , g_n}$ satisfies the commutativity of the diagram (\ref{tildesigma}). Therefore, the data $\left( \Tr(X), \{ \widetilde{\sigma}^{\Tr(X)}_{g,g_1, \ldots g_n }\}_{g,g_1, \ldots, g_n} \right)$ define a partially equivariant object in $\C ^{\underline{G}}$.
\end{proof}

\begin{Theorem}
    Let  
    \[ 
    T= \left( \{ \C _g =\overline{\1_g \otimes \C } \}_{g\in G} , \{ T_g \colon \C_{g^{-1}} \rightarrow \C _g \}_{g\in G} , \{ \gamma_{g,h} \}_{g,h\in G}  , u\right)
    \]
    be a partial action of a finite group $G$ on a monoidal and abelian category $\C $ defined by a family of central idempotent objects $\{ \1_g \}_{g\in G}$ in which the functors $T_g$ are additive. Then the trace defines a semigroupal, but not monoidal, category and additive functor $\Tr \colon \C \rightarrow \C ^{\underline{G}}$.   
\end{Theorem}

\begin{proof}
    For any morphism $f:X\rightarrow Y$ in $\C $, define the map
    \[
    \Tr(f)=\bigoplus_{g\in G} T_g( f\otimes \1_{g^{-1}}) .
    \]
    This is a well defined morphism in $\C $ between $\Tr(X)$ and $\Tr (Y)$ by the universal property of the direct sum. In order to check that $\Tr(f)$ is indeed a morphism in the category $\C ^{\underline{G}}$, take an element $g\in G$ and an object $Z\in \C _{g^{-1}}$, then
    \begin{align*}
    &\hskip-1,2em
        \left( \Tr (f) \otimes  T_g (Z) \right)  \circ \left( \sigma^{\Tr(X)}_g \right)_Z
         = \left( \Tr (f) \otimes \1_g \otimes  T_g (Z) \right) \circ \left( \widetilde{\sigma}^{\Tr(X)}_g \otimes T_g (Z) \right) \circ \left( J^g_{\Tr(X) \otimes \1_{g^{-1}}, Z}\right)^{-1} \\
    \overset{(I)}&=
        \left( \widetilde{\sigma}^{\Tr(Y)}_g \otimes T_g (Z) \right) \circ \left( T_g (\Tr(f) \otimes \1_{g^{-1}}) \otimes T_g (Z) \right)  \circ \left( J^g_{\Tr(X) \otimes \1_{g^{-1}}, Z}\right)^{-1} \\
    \overset{(II)}&=
        \left( \widetilde{\sigma}^{\Tr(Y)}_g \otimes T_g (Z) \right) \circ   \left( J^g_{\Tr(Y) \otimes \1_{g^{-1}}, Z}\right)^{-1}  T_g \left( \Tr(f) \otimes \1_{g^{-1}}) \otimes Z \right)  \\
    & =
        \left( \sigma^{\Tr(Y)}_g \right)_Z \circ T_g \left( \Tr(f) \otimes Z \right) .
    \end{align*}
    Therefore, $\Tr$ is a functor between the categories $\C $ and $\C ^{\underline{G}}$.

    The functor $\Tr$ is additive because of the universal property of the direct sum and it is semigroupal from the distributivity of the tensor product with respect to the direct sum. It is not monoidal because 
    \[
    \Tr(\1) =\bigoplus_{g\in G} T_g (\1 \otimes \1_{g^{-1}}) \cong \bigoplus_{g\in G} \1_g .
    \qedhere
    \]
    \end{proof}
    
    In general, for $(X, \sigma )\in \C ^{\underline{G}}$, we have
    \[
    \Tr(X) =\bigoplus_{g\in G} T_g (X\otimes \1_{g^{-1}}) \cong \bigoplus_{g\in G} X\otimes \1_g .
    \]
    Then $X$ is isomorphic to a direct summand of $\Tr(X)$ as objects in $\C $. But, in general, this isomorphism cannot be lifted to the category $\C ^{\underline{G}}$.

\subsection{The partial smash product \texorpdfstring{$\C \smp G$}{\emph{C G}}}
\phantom{|}\medskip

Given a partial action $\alpha= \left( \{ A_g \}_{g\in G} , \{ \alpha_g :A_{g^{-1}} \rightarrow A_g \}_{g\in G} \right)$ of a group $G$ on an algebra $A$, one can define the partial smash product $A\smp G$ \cite{DE}. Basically, the partial smash product, as a vector space, is simply the direct sum 
\[ (A\smp G)^{(0)} =\bigoplus_{g\in G} A_g \delta_g, 
\]
and the multiplication is given by
\begin{equation}\label{psp}
(a_g \delta_g)(b_h \delta_h) =\alpha_g (\alpha_{g^{-1}} (a_g) b_h) \delta_{gh} .
\end{equation}
One cannot guarantee, in general, that even this above defined  multiplication is associative. Neither is it possible to  ensure the existence of unity. Associativity of the multiplication (\ref{psp}) can be ensured by requiring that the ideals $A_g \trianglelefteq A$ be non-degenerate or idempotent. 

In the case where $\alpha$ is a unital partial action, that is, the ideal $A_g$ is generated by a central idempotent element $\1_g \in A$ for all $g\in G$, then the ideals become idempotent, making the multiplication (\ref{psp}) associative and the element $\1_e \delta_e$ becomes automatically the unit for this multiplication. Moreover, formula (\ref{psp}) admits an equivalent simplified version as
\[
(a_g \delta_g)(b_h \delta_h)=a_g \alpha_g (b_h \1_{g^{-1}}) \delta_{gh}.
\]

The partial group actions on semigroupal categories can be interpreted as a categorification of partial group actions on algebras. Among the partial group actions on semigroupal categories, the unital ones are those with a richer class of examples and those which allow one to define new and relevant constructions upon these actions. Then, for the construction of a partial smash product, we will restrict ourselves to unital partial actions, that is partial actions in which the tensor ideals are generated by central idempotent objects. As the categorical analogue of the partial smash product will require additive structures, in this section we shall consider $\C $ being a $\Bbbk$-linear abelian and monoidal category.

\begin{Definition}
    Let \(T = \left( \{ \C _g =\overline{\1_g \otimes \C } \}_{g\in G}, \{ T_g \}_{g \in G}, \{ \gamma_{g,h} \}_{g, h \in G}, u \right)\) be a unital partial action of a group \(G\) on a $\Bbbk$-linear, abelian and monoidal category \C\ . The partial smash product is the category \(\C \smp G \coloneqq \bigoplus_{g \in G} \C_g \delta_g \) whose objects are
    \begin{align*}
         (\C \smp G)^{(0)} &= \left\{ \bigoplus_{g \in G} X_g \delta_g \mathrel{} \middle | \mathrel{} X_g \in  (\C_g)^{(0)} \right\}
    \\
        \intertext{and morphisms}
         (\C \smp G)^{(1)} &= \left\{ \bigoplus_{g \in G} (f_g, g) \colon \bigoplus_{g\in G} X_g \delta_g \longto \bigoplus_{g \in G} X'_g \delta_g \mathrel{} \middle | \mathrel{} f_{g} \colon X_g \longto X'_g \right\}
    \end{align*}
    with monoidal structure given by
    \begin{equation}\label{smash_tensor_product_objects}
        ( X_{g} \delta_{g} ) \boxtimes ( X_{h} \delta_{h} )
            \coloneqq ( X_{g} \otimes T_{g} (X_{h} \otimes \1_{g^{-1}} )) ) \delta_{gh}
    \end{equation}        
    and
    \begin{equation}\label{smash_tensor_product_morphisms}
        ( f_g, g ) \boxtimes ( f_h, h ) \coloneqq ( f_g \otimes T_{g}(f_{h} \otimes \1_{g^{-1}}), g  h ).
    \end{equation}
    The unit is \( \1_{\C \smp G} \coloneqq \1_{e} \delta_{e} = \1_\C \delta_{e} \).
\end{Definition}

For the case of $T:\underline{G} \rightarrow \Aut_{\otimes} (\C )$ being a global action of $G$ by monoidal autoequivalences, then the partial smash product $\C  \underline{\rtimes} G$ coincides with the semidirect product $\C [G] =\bigoplus_{g\in G} \C $, introduced by D. Tambara in \cite{Tambara}.

\begin{Proposition}
   Let \(T = \left( \{ \C _g =\overline{\1_g \otimes \C }\}_{g\in G}, \{ T_g \}_{g \in G}, \{ \gamma_{g,h} \}_{g, h \in G}, u \right) \) be a unital partial action of a group \(G\) on a monoidal category \C\ . Then the partial smash product \( \left( \C \smp G , \boxtimes, \1_{\C \smp G} \right) \) defines a monoidal category. 
\end{Proposition}

\begin{proof} Even when the monoidal category $\C $ is supposed to be strict, the form of the tensor product defined in (\ref{smash_tensor_product_objects}) and (\ref{smash_tensor_product_morphisms}) gives rise to nontrivial associator and unit morphisms.

For the associator, consider the product
\[
( X_{g} \delta_{g} \boxtimes X_{h} \delta_{h} ) \boxtimes X_{l} \delta_{l}
    = ( X_{g} \otimes T_{g} ( X_{h} \otimes \1_{g^{-1}} ) ) \delta_{gh} \boxtimes X_{l} \delta_{l} 
    = X_{g} \otimes T_{g} ( X_{h} \otimes \1_{g^{-1}} ) \otimes T_{gh} ( X_{l} \otimes \1_{(gh)^{-1}} ) \delta_{ghl} .
\]
On the other hand
\[
X_{g} \delta_{g} \boxtimes ( X_{h} \delta_{h} \boxtimes X_{l} \delta_{l} ) =X_{g} \delta_{g} \boxtimes ( X_{h} \otimes T_{h} ( X_{l} \otimes \1_{h^{-1}} ) ) \delta_{hl} =X_{g} \otimes T_{g} ( X_{h} \otimes T_{h} ( X_{l} \otimes \1_{h^{-1}} ) \otimes \1_{g^{-1}} ) \delta_{ghl} .
\]
Then, we have the following sequence of isomorphisms:\\
\begin{align*}
&\hskip-1,2em
    X_{g} \otimes T_{g} ( X_{h} \otimes \1_{g^{-1}} ) \otimes T_{gh} ( X_{l} \otimes \1_{(gh)^{-1}} ) \delta_{ghl} =\\
&=
    X_{g} \otimes T_{g} ( X_{h} \otimes \1_{g^{-1}} ) \otimes T_{gh} ( X_{l} \otimes \1_{(gh)^{-1}} ) \otimes \1_{g} \otimes \1_{gh} \delta_{ghl} \\
&\cong
    X_{g} \otimes T_{g} ( X_{h} \otimes \1_{g^{-1}} ) \otimes T_{gh} ( X_{l} \otimes \1_{(gh)^{-1}} ) \otimes T_{gh}(\1_{h^{-1}} \otimes \1_{(gh)^{-1}}) \delta_{ghl} \\
& \cong
    X_{g} \otimes T_{g} ( X_{h} \otimes \1_{g^{-1}} ) \otimes T_{gh} ( X_{l} \otimes \1_{h^{-1}} \otimes \1_{(gh)^{-1}} ) \delta_{ghl} \\
& \cong
    X_{g} \otimes T_{g} ( X_{h} \otimes \1_{g^{-1}} ) \otimes T_{g} ( T_{h} ( X_{l} \otimes \1_{h^{-1}} \otimes \1_{(gh)^{-1}} ) ) \delta_{ghl} \\
& \cong
    X_{g} \otimes T_{g} ( X_{h} \otimes \1_{g^{-1}} ) \otimes T_{g} ( T_{h} ( X_{l} \otimes \1_{h^{-1}} ) \otimes T_{h} ( \1_{(gh)^{-1}} \otimes \1_{h^{-1}} ) ) \delta_{ghl} \\
& \cong
    X_{g} \otimes T_{g} ( X_{h} \otimes \1_{g^{-1}} ) \otimes T_{g} ( T_{h} ( X_{l} \otimes \1_{h^{-1}} ) \otimes \1_{g^{-1}} ) \delta_{ghl} \\
& \cong
    X_{g} \otimes T_{g} ( X_{h} \otimes T_{h} ( X_{l} \otimes \1_{h^{-1}} ) \otimes \1_{g^{-1}} ) \delta_{ghl} .
\end{align*}
Then the associator can be written explicitly as
\begin{align*}
    A_{X_g \delta_g ,X_h \delta_h , X_l \delta_l}
&=
    \left( X_{g} \otimes J^g_{X_{h} \otimes \1_{g^{-1}} , T_{h} ( X_{l} \otimes \1_{h^{-1}} ) \otimes \1_{g^{-1}}} \right)  \circ \\
 & \circ
    \left(  X_{g} \otimes T_{g} ( X_{h} \otimes \1_{g^{-1}} ) \otimes T_{g} ( T_{h} ( X_{l} \otimes \1_{h^{-1}} ) \otimes (\varphi^h_{g^{-1}})^{-1}) \right) \\
 & \circ
    \left( X_{g} \otimes T_{g} ( X_{h} \otimes \1_{g^{-1}} ) \otimes T_{g} ((J^h)^{-1}_{X_{l} \otimes \1_{h^{-1}} , \1_{h^{-1}} \otimes \1_{(gh)^{-1}}}) \right) \\
 & \circ
    \left(  X_{g} \otimes T_{g} ( X_{h} \otimes \1_{g^{-1}} ) \otimes \left( \gamma_{g,h} \right)^{-1}_{X_{l} \otimes \1_{h^{-1}} \otimes \1_{(gh)^{-1}}}  \right) \\
 & \circ
    \left(  X_{g} \otimes T_{g} ( X_{h} \otimes \1_{g^{-1}} ) \otimes J^{gh}_{X_{l} \otimes \1_{(gh)^{-1}} ,\1_{h^{-1}} \otimes \1_{(gh)^{-1}} } \right) \\
 & \circ
    \left(  X_{g} \otimes T_{g} ( X_{h} \otimes \1_{g^{-1}} ) \otimes T_{gh} ( X_{l} \otimes \1_{(gh)^{-1}} ) \otimes \varphi^{gh}_{g}\1_{g} \right)
\end{align*}

The left and right unit morphisms,           
    \begin{align*}
        L \colon \1_{\C \smp G} \boxtimes \esp \Rightarrow \Id_{\C \smp G}(\esp) \colon \C \smp G \longto \C \smp G
    \\
        \intertext{and}
        R \colon \esp \boxtimes \1_{\C \smp G} \Rightarrow \Id_{\C \smp G}(\esp) \colon \C \smp G \longto \C \smp G
    \end{align*}
are given, respectively by
    \[
        \1_{\C \smp G} \boxtimes X_g \delta_g = 
        \1_e \delta_e  \boxtimes X_g \delta_g =  
        (\1_\C \otimes T_e( X_g \otimes \1_\C ) ) \delta_g =T_e (X_g) \delta_g \cong X_g \delta_g
    \]
    and
    \[
        X_g \delta_g \boxtimes \1_{\C \smp G} = 
        X_g \delta_g  \boxtimes \1_e \delta_e =  
        (X_g \otimes T_g ( \1_\C \otimes \1_{g^{-1}} )) \delta_g = (X_g \otimes T_g (\1_{g^{-1}})) \delta_g  \Iso X_g \delta_g,
    \]
or, explicitly,
\[
L_{X_g \delta_g} =u^{-1}_{X_g} \qquad \text{ and } \qquad R_{X_g \delta_g} =X_g \otimes (\varphi^g)^{-1}.
\]

For sake of brevity, we shall omit the proofs of the triangle and the pentagon axioms for the monoidal structure of the category $\C  \smp G$, consisting only of a long and exhaustive calculation with composition of morphisms which cannot be properly displayed in this paper and is not very instructive. Even in the global case, where the associator consists of the composition of only three morphisms instead of six, as in the partial case, the verification of the pentagon axiom took four pages (see \cite{Tambara}, Section 8).
\end{proof}

Given a partial action \( T \) of a group \(G\) on a monoidal category \C, one can define a functor
\begin{equation}\label{eq:rep-par-no-smash}
\begin{split}
    \pi_{0} \colon \underline{G} & \longto \C \smp G\\
    g & \longmapsto \1_{g} \delta_{g}
\end{split} ,
\end{equation}
in which $\underline{G}$ is the monoidal category whose objects are the elements of $G$, the only morphisms are the identities on objects and the tensor product is the product among the elements of the group.  

\begin{Proposition}
    The functor presented in Equation~\eqref{eq:rep-par-no-smash} satisfies the following properties:
    \begin{enumerate}[(PR1)]
        \item \( \pi_{0}(e) = \1_{\C \smp G}, \)\label{prop:def-rep-par-i}
        \item \( \pi_{0} (g) \boxtimes \pi_{0} (h) \boxtimes \pi_{0} (h^{-1}) \Iso \pi_{0} (g h) \boxtimes \pi_{0} (h^{-1}), \)\label{prop:def-rep-par-ii}
        \item \( \pi_{0}(g^{-1}) \boxtimes \pi_{0}(g) \boxtimes \pi_{0}(h) \Iso \pi_{0}(g^{-1}) \boxtimes \pi_{0}(gh), \)\label{prop:def-rep-par-iii}
        \item \( \pi_{0}(g) \boxtimes \pi_{0}(g^{-1}) \boxtimes \pi_{0}(g) \Iso \pi_{0}(g), \)\label{prop:def-rep-par-iv}
    \end{enumerate}
    for any \(g,h \in G\).
\end{Proposition}

\begin{proof}
    \ref{prop:def-rep-par-i}:
    \[
        \pi_{0}(e) = \1_{e} \delta_{e} = \1_{\C \smp G}.
    \]
    \ref{prop:def-rep-par-ii}: Given any \(g,h \in G\), we have that
    \begin{align*}
    \pi_{0} (g) \boxtimes \pi_{0} (h) \boxtimes \pi_{0} (h^{-1})
        & = (\1_{g} \delta_{g}) \boxtimes (\1_{h} \delta_{h}) \boxtimes (\1_{h^{-1}} \delta_{h^{-1}})\\
        & = (\1_{g} \delta_{g}) \boxtimes ((\1_{h} \otimes T_h(\1_{h^{-1}} \otimes \1_{h^{-1}})) \delta_{e})\\
        & \Iso (\1_{g} \delta_{g}) \boxtimes(\1_{h} \delta_{e})\\
        & = (\1_{g} \otimes T_{g}(\1_{h} \otimes 1_{g^{-1}})) \delta_{g}\\
        & \Iso (\1_{g} \otimes \1_{gh}) \delta_{g}
    \\
        \intertext{and}
    \pi_{0} (g h) \boxtimes \pi_{0} (h^{-1})
        & = (\1_{gh} \delta_{gh}) \boxtimes (\1_{h^{-1}} \delta_{h^{-1}})\\
        & = (\1_{gh} \otimes T_{gh}(\1_{h^{-1}} \otimes \1_{(gh)^{-1}})) \delta_{g}\\
        & = (\1_{gh} \otimes \1_{g} ) \delta_{g}.
    \end{align*}
    
    \ref{prop:def-rep-par-iii}:
    \begin{align*}
    \pi_{0}(g^{-1}) \boxtimes \pi_{0}(g) \boxtimes \pi_{0}(h)
        & = (\1_{g^{-1}} \delta_{g^{-1}}) \boxtimes (\1_{g} \delta_{g}) \boxtimes (\1_{h} \delta_{h})\\
        & = ((\1_{g^{-1}} \otimes T_{g^{-1}}(\1_{g} \otimes \1_{g} )) \delta_{e}) \boxtimes (\1_{h} \delta_{h})\\
        & \Iso (\1_{g^{-1}}  \otimes T_{e}(\1_{h} \otimes \1_{e})) \delta_{h}\\
        & \Iso (\1_{g^{-1}} \otimes \1_{h} ) \delta_{h}
    \\
        \intertext{and}
    \pi_{0}(g^{-1}) \boxtimes \pi_{0}(gh)
        & = (\1_{g^{-1}} \delta_{g^{-1}}) \boxtimes (\1_{gh} \delta_{gh})\\
        & = (\1_{g^{-1}} \otimes T_{g^{-1}}(\1_{gh} \otimes \1_{g})) \delta_{h}\\
        & \Iso (\1_{g^{-1}} \otimes \1_{h} ) \delta_{h}.
    \end{align*}
    \ref{prop:def-rep-par-iv}: It follows directly from \ref{prop:def-rep-par-iii}, \ref{prop:def-rep-par-i} and from the monoidal structure of $\C  \underline{\rtimes} G$.
\end{proof}

Given a partial action \( T \) of a group \(G\) on a monoidal category \C, one can define a functor
\begin{equation}\label{eq:funtor-no-smash}
\begin{split}
    \phi_{0} \colon \C & \longto \C \smp G\\
     (\C)^{(0)}  \ni     X & \longmapsto X \delta_{e}\\
    \C^{(1)} \ni     f & \longmapsto (f, e)
\end{split}
\end{equation}

It is easy to see that the functor $\phi_0$ is monoidal. Indeed,
\begin{align*}
    \phi_0 (X)\boxtimes \phi_0 (Y) & = (X\delta_e )\boxtimes (Y\delta_e) \\
    & = (X\otimes T_e (Y\otimes \1_e))\delta_e \\
    & \Iso (X\otimes Y) \delta_e \\
    & \Iso \phi_0 (X\otimes Y) ,
\end{align*}
and
\[
\phi_0 (\1_{\C }) =\1_{\C } \delta_e =\1_{\C  \underline{\rtimes}G} .
\]

\begin{Proposition}
    The functors $\pi_0 :\underline{G} \longto \C \smp G$, defined by (\ref{eq:rep-par-no-smash}) and $\phi_0 : \C  \longto \C \smp G$ satisfy the following identities, for each object $X\in \C $ and each element $g\in G$:
    \begin{enumerate}
        \item $\phi_0 (X) \boxtimes \pi_0 (g) \boxtimes \pi_0(g^{-1}) \cong \pi_0 (g) \boxtimes \pi_0(g^{-1}) \boxtimes \phi_0 (X)$.
        \item $\pi_0 (g) \boxtimes \phi_0 (X) \boxtimes \pi_0(g^{-1}) \cong  \phi_0(T_g (X \otimes \1_{g^{-1}}))$.
    \end{enumerate}
\end{Proposition}

\begin{proof} For item (i), on the one hand, we have
\begin{align*}
\phi_0 (X) \boxtimes \pi_0 (g) \boxtimes \pi_0(g^{-1})
    & =    X \delta_e \boxtimes \1_g \delta_g \boxtimes \1_{g^{-1}} \delta_{g^{-1}}\\
    & =    X \delta_e \boxtimes (\1_g \otimes T_g(\1_{g^{-1}} \otimes \1_{g^{-1}})) \delta_e\\
    & \Iso X \delta_e \boxtimes \1_g \delta_e\\
    & =   (X \otimes T_e ( \1_g \otimes \1_e )) \delta_e
\\
    & \Iso(X \otimes \1_g ) \delta_e ,
\\
    \intertext{and, on the other hand,}
\pi_0 (g) \boxtimes \pi_0(g^{-1}) \boxtimes \phi_0 (X)
    & =    \1_g \delta_g \boxtimes \1_{g^{-1}} \delta_{g^{-1}} \boxtimes X \delta_e \\
    & =   (\1_g \otimes T_g (\1_{g^{-1}} \otimes \1_{g^{-1}} ))\delta_e \boxtimes X \delta_e \\
    & \Iso \1_g \delta_e \boxtimes X \delta_e \\
    & =   (\1_g \otimes T_e ( X \otimes \1_e ) ) \delta_e \\
    & \Iso(\1_g \otimes X ) \delta_e .\\
\end{align*}

Now, for item (ii),
\begin{align*}
\pi_0 (g) \boxtimes \phi_0 (X) \boxtimes \pi_0(g^{-1})
    & =    \1_g \delta_g \boxtimes X \delta_e \boxtimes \1_{g^{-1}} \delta_{g^{-1}}\\
    & =    \1_g \delta_g \boxtimes (X \otimes T_e (\1_{g^{-1}} \otimes \1_e ) \delta_{g^{-1}} \\
    & \Iso \1_g \delta_g \boxtimes (X \otimes \1_{g^{-1}} ) \delta_{g^{-1}} \\
    & =   (\1_g \otimes T_g (X \otimes \1_{g^{-1}} \otimes \1_{g^{-1}}) ) \delta_e \\
    & \Iso(T_g (X \otimes \1_{g^{-1}}) ) \delta_e \\
    & =    \phi_0(T_g (X \otimes \1_{g^{-1}})) .\qedhere
\end{align*}
\end{proof}

\begin{Definition}
    Consider a group $G$, two monoidal categories $\C $ and $\mathcal{D}$, and a unital partial action 
    \[
    T = \left( \{ T_g \}_{g \in G}, \{ \gamma_{g,h} \}_{g, h \in G}, u \right)
    \]
    of $G$ on $\C $. A covariant pair $(\phi , \pi)$  related to $\mathcal{D}$ is a pair of functors $\phi :\C  \rightarrow \mathcal{D}$ and $\pi :\underline{G} \rightarrow \mathcal{D}$ such that
    \begin{enumerate}[({CV}1)]
      \item The functor $\phi :\C  \rightarrow \mathcal{D}$ is monoidal.
      \item The functor $\pi :\underline{G} \rightarrow \mathcal{D}$ satisfies
      \begin{enumerate}
          \item \( \pi (e) = \1_{\mathcal{D}}, \)
          \item \( \pi (g) \bar{\otimes} \pi (h) \bar{\otimes} \pi (h^{-1}) \Iso \pi (g h) \bar{\otimes} \pi (h^{-1})\), for all $g,h\in G$;
          \item \( \pi (g^{-1}) \bar{\otimes} \pi (g) \bar{\otimes} \pi (h) \Iso \pi (g^{-1}) \bar{\otimes} \pi (gh)\), for all $g,h \in G$.
      \end{enumerate}
      \item The functors $\phi$ and $\pi$ satisfy the following compatibility relations, for all objects $X\in \C $ and for all elements $g\in G$:
      \begin{enumerate}
         \item $\phi (X) \bar{\otimes} \pi (g) \bar{\otimes} \pi (g^{-1}) \cong \pi (g) \bar{\otimes} \pi (g^{-1}) \bar{\otimes} \phi (X)$.
        \item $\pi (g) \bar{\otimes} \phi (X) \bar{\otimes} \pi (g^{-1}) \cong  \phi (T_g (X \otimes \1_{g^{-1}}))$.
      \end{enumerate}
    \end{enumerate}
\end{Definition}

\begin{Theorem}\label{parcovariante}
    Given a covariant pair \((\phi, \pi)\) related with a category \(\mathcal{D}\), there exists a monoidal functor
    \[
        \Psi \colon \C \smp G \longto \mathcal{D}
    \]
    such that the following diagram commutes:
    \begin{center}
        \input{diagramas/teo-iso-smash.tikz}
    \end{center}
\end{Theorem}

\begin{proof}
    Define
    \begin{align*}
        \Psi \colon \C \smp G & \longto \mathcal{D}\\
        \bigoplus_{g \in G} X_g \delta_g & \longmapsto \bigoplus_{g \in G} \phi(X_g) \mathop{\bar\otimes} \pi(g)\\
        \bigoplus_{g \in G} (f_g, g) & \longmapsto \bigoplus_{g \in G} ( \phi(f_g) \mathop{\bar\otimes} \pi(g) ).
    \end{align*}

    Given \(X \in  (\C)^{(0)}\), we have that
    \[
        \Psi(\phi_0 (X)) = 
        \Psi(X \delta_e) = 
        \phi(X) \mathop{\bar\otimes} \pi(e) = 
        \phi(X) \mathop{\bar\otimes} \1_{\mathcal{D}} \Iso
        \phi(X).
    \]
    
    Now, consider a morphism \( f:X\rightarrow Y\) in $\C $, then
    \[
        \Psi(\phi_0 (f)) = 
        \Psi (f ,e) = 
        \phi(f) \mathop{\bar\otimes} \pi(e) = 
        \phi(f) \mathop{\bar\otimes} \1_{\mathcal{D}} \Iso
        \phi(f).
    \]
    
    Given \(g \in G\),
    \begin{align*}
    \Psi(\pi_0(g))
        & = \Psi(\1_g \delta_g) \\
        & = \phi(\1_g) \mathop{\bar\otimes} \pi(g) \\
        & \Iso \phi( T_g( \1_{\C} \otimes \1_{g^{-1}} ) ) \mathop{\bar\otimes} \pi(g) \\
        & \Iso \pi(g) \mathop{\bar\otimes} \phi( \1_{\C} ) \mathop{\bar\otimes} \pi(g^{-1}) \mathop{\bar\otimes} \pi(g) \\
        & \Iso \pi(g) \mathop{\bar\otimes} \pi(g^{-1}) \mathop{\bar\otimes} \pi(g) \mathop{\bar\otimes} \phi( \1_{\C} ) \\
        & \Iso \pi(g) \mathop{\bar\otimes} \1_{\mathcal{D}} \\
        & \Iso \pi(g).
    \end{align*}

    In order to verify that the functor $\Psi$ is monoidal, first observe that
    \[
    \Psi (\1_{\C  \underline{\rtimes} G}) =\Psi (\1_{\C }\delta_e) =\phi (\1_{\C }) \bar{\otimes} \pi (e) =\1_{\mathcal{D}} \bar{\otimes} \1_{\mathcal{D}} \cong \1_{\mathcal{D}}.
    \]
    Now, take two objects $X_g \delta_g$ and $Y_h \delta_h$ in $\C \underline{\rtimes} G$, then
    \begin{align*}
        \Psi ((X_g \delta_g) \boxtimes (Y_h \delta_h)) & = \Psi ((X_g \otimes T_g(Y_h \otimes \1_{g^{-1}}))\delta_{gh}) \\
        & = \phi (X_g \otimes T_g(Y_h \otimes \1_{g^{-1}})) \bar{\otimes} \pi (gh) \\
        & \cong \phi (X_g ) \bar{\otimes} \phi( T_g(Y_h \otimes \1_{g^{-1}})) \bar{\otimes} \pi (gh) \\
        & \cong \phi (X_g ) \bar{\otimes} \pi (g) \phi( Y_h ) \bar{\otimes} \pi (g^{-1}) \bar{\otimes} \pi (gh) \\
        & \cong \phi (X_g ) \bar{\otimes} \pi (g) \phi( Y_h ) \bar{\otimes} \pi (g^{-1}) \bar{\otimes} \pi (g) \bar{\otimes} \pi(h) \\
        & \cong \phi (X_g ) \bar{\otimes} \pi (g)  \bar{\otimes} \pi (g^{-1}) \bar{\otimes} \pi (g) \bar{\otimes} \phi( Y_h )  \bar{\otimes} \pi(h) \\
        & \cong \phi (X_g ) \bar{\otimes} \pi (g)  \bar{\otimes} \phi( Y_h )  \bar{\otimes} \pi(h) \\
        & =  \Psi (X_g \delta_g) \bar{\otimes} \Psi (Y_h \delta_h) .\qedhere
     \end{align*}
\end{proof}

In particular, in the case of the monoidal category $\End (\mathcal{M})$, for a given category $\mathcal{M}$, a monoidal functor $\phi : \C  \rightarrow \End(\mathcal{M})$ means that $\mathcal{M}$ is a module category over the monoidal category $\C $. Then, we have the following result.

\begin{Corollary}
    Let $\C $ be a $\Bbbk$-linear, abelian and monoidal category, $T = \left( \{ T_g \}_{g \in G}, \{ \gamma_{g,h} \}_{g, h \in G}, u \right)$ a partial action of a group $G$ on $\C $ generated by central idempotent objects $\{\1_g \}_{g\in G}$. Let $\mathcal{M}$ be a $\C $-module category with the module structure defined by a monoidal functor $\phi :\C  \rightarrow \End(\mathcal{M})$ and a functor $\pi :\underline{G} \rightarrow \End(\mathcal{M})$, such that $\phi$ and $\pi$ satisfy the conditions of the Theorem \ref{parcovariante}. Then $\mathcal{M}$ is a $\C  \smp G$-module category.
\end{Corollary}

In particular, if $\C $ is a $\Bbbk$-linear, abelian, monoidal category and $T = \left( \{ T_g \}_{g \in G}, \{ \gamma_{g,h} \}_{g, h \in G}, u \right)$ is a partial action of a group $G$ on  $\C $ generated by central idempotent objects $\{\1_g \}_{g\in G}$, then, as a consequence of Theorem \ref{teoremarepparcial}, the functor $\pi : \underline{G} \rightarrow \End(\C )$ given by $\pi (g) (X)=T_g (X\otimes \1_{g^{-1}})$ together with the functor $\phi :\C  \rightarrow \End(\C )$ given by $\phi (X)(Y)=X\otimes Y$, satisfy the conditions of Theorem \ref{parcovariante}. Therefore, $\C $ is a $\C  \smp G$-module category with the action given by
\[
X\delta_g\odot Y =X\otimes T_g (Y\otimes \1_{g^{-1}}) .
\]
Conversely, any $\C  \smp G$-module category can be made into a $\C $-module category with the action
\[
X\overline{\otimes} M=X\delta_e \odot M.
\]

\section{Conclusions and outlook}
We presented the notion of a partial group action on a monoidal category and, following the steps of \cite{DE}, we constructed the standard globalization for a partial action of a group $G$ on a monoidal category $\C $. In references \cite{KK} and \cite{SV}, the authors proposed two different, but related, categorical approaches for globalization of a partial group action. It would be interesting to explore the globalization of a partial group action on a monoidal category from an intrinsically categorical point of view.

Following \cite{Tambara}, we construct the category of partial invariants, $\C ^{\underline{G}}$ or the associated partial equivariantized category, and the partial smash product category, $\C  \smp G$. These are, respectively, generalizations of the equivariantized category $\C ^G$ and the semidirect product $\C [G]$, introduced in \cite{Tambara}, for the case of a global action $T:\underline{G} \rightarrow \End_{\otimes} (\C )$.

The Morita equivalence between the monoidal categories $\C ^G$ and $\C [G]$, for $G$ a finite group is also discussed in \cite{Tambara}. The basic idea comes from the fact that for the monoidal category $\underline{Vect}_{\Bbbk}$ of $\Bbbk$-vector spaces, with the trivial group action, that is, each element $g\in G$ acts as the identity functor in $\underline{Vect}_{\Bbbk}$. In that case, the equivariantized category $\underline{Vect}_{\Bbbk}^{G}$ is equivalent to the category of $\Bbbk G$-modules, and the semidirect product $\underline{Vect}_{\Bbbk} [G]$ is equivalent to the category of $G$-graded vector spaces, which is the category of $(\Bbbk G)^*$-modules. Then, in this case, the Morita equivalence becomes a particular case of the classical Morita equivalence between the monoidal categories of $H$-modules and $H^*$-modules, for $H$ being a finite-dimensional Hopf algebra \cite{AM}.

For partial group actions, the relationship between $\C ^{\underline{G}}$ and $\C  \smp G$ becomes more complicated. The main reason is that we cannot use the same comparison with the category of $\Bbbk$-vector spaces, because the trivial action, being global, lead to the same classical constructions of the equivariantized category and of the semidirect product. Hence, the subtleties concerning the partial actions are completely eluded. Then, we left the study of Morita equivalence in the partial case for a future investigation.

There are two results coming from the theory of partial group actions whose categorical generalizations are worthy to be studied in future investigations: Consider a unital partial action $\alpha$ of a finite group $G$, over an algebra $A$ and denote by $\beta$ the globalization of $\alpha$, acting over an algebra $B$. Then, in \cite{DFP}, it is proved that the partial trace $\underline{Tr}:A\rightarrow A^{\underline{G}}$, given by
\[
\underline{Tr}(a) =\frac{1}{|G|} \sum_{g\in G} \alpha_g (a\1_{g^{-1}})
\]
is onto if, and only if, the classical trace of the globalized action, $Tr: B\rightarrow B^G$, given by 
\[
Tr (b)=\frac{1}{|G|}\sum_{g\in G} \beta_g (b)
\]
is onto (see \cite{DFP}, Corollary 2.2). Moreover, if the partial trace is onto, then there is a ring isomorphism between the invariant subalgebras $A^{\underline{G}}$ and $B^G$ (see \cite{DFP}, Proposition 2.3). The main difficulty about the categorical generalizations of these results is the fact that, even in the global case, any object in the category $\C $ is isomorphic only to a direct summand of its trace. Also, under the same conditions, in \cite{DE}, the authors proved that there is a Morita equivalence between the partial smash product $A\smp_\alpha G$ and the global smash product $B\rtimes_\beta G$. This result seems to be more straightforward to be formulated in the categorical case.

Yet, there are some remaining questions related to the partial equivariantization, $\C ^{\underline{G}}$. In fact, in Section 4, we showed that a partial action of $G$ on $\C $ defines a Hopf Polyad 
\begin{equation}\label{polyadfinal}
\left( \underline{G},\; \{ \C _g \}_{g\in G} ,\; \{ P_g =T_g \circ T_{g^{-1}}:\C _g \rightarrow \C _g \}_{g\in G},\; \{ \mu_g : P_g \circ P_g \Rightarrow P_g \}_{g\in G} ,\; \{ \eta_g : \Id_{\C _g}  \Rightarrow P_g \}_{g\in G}\right) .
\end{equation}
According to \cite{Bohm}, Hopf Polyads are nothing but Hopf monads in an appropriate monoidal category. Therefore, it is natural to ask whether the partial equivariantization can be seen as an Eilenberg-Moore category of a monad, following the ideas of \cite{MN} to construct an equivariant category associated to a monad acting on a monoidal category.

Finally, there are some universal constructions associated to partial representations of a group as the Exel groupoid $\Gamma (G)$ and the partial group algebra $\Bbbk_{par} G$ \cite{DEP}. The Hopf polyad (\ref{polyadfinal}) seems to be related to these multi object construction. Certainly, a deeper examination of this relationship would be welcome. Moreover, the partial group algebra $\Bbbk_{par}G$ is isomorphic to a partial crossed product $A_{par}(G) \rtimes G$. Then, the categorifications of these constructions are yet to be explored in further investigations.

\section{Acknowledgements}

The authors would like to thank the referee of this article for his/her many valuable suggestions which helped to improve the quality of the text, the precision of definitions and results, and made questions which opened several directions for further investigations.

\end{document}

%% file: diagramas/cat-mon-1.tikz
\begin{tikzpicture}[xscale=2]
\path
    ( 0, 0)         node (P0) {$((X\ot Y)\ot Z)\ot W$}
    ++(-1,-1)       node (P1) {$(X\ot Y)\ot(Z\ot W)$}
    ++( 0,-1)       node (P2) {$X\ot (Y\ot(Z\ot W))$}
    (P0)++(1,-1)    node (P4) {$(X\ot (Y\ot Z))\ot W$}
    ++( 0,-1)       node (P3) {$X\ot((Y\ot Z)\ot W)$}
[mor]
    (P0) edge node[left ] {$a_{X\ot Y,Z,W}$} (P1)
    (P1) edge node[left ] {$a_{X,Y,Z\ot W}$} (P2)
    (P3) edge node[below] {$\id_X\ot a_{Y,Z,W}$} (P2)
    (P4) edge node[right] {$a_{X,Y\ot Z,W}$} (P3)
    (P0) edge node[right] {$a_{X,Y,Z}\ot\id_W$} (P4)
;
\end{tikzpicture}

%% file: diagramas/cat-mon-2.tikz
\begin{tikzpicture}
\path
    ( 0,0)      node (P2) {$X\ot Y$}
    ++( 1,1)    node (P0) {$X\ot(\1\ot Y)$}
    ++(-2,0)    node (P1) {$(X\ot\1)\ot Y$}
[mor]
    (P1) edge node[above] {$a_{X,\1,Y}$}   (P0)
    (P1) edge node[left ] {$r_X\ot\id_Y$}  (P2)
    (P0) edge node[right] {$\id_X\ot l_Y$} (P2)
;
\end{tikzpicture}

%% file: diagramas/cat-subg-1.tikz
\begin{tikzpicture}[xscale=1.5]
\path ( 0, 0) node (P0) {$(F(X)\ot' F(Y))\ot' F(Z)$}
    ++(-1,-1) node (P1) {$F(X\ot Y)\ot' F(Z)$} 
    ++( 0,-1) node (P2) {$F((X\ot Y)\ot Z)$}
    ++( 1,-1) node (P3) {$F(X\ot (Y\ot Z))$}
    ++( 1, 1) node (P4) {$F(X)\ot' F(Y\ot Z)$}
    ++( 0, 1) node (P5) {$F(X)\ot' (F(Y)\ot' F(Z))$}
[mor]
    (P0) edge node[left =1ex] {$J_{X,Y}\ot'\id_{F(Z)}$} (P1)
    (P1) edge node[left     ] {$J_{X\ot Y,Z}$}          (P2)
    (P2) edge node[left =1ex] {$F(a_{X,Y,Z})$}          (P3)
    (P4) edge node[right=1ex] {$J_{X,Y\ot Z}$}          (P3)
    (P5) edge node[right    ] {$\id_{F(X)}\ot' J_{Y,Z}$}(P4)
    (P0) edge node[right    ] {$a_{F(X),F(Y),F(Z)}$}    (P5)
;
\end{tikzpicture}

%% file: diagramas/cat-subg-2.tikz
\begin{tikzpicture}
\path ( 0,0) node (P0) {$F(\1)\ot'F(X)$}
    ++( 2,0) node (P1) {$F(\1\ot X)$}
    ++( 0,1) node (P2) {$F(X)$}
    ++(-2,0) node (P3) {$\1'\ot'F(X)$}
[mor]
    (P0) edge node [below] {$J_{\1,X}$} (P1)
    (P1) edge node [right] {$F(l_X)$} (P2)
    (P3) edge node [above] {$l'_{F(X)}$} (P2)
    (P3) edge node [left ] {$J^0\ot'\id_{F(X)}$} (P0)
;
\end{tikzpicture}

%% file: diagramas/cat-subg-3.tikz
\begin{tikzpicture}
\path ( 0,0) node (P0) {$F(X)\ot'F(\1)$}
    ++( 2,0) node (P1) {$F(X\ot\1)$}
    ++( 0,1) node (P2) {$F(X)$}
    ++(-2,0) node (P3) {$F(X)\ot'\1'$}
[mor]
    (P0) edge node[below]{$J_{X,\1}$} (P1)
    (P1) edge node[right]{$F(r_X)$} (P2)
    (P3) edge node[above]{$r'_{F(X)}$} (P2)
    (P3) edge node[left]{$\id_{F(X)}\ot'J^0$} (P0)
;
\end{tikzpicture}

%% file: diagramas/mor-subg-1.tikz
\begin{tikzpicture}
\path ( 0,0) node (P0) {$\widetilde{F}(X)\ot'\widetilde{F}(Y)$}
    ++( 2,0) node (P1) {$\widetilde{F}(X\ot Y)$}
    ++( 0,1) node (P2) {$F(X\ot Y)$}
    ++(-2,0) node (P3) {$F(X)\ot'F(Y)$}
[mor]
    (P0) edge node [below] {$\widetilde{J}_{X,Y}$} (P1)
    (P2) edge node [right] {$\eta_{X\ot Y}$} (P1)
    (P3) edge node [above] {$J_{X,Y}$} (P2)
    (P3) edge node [left ] {$\eta_X\ot'\eta_Y$} (P0)
;
\end{tikzpicture}

%% file: diagramas/mor-subg-2.tikz
\begin{tikzpicture}
\path ( 0, 1) node (P0) {$\1'$}
    ++(-1,-1) node (P1) {$F(\1)$}
    ++( 2, 0) node (P2) {$\widetilde{F}(\1)$}
[mor]
    (P0) edge node[left ] {$J^0$}       (P1)
    (P1) edge node[below] {$\eta_\1$}   (P2)
    (P0) edge node[right] {$\widetilde{J}^0$}     (P2)
;
\end{tikzpicture}

%% file: diagramas/replete.tikz
\begin{tikzpicture}
\path ( 0,0) node (P0) {$X$}
    ++( 1,0) node (P1) {$Y$}
    ++( 0,1) node (P2) {$Y'$}
    ++(-1,0) node (P3) {$X'$}
[mor]
    (P0) edge node [below] {$\vf$} (P1)
    (P1) edge node [right] {$\nu$} (P2)
    (P3) edge node [above] {$\vf'$} (P2)
    (P0) edge node [left ] {$\mu$} (P3)
;
\end{tikzpicture}

%% file: diagramas/id-cent-1.tikz
\begin{tikzpicture}[xscale=2]
\path ( 0, 0) node (P1) {\(e \otimes e \otimes e\)}
    ++( 1, 0) node (P2) {\(e \otimes e\)}
    ++( 0, -1) node (P3) {\( e \)}
    ++( -1,0) node (P4) {\( e \otimes e\)}
[mor]
    (P1) edge node [above] {\(e \otimes \Phi_e \)} (P2)
    (P2) edge node [right] {\( \Phi_e\)} (P3)
    (P1) edge node [left ] {\(\Phi_e \otimes e \)} (P4)
    (P4) edge node [below] {\(\Phi_e \)} (P3)
;
\end{tikzpicture}

%% file: diagramas/id-cent-2.tikz
\begin{tikzpicture}
\path (0, 0) node (P1) {\(e \otimes e \)}
    ++(2, 0) node (P2) {\( e \otimes e \)}
    ++(-1,-1) node (P3) {\( e\)}
[mor]
    (P1) edge node [above ] {\(\sigma^e_e \)}   (P2)
    (P2) edge node [right] {\( \Phi_e\)}   (P3)
    (P1) edge node [left ] {\(\Phi_e \)} (P3)
;
\end{tikzpicture}

%% file: diagramas/id-cent-3.tikz
\begin{tikzpicture}[xscale=2]
\path ( 0, 0) node (P1) {\(e \otimes e \otimes A\)}
    ++( 1, 0) node (P2) {\(e \otimes A \otimes e\)}
    ++( 1, 0) node (P3) {\(A \otimes e \otimes e\)}
    ++( 0,-1) node (P4) {\(A \otimes e\)}
    ++(-2, 0) node (P5) {\(e \otimes A\)}
[mor]
    (P1) edge node [above] {\(e \otimes \sigma^e_A\)} (P2)
    (P2) edge node [above] {\(\sigma^e_A \otimes e\)} (P3)
    (P3) edge node [right] {\(A \otimes \Phi_e\)} (P4)
    (P1) edge node [left ] {\(\Phi_e \otimes A\)} (P5)
    (P5) edge node [below] {\(\sigma^e_A\)} (P4)
;
\end{tikzpicture}

%% file: diagramas/id-cent-4.tikz
\begin{tikzpicture}
\path (0, 0) node (P1) {\(e \otimes A \otimes B\)}
    ++(1, 1) node (P2) {\(A \otimes e \otimes B\)}
    ++(1,-1) node (P3) {\(A \otimes B \otimes e\)}
[mor]
    (P1) edge node [left ] {\(\sigma^e_A \otimes B\)}   (P2)
    (P2) edge node [right] {\(A \otimes \sigma^e_B\)}   (P3)
    (P1) edge node [below] {\(\sigma^e_{A \otimes B}\)} (P3)
;
\end{tikzpicture}

%% file: diagramas/teo-idemp-mon-1.tikz
\begin{tikzpicture}[xscale=1.25]
\path
    (0,0) node (1) {\(e\otimes e\otimes X''\)}
    (3,0) node (2) {\(e\otimes X''\)}
    (1,1) node (3) {\(e\otimes X\)}
    (2,1) node (4) {\(X\)}
    (0,2) node (5) {\(e\otimes e \otimes X'\)}
    (3,2) node (6) {\( e\otimes X'\)}
[mor]
    (1) edge
        node [below] {\(\Phi_e \otimes X''\)}
        (2)
    (2) edge
        node [left =4pt] {\((\psi_X)^{-1}\)}
        (4)
    (3) edge
        node [right] {\(e\otimes \psi_X\)}
        (1)
    (3) edge
        node [above] {\(L^e_X\)}
        (4)
    (3) edge
        node [right] {\(e\otimes \varphi_X\)}
        (5)
    (5) edge
        node [left ] {\(e\otimes \theta\)}
        (1)
    (5) edge
        node [above] {\(\Phi_e \otimes X'\)}
        (6)
    (6) edge
        node [right] {\(\theta\)}
        (2)
    (6) edge
        node [left ] {\((\varphi_X)^{-1}\)}
        (4)
;
\end{tikzpicture}

%% file: diagramas/teo-idemp-mon-2.tikz
\begin{tikzpicture}[xscale=1.75]
\path
    (0, 0) node (1) {\(X\)}
    (1, 0) node (2) {\(Y\)}
    (0,-1) node (3) {\(e \otimes X'\)}
    (1,-1) node (4) {\(e \otimes Y'\)}
[mor]
    (1) edge
        node [above] {\(f\)}
        (2)
    (1) edge
        node [left ] {\(\varphi_X\)}
        (3)
    (2) edge
        node [right] {\(\varphi_Y\)}
        (4)
    (3) edge
        node [below] {\(e \otimes f'\)}
        (4)
;
\end{tikzpicture}

%% file: diagramas/rem-acao-parcial-1.tikz
\begin{tikzpicture}[xscale=2]
\path ( 0,0) node (P0) {$T_{gh}T_k(X)$}
    ++( 1,0) node (P1) {$T_{ghk}(X)$}
    ++( 0,1) node (P2) {$T_gT_{hk}(X)$}
    ++(-1,0) node (P3) {$T_gT_hT_k(X)$}
[mor]
    (P0) edge node [below] {$(\g_{gh,k})_X$} (P1)
    (P2) edge node [right] {$(\g_{g,hk})_X$} (P1)
    (P3) edge node [above] {$T_g((\g_{h,k})_X)$} (P2)
    (P3) edge node [left ] {$(\g_{g,h})_{T_k(X)}$} (P0)
;
\end{tikzpicture}

%% file: diagramas/rem-acao-parcial-2.tikz
\begin{tikzpicture}[xscale=1.5]
\path (0,0) node (P0) {$T_g(X)$}
    ++(1,0) node (P1) {$T_eT_g(X)$}
[mor]
    ([yshift= .5ex]P0.east) edge node [above] {$u_{T_g(X)}$}   ([yshift= .5ex]P1.west)
    ([yshift=-.5ex]P1.west) edge node [below] {$(\g_{e,g})_X$} ([yshift=-.5ex]P0.east)
;
\end{tikzpicture}

%% file: diagramas/rem-acao-parcial-3.tikz
\begin{tikzpicture}[xscale=1.5]
\path (0,0) node (P0) {$T_g(X)$}
    ++(1,0) node (P1) {$T_gT_e(X)$}
[mor]
    ([yshift= .5ex]P0.east) edge node [above] {$T_g(u_X)$}     ([yshift= .5ex]P1.west)
    ([yshift=-.5ex]P1.west) edge node [below] {$(\g_{g,e})_X$} ([yshift=-.5ex]P0.east)
;
\end{tikzpicture}

%% file: diagramas/lem-acao-parcial-1.tikz
\begin{tikzpicture}[xscale=1.5]
\path
    (0,0) node (P0) {$X$}
    (1,0) node (P1) {$Y$}
    (0,1) node (P2) {$T_e(X)$}
    (1,1) node (P3) {$T_e(Y)$}
    (0,2) node (P4) {$X'$}
    (1,2) node (P5) {$Y'$}
[mor]
    (P0) edge node [above] {$f$}      (P1)
    (P0) edge node [left ] {$u_X$}    (P2)
    (P1) edge node [right] {$u_Y$}    (P3)
    (P2) edge node [above] {$T_e(f)$} (P3)
    (P2) edge node [left ] {$\mu$}    (P4)
    (P3) edge node [right] {$\nu$}    (P5)
    (P4) edge node [above] {$f'$}     (P5)
;
\end{tikzpicture}

%% file: diagramas/lem-acao-parcial-2.tikz
\begin{tikzpicture}
\path
    node (P0) at (0,0) {$T_h(X)$}
    node (P1) at (2,0) {$T_h(Y)$}
    node (P2) at (0,1) {$X'$}
    node (P3) at (2,1) {$Y'$}
[mor]
    (P0) edge node [below] {$T_h(f)$} (P1)
    (P0) edge node [left ] {$\mu$} (P2)
    (P1) edge node [right] {$\nu$} (P3)
    (P2) edge node [above] {$f'$} (P3)
;
\end{tikzpicture}

%% file: diagramas/lem-acao-parcial-3.tikz
\begin{tikzpicture}[xscale=1.3]
\path
    node (P0) at (0,0) {$T_{gh}(X)$}
    node (P1) at (2,0) {$T_{gh}(Y)$}
    node (P2) at (0,1) {$T_gT_h(X)$}
    node (P3) at (2,1) {$T_gT_h(Y)$}
    node (P4) at (0,2) {$T_g(X')$}
    node (P5) at (2,2) {$T_g(Y')$}
[mor]
    (P0) edge node [above] {$T_{gh}(f)$}      (P1)
    (P0) edge node [left ] {$(\g\m_{g,h})_X$} (P2)
    (P1) edge node [right] {$(\g\m_{g,h})_Y$} (P3)
    (P2) edge node [above] {$T_gT_h(f)$}      (P3)
    (P2) edge node [left ] {$T_g(\mu)$}       (P4)
    (P3) edge node [right] {$T_g(\nu)$}       (P5)
    (P4) edge node [above] {$T_g(f')$}        (P5)
;
\end{tikzpicture}

%% file: diagramas/def-acao-parcial-1.tikz
\begin{tikzpicture}[xscale=1.5]
\path
    node (P1) at ( 0,2) {$T_g T_h(X) \otimes T_g T_h(Y)$}
    node (P2) at (-1,1) {$T_g (T_h(X) \otimes T_h(Y))$}
    node (P3) at (-1,0) {$T_g T_h(X \otimes Y)$}
    node (P4) at ( 1,0) {$T_{gh}(X \otimes Y)$.}
    node (P5) at ( 1,1) {$T_{gh}(X) \otimes T_{gh}(Y)$}
[mor]
    (P1) edge node [left ] {$J^g_{T_h(X), T_h(Y)}$} (P2)
    (P2) edge node [left ] {$T_g J^h_{X, Y}$} (P3)
    (P3) edge node [below] {$(\g_{g,h})_{X \otimes Y}$} (P4)
    (P1) edge node [right] {$(\gamma_{g,h})_X \otimes (\gamma_{g,h})_Y$} (P5)
    (P5) edge node [right] {$J^{gh}_{X, Y}$} (P4)
;
\end{tikzpicture}

%% file: diagramas/def-acao-parcial-2.tikz
\begin{tikzpicture}
\path
    node (P1) at (-1,1) {$X \otimes Y$}
    node (P2) at ( 0,0) {$T_e (X \otimes Y)$.}
    node (P3) at ( 1,1) {$T_e(X) \otimes T_e(Y)$}
[mor]
    (P1) edge node[left ] {$u_{X \otimes Y}$} (P2)
    (P3) edge node[right] {$J^e_{X, Y}$}      (P2)
    (P1) edge node[above] {$u_X \otimes u_Y$} (P3)
;
\end{tikzpicture}

%% file: diagramas/exe-acao-parcial.tikz
\begin{tikzpicture}[xscale=2.5]
\path
    node (P1) at (0,1) {\(T_g(X)\)}
    node (P2) at (1,1) {\(\1_g \otimes T_g( X )\)}
    node (P3) at (0,0) {\(T_g( \1_{g^{-1}} \otimes X )\)}
    node (P4) at (1,0) {{\(T_g( \1_{g^{-1}} ) \otimes T_g( X )\) .}}
[mor]
    (P1)
         edge node [above] {\Iso} (P2)
         edge node [left ] {\Iso} (P3)
    (P2) edge node [right] {\(\varphi^g \otimes T_g ( X )\)} (P4)
    (P3) edge node [above] {\( (J^g)^{-1}_{\1_{g^{-1}},X}\)} (P4)
;
\end{tikzpicture}

%% file: diagramas/lem-acao-parcial-idemp.tikz
\begin{tikzpicture}[xscale=3]
\path
    (0, 0) node (1) {\( \1_g \otimes \1_{gh} \)}
    (1, 0) node (2) {\( T_{gh}(\1_{h^{-1}} \otimes \1_{(gh)^{-1}}) \)}
    (0,-1) node (3) {\( T_g (\1_{g^{-1}} \otimes \1_h ) \)}
    (1,-1) node (4) {\( T_g(T_h(\1_{(gh)^{-1}} \otimes \1_{h^{-1}} )) \).}
[mor]
    (1) edge
        node [above] {\(\varphi^{gh}_g\)}
        (2)
    (1) edge
        node [left ] {\(\varphi^g_{gh}\)}
        (3)
    (3) edge
        node [below] {\(T_g \left(\varphi^h_{g^{-1}}\right)\)}
        (4)
    (4) edge
        node [right] {\((\gamma_{g,h})_{\1_{h^{-1}}\otimes \1_{(gh)^{-1}}}\)}
        (2)
;
\end{tikzpicture}

%% file: diagramas/mor-acao-parcial-1.tikz
\begin{tikzpicture}
\path
    node (P0) at ( 0, 0) {$FT_gT_h(X)$}
    node (P1) at (-1,-1) {$FT_{gh}(X)$}
    node (P2) at (-1,-2) {$T'_{gh}F(X)$}
    node (P3) at ( 1,-2) {$T'_gT'_hF(X)$}
    node (P4) at ( 1,-1) {$T'_gFT_h(X)$}
[mor]
    (P0) edge node[left] {$F((\g_{g,h})_X)$} (P1)
    (P1) edge node[left] {$(\tau_{gh})_X$} (P2)
    (P3) edge node[above] {$(\g'_{g,h})_{F(X)}$} (P2)
    (P4) edge node[right] {$T'_g((\tau_h)_X)$} (P3)
    (P0) edge node[right] {$(\tau_g)_{T_h(X)}$} (P4)
;
\end{tikzpicture}

%% file: diagramas/mor-acao-parcial-2.tikz
\begin{tikzpicture}
\path
    node (P0) at (0,0) {$FT_g (X) \otimes'FT_g (Y)$}
    node (P1) at (3,0) {$F(T_g (X) \otimes T_g (Y))$}
    node (P2) at (6,0) {$FT_g(X\otimes Y)$}
    node (P3) at ( 0,-1) {$T'_g F(X) \otimes 'T'_gF(Y)$}
    node (P4) at ( 3,-1) {$T'_g (F(X) \otimes'F(Y))$}
    node (P5) at (6,-1) {$T'_gF(X\otimes Y)$}
[mor]
    (P0) edge node[above] {$J^F_{T_g(X),T_gY(Y)}$} (P1)
    (P0) edge node[left]
    {$(\tau_g)_X\otimes'(\tau_g)_Y$} (P3)
    (P1) edge node[above] {$F(J^g_{X,Y})$} (P2)
    (P2) edge node[right]
    {$(\tau_g)_{X\otimes Y}$} (P5)
    (P3) edge node[above] {$J'^g_{F(X),F(Y)}$} (P4)
    (P4) edge node[above] {$T'_g(J^F_{X,Y})$} (P5)
;
\end{tikzpicture}

%% file: diagramas/pro-mor-acao-parcial-1.tikz
\begin{tikzpicture}
\path
    node (P0) at (1,1) {$FT_e(X)$}
    node (P1) at (-1,1) {$F(X)$}
    node (P2) at (0,0) {$T'_eF(X)$}
[mor]
    (P1) edge node[above] {$F(u_X)$} (P0)
    (P1) edge node[left ] {$u'_{F(X)}$} (P2)
    (P0) edge node[right] {$(\tau_e)_X$} (P2)
;
\end{tikzpicture}

%% file: diagramas/pro-mor-acao-parcial-2.tikz
\begin{tikzpicture}[xscale=2.25]
\path
    node (1) at (1, 2) {\(F T_e ( X )\)}
    node (2) at (3, 2) {\(T'_e F ( X )\)}
    node (3) at (1, 1) {\(F T_e T_e ( X )\)}
    node (4) at (2, 1) {\(T'_e F T_e ( X )\)}
    node (5) at (3, 1) {\(T'_e T'_e F ( X )\)}
    node (6) at (1, 0) {\(FT_e ( X )\)}
    node (7) at (3, 0) {\(T'_e F ( X )\)}
[mor]
    (1)
        edge node [above] {\((\tau_e)_X\)} (2)
        edge node [right] {\(F T_e(u_X)\)} (3)
    (2)
        edge node [left ] {\(T'_e F(u_X)\)} (4)
        edge node [right] {\(T'_e u'_{F(X)}\)} (5)
    (3)
        edge node [below] {\((\tau_e)_{T_e(X)}\)} (4)
        edge node [right] {\(F((\gamma_{e,e})_X)\)} (6)
    (4) edge node [below] {\(T'_e(\tau_X)\)} (5)
    (5) edge node [left ] {\((\gamma'_{e,e})_{F(X)}\)} (7)
    (6) edge node [below] {\((\tau_e)_X\)} (7)
;
\draw [mor,igual]
    (1.west) arc (90:270:1cm and 2cm)
        coordinate [pos=.5] (_1)
    (2.east) arc (90:-90:1cm and 2cm)
        coordinate [pos=.5] (_2)
;
\path
    (1)  -- node [xshift={18pt}] {(I)} (4)
    (3)  -- node {(II)} (_1)
    (3)  -- node {(III)} (7)
    (5)  -- node {(IV)} (_2)
    (5)  -- coordinate (_1) (2) -- coordinate (_2) (4)
    (_1) -- node {(V)} (_2)
;
\end{tikzpicture}

%% file: diagramas/pro-mor-acao-parcial-3.tikz
\begin{tikzpicture}[xscale=1.5]
\path
    node (1) at (0,2) {\( T'_{e} F T_{e}  (X) \)}
    node (2) at (0,0) {\( T'_{e} T'_{e} F (X) \)}
    node (3) at (1,1) {\( T'_{e} F (X) \)}
    node (4) at (2,1) {\( F (X) \)}
    node (5) at (3,2) {\( F T_{e} (X) \)}
    node (6) at (3,0) {\( T'_{e} F (X) \)}
    (3) -- node {(I)}   (3 -| 1)
    (1.center) -- node {(II)} (4 -| 5)
    (2.center) -- node {(III)} (4 -| 5)
    (4) -- node {(IV)}  (4 -| 5)
[mor]
    (1) edge node [left ] {\( T'_{e}\)} (2)
    (3)
        edge node [right] {\( T'_{e} F u_{X} \)}   (1)
        edge node [right] {\( T'_{e} u'_{F(X)} \)} (2)
    (4)
        edge node [below] { \( u'_{F(X)} \)} (3)
        edge node [left ] { \( F u_{X} \)}   (5)
        edge node [left ] { \( u'_{F(X)} \)} (6)
    (5)
        edge node [above] { \( u'_{FT_{e}(X)} \)} (1)
        edge node [right] { \( (\tau_{e})_{X} \)} (6)
    (6) edge node [below] { \( u'_{T'_{e} F(X)} \) }(2)
;
\end{tikzpicture}

%% file: diagramas/pro-monad-1.tikz
\begin{tikzpicture}[xscale=3.0]
\path
    node (1)  at (0, 0) {\( T_{g} T_{g^{-1}}T_{g} T_{g^{-1}}T_{g} T_{g^{-1}} \)}
    node (2)  at (1, 0) {\( T_{g} T_{g^{-1}}T_{g}T_{e}  T_{g^{-1}} \)}
    node (3)  at (2, 0) {\( T_{g} T_{g^{-1}} T_{g} T_{g^{-1}} \)}
    node (4)  at (0, -1) {\( T_{g} T_e T_{g^{-1}} T_{g} T_{g^{-1}} \)}
    node (5)  at (1,-1) {\( T_{g} T_e T_e T_{g^{-1}} \)}
    node (6)  at (2,-1) {\( T_{g} T_e T_{g^{-1}} \)}
    node (7)  at (0,-2) {\( T_{g} T_{g^{-1}} T_{g} T_{g^{-1}} \)}
    node (8)  at (1,-2) {\( T_{g} T_e T_{g^{-1}} \)}
    node (9)  at (2,-2) {\( T_{g}  T_{g^{-1}} \)}
[nat]
    (1) edge node [above] {\(  T_{g} T_{g^{-1}} T_g \gamma_{g^{-1} g} T_{g^{-1}} \)} (2)
    (2) edge node [above] {\(  T_{g} T_{g^{-1}} T_g \gamma_{e,g^{-1}} \)} (3)
    (1) edge node [left ] {\( T_g \gamma_{g^{-1} ,g} T_{g^{-1}} T_g T_{g^{-1}} \)} (4)
    (2) edge node [left ] {\( T_g \gamma_{g^{-1}, g} T_e T_{g^{-1}} \)} (5)
    (3) edge node [left ] {\( T_{g} \gamma_{g^{-1}, g} T_{g^{-1}} \)} (6)
    (4) edge node [below] {\( T_g T_e \gamma_{ g^{-1},g} T_{g^{-1}} \)} (5)
    (5) edge node [below] {\( T_g T_e \gamma_{e,g^{-1}}  \)} (6)
    (4) edge node [left ] {\( T_{g} \gamma_{e,g^{-1}} T_{g^{-1}} T_g T_{g^{-1}} \)} (7)
    (5) edge node [left ] {\( T_{g} T_e \gamma_{e,g^{-1}}  \)} (8)
    (6) edge node [left ] {\( T_{g}  \gamma_{e,g^{-1}}  \)} (9)
    (7) edge node [below] {\( T_g  \gamma_{ g^{-1},g} T_{g^{-1}} \)} (8)
    (8) edge node [below] {\( T_g  \gamma_{e,g^{-1}}  \)} (9)
;
\end{tikzpicture}

%% file: diagramas/pro-monad-2.tikz
\begin{tikzpicture}[xscale=2.25]
\path
    node (1)  at (0, 0) {\( T_{g} T_{g^{-1}} \)}
    node (2)  at (1, 0) {\( T_{e} T_{g} T_{g^{-1}} \)}
    node (3)  at (2, 0) {\( T_{g g^{-1}} T_{g} T_{g^{-1}} \)}
    node (4)  at (3, 0) {\( T_{g} T_{g^{-1}} T_{g} T_{g^{-1}} \)}
    node (5)  at (1,-1) {\( T_{g} T_{g^{-1}} \)}
    node (6)  at (2,-1) {\( T_{g g^{-1} g} T_{g^{-1}} \)}
    node (7)  at (3,-1) {\( T_{g} T_{g^{-1} g} T_{g^{-1}} \)}
    node (8)  at (2,-2) {\( T_{g} T_{g^{-1}} \)}
    node (9)  at (3,-2) {\( T_{g} T_{e} T_{g^{-1}} \)}
    node (10) at (3,-3) {\( T_{g} T_{g^{-1}} \)}
[nat]
    (1) edge node [above] {\( u T_{g} T_{g^{-1}} \)} (2)
    (3) edge node [above] {\( \gamma_{g, g^{-1}}^{-1} T_{g} T_{g^{-1}} \)} (4)
    (2) edge node [left ] {\( \gamma_{e, g} T_{g^{-1}} \)} (5)
    (3) edge node [left ] {\( \gamma_{gg^{-1}, g} T_{g^{-1}} \)} (6)
    (4) edge node [left ] {\( T_{g} \gamma_{g^{-1}, g} T_{g^{-1}} \)} (7)
    (6) edge node [below] {\( \gamma_{g, g^{-1}g}^{-1} T_{g^{-1}} \)} (7)
    (8) edge node [above] {\( \gamma_{g, e}^{-1} T_{g^{-1}} \)} (9)
    (9) edge node [left ] {\( T_{g} u^{-1} T_{g^{-1}} \)} (10)
    (1.north) edge [bend left] node [above] {\( \eta_{g} T_g T_{g^{-1}} \)} (4.north)
    (4) edge [bend left=20] node [right] {\( \mu_{g} \)} (10)
[igual]
    (1) edge (5)
    (2) edge (3)
    (5) edge (6)
        edge (8)
    (6) edge (8)
    (7) edge (9)
    (8) edge (10)
;
\end{tikzpicture}

%% file: diagramas/teo-glob-acao-parcial-1.tikz
\begin{tikzpicture}[xscale = 1.5]
\path
    (0,0) node (1) {\(\mathcal{T}_{g} \mathcal{T}_{h} \Phi(X) (k)\)}
    (2,0) node (2) {\(\mathcal{T}_{gh} \Phi(X) (k)\)}
    (0,1) node (3) {\(\mathcal{T}_{g} \Phi(T_{h}(X)) (k)\)}
    (2,1) node (4) {\(\Phi(T_{gh}(X)) (k)\)}
    (1,2) node (5) {\(\Phi(T_{g} T_{h}(X)) (k)\)}
[mor]
    (5) edge node [left =1em] {\(\left( \left( \tau_g \right)_{T_{h}(X)} \right)_{k}\)} (3)
    (5) edge node [right] {\(\Phi \left( \left(\gamma_{g,h} \right)_{X} \right)_{k}\)} (4)
    (3) edge node [left ] {\( \left( \left( \mathcal{T}_{g} \tau_{h} \right)_{X} \right)_{k} \)} (1)
    (4) edge node [right] {\( \left( \left( \tau_{gh} \right)_{X} \right)_{k} \)} (2)
    (1) edge node [below] {\( \left( \Gamma_{g,h} \Phi (X) \right)_{k} \)} (2)
;
\end{tikzpicture}

%% file: diagramas/teo-glob-acao-parcial-2.tikz
\begin{tikzpicture}[xscale = 3.5]
\path
    (0,0) node (1) {\( T_{kg}(T_h (X) \otimes \1_{(kg)^{-1}}) \)}
    (1,0) node (2) {\( T_{kgh} (X\otimes \1_{(kgh)^{-1}}) \)}
    (0,1) node (3) {\( T_k(T_g T_h (X) \otimes \1_{k^{-1}}) \)}
    (1,1) node (4) {\( T_k(T_{gh}(X) \otimes \1_{k^{-1}}) \)}
[mor]
    (1) edge node [below] {\( ((\tau_h)_X)_{kg} \)} (2)
    (3) edge node [left ] {\( ((\tau_g)_{T_h(X)})_k \)} (1)
        edge node [above] {\( T_k((\gamma_{g,h})_X \otimes \1_{k^{-1}}) \)} (4)
    (4) edge node [right] {\( ((\tau_{gh})_X)_k \)} (2)
;
\end{tikzpicture}

%% file: diagramas/def-equivar-global-1.tikz
\begin{tikzpicture}[xscale=1.5]
\path
    (0,0) node (P0) {\(T_{gh}(X)\)}
    (1,0) node (P1) {\( X\)}
    (1,1) node (P2) {\( T_g(X) \)}
    (0,1) node (P3) {\(T_g T_h (X)\)}
[mor]
    (P3) edge node [above] {\(T_g(\theta_h)\) } (P2)
    (P3) edge node [left ] {\((\g_{g, h})_X\)} (P0)
    (P0) edge node [above] {\(\theta_{gh}\) } (P1)
    (P2) edge node [right] {\(\theta_g\)} (P1)
;
\end{tikzpicture}

%% file: diagramas/def-equivar-global-2.tikz
\begin{tikzpicture}[xscale=1.5]
\path
    (1,0) node (P1) {\( X\)}
    (1,1) node (P2) {\( T_e(X) \)}
    (0,1) node (P3) {\( X\)}
[mor]
    (P3) edge node [above] {\( u_X \) } (P2)
    (P2) edge node [right] {\(\theta_e\)} (P1)
[igual]
    (P3) edge (P1)
;
\end{tikzpicture}

%% file: diagramas/def-equivar-global-3.tikz
\begin{tikzpicture}[xscale=1.5]
\path
    (0,0) node (P0) {\(T_g (X')\)}
    (1,0) node (P1) {\(X'\)}
    (1,1) node (P2) {\(X\)}
    (0,1) node (P3) {\(T_g(X)\)}
[mor]
    (P3) edge node [above] {\(\theta_g\)} (P2)
    (P2) edge node [right] {\(\varphi\)} (P1)
    (P3) edge node [left ] {\(T_g(\varphi)\)} (P0)
    (P0) edge node [above] {\(\theta'_g\)} (P1)
;
\end{tikzpicture}

%% file: diagramas/def-equivar-parcial-1.tikz
\begin{tikzpicture}[xscale=1.5]
\path
    ( 0, 0) node (P1) {$T_gT_h(X \otimes Y)$}
    (-1,-1) node (P2) {$T_g(X \otimes  T_h(Y))$}
    (-1,-2) node (P3) {$X \otimes T_gT_h(Y)$}
    ( 1,-2) node (P4) {$X \otimes T_{gh}(Y)$}
    ( 1,-1) node (P5) {$T_{gh}(X \otimes Y)$}
[mor]
    (P1) edge node[left ,xshift=-1ex] {$T_g\left(\sigma_h^{X}\right)_Y$}        (P2)
    (P2) edge node[left             ] {$\left(\sigma_{g}^{X}\right)_{T_h(Y)}$}  (P3)
    (P3) edge node[below            ] {$X \otimes \left(\gamma_{g,h}\right)_Y$} (P4)
    (P1) edge node[right,xshift= 1ex] {$(\gamma_{g,h})_{X \otimes Y}$}          (P5)
    (P5) edge node[right            ] {$\left(\sigma_{gh}^{X}\right)_Y$}        (P4)
;
\end{tikzpicture}

%% file: diagramas/def-equivar-parcial-2.tikz
\begin{tikzpicture}
\path
    node (A) at (0, 0) {$X \otimes Y$}
    node (B) at (2, 0) {$T_e(X \otimes Y)$}
    node (C) at (1,-1) {$X \otimes T_e(Y)$}
[mor]
    (A) edge node [left ] {$X \otimes u_Y$} (C)
    (A) edge node [above] {$u_{X \otimes Y}$} (B)
    (B) edge node [right] {$(\sigma_e^X)_Y$}(C)
;
\end{tikzpicture}

%% file: diagramas/def-equivar-parcial-3.tikz
\begin{tikzpicture}[xscale = 2.0]
\path
    node (A) at (0, 0) {$T_g (X\otimes Y) \otimes T_g (Z)$}
    node (B) at (2, 0) {$X \otimes T_g (Y) \otimes T_g (Z)$}
    node (C) at (0, -1) {$T_g (X\otimes Y \otimes Z)$}
    node (D) at (2, -1) {$X \otimes T_g (Y \otimes Z)$}
[mor]
    (A) edge node [above] {$(\sigma^X_g)_Y \otimes T_g (Z)$} (B)
    (A) edge node [left] {$J^g_{X\otimes Y , Z}$} (C)
    (B) edge node [right] {$X\otimes J^g_{Y,Z}$} (D)
    (C) edge node [below] {$(\sigma^X_g)_{Y\otimes Z}$} (D)
;
\end{tikzpicture} 

%% file: diagramas/def-equivar-parcial-4.tikz
\begin{tikzpicture}[xscale=2]
\path
    node (P1) at (0, 0) {$T_g(X \otimes \esp)$}
    node (P2) at (0,-1) {$T_g(Y \otimes \esp)$}
    node (P3) at (1,-1) {$Y \otimes T_g(\esp)$}
    node (P4) at (1, 0) {$X \otimes T_g(\esp)$}
[nat]
    (P1)
         edge node [left ] {\( T_g (\varphi \otimes \esp ) \)} (P2)
         edge node [above] {\( \sigma_g^X \)} (P4)
    (P2) edge node [below] {\( \sigma_g^Y \)} (P3)
    (P4) edge node [right] {\( \varphi \otimes T_g ( \esp ) \)} (P3)
;
\end{tikzpicture}

%% file: diagramas/lem-equivar-idemp-1.tikz
\begin{tikzpicture}[xscale=3.5]
\path
    node (P0) at (0, 0) {$T_g(T_h(X \otimes \1_{h^{-1}} \otimes \1_{(gh)^{-1}}))$}
    node (P1) at (0,-1) {$T_{gh}(X \otimes \1_{h^{-1}} \otimes \1_{(gh)^{-1}})$}
    node (P2) at (1,-1) {$X  \otimes \1_{g} \otimes \1_{gh}$}
    node (P3) at (1,0) {$T_g(X  \otimes \1_{h} \otimes \1_{g^{-1}})$}
[mor]
    (P0) edge node [left ] {$(\gamma_{g,h})_{X\otimes \1_{h^{-1}} \otimes \1_{(gh)^{-1}}}$} (P1)
    (P1) edge node [below] {$\widetilde{\sigma}^X_{gh,h^{-1}}$} (P2)
    (P0) edge node [above] {$T_g(\widetilde{\sigma}_{h, (gh)^{-1}}^X)$} (P3)
    (P3) edge node [right] {$\widetilde{\sigma}^X_{g,h}$} (P2)
;
\end{tikzpicture}

%% file: diagramas/lem-equivar-idemp-2.tikz
\begin{tikzpicture}[xscale = 3.5]
\path
    node (00) at (0, 0) {\( \Tgh ( X \otimes \uhi \otimes \ughi ) \)}
    node (10) at (1, 0) {\( X \otimes \Tgh ( \uhi \otimes \ughi ) \)}
    node (20) at (2, 0) {\( X \otimes \ug \otimes \ugh \)}
    node (11) at (1, 1) {\( X \otimes \Tg ( \Th ( \uhi \otimes \ughi ) ) \)}
    node (21) at (2, 1) {\( X \otimes \Tg ( \uh \otimes \ugi ) \)}
    node (02) at (0, 2) {\( \Tg ( \Th ( X \otimes \uhi \otimes \ughi ) ) \)}
    node (12) at (1, 2) {\( \Tg ( X \otimes \Th ( \uhi \otimes \ughi ) ) \)}
    node (22) at (2, 2) {\( \Tg ( X \otimes \uh \otimes \ugi ) \)}
    (02) -- node [left ] {(I)  } (10)
    (11) -- node         {(II) } (20)
    (12) -- node [above] {(III)} (21)
[mor]
    (02) edge node [above] {\( \Tg \left( \middle( \sigma^{X}_{\Eh} \middle)_{\Eghi} \right) \)} (12)
    (12) edge node [above] {\( \Tg ( X \otimes (\varphi^{\Eh}_{\Egi})^{-1} ) \)} (22)
    (02) edge node [left ] {\( (\gamma_{g, h})_{X \otimes \uhi \otimes \ughi} \)} (00)
    (12) edge node [left ] {\( \left( \sigma^{X}_{\Eg} \right)_{\uhi \otimes \ughi} \)} (11)
    (22) edge node [left] {\( \left( \sigma^{X}_{\Eg} \right)_{\uhi \otimes \ugi} \)} (21)
    (11) edge node [above] {\( X \otimes \Tg ( (\varphi^{\Eh}_{\Egi})^{-1} ) \)} (21)
    (11) edge node [left ] {\( X \otimes (\gamma_{g,h})_{\uhi \otimes \ughi } \)} (10)
    (21) edge node [left] {\( X \otimes (\varphi^{\Eg}_{\Egh})^{-1} \)} (20)
    (00) edge node [below] {\( \left( \sigma^{X}_{\Egh} \right)_{\uhi \otimes \ughi} \)} (10)
    (10) edge node [below] {\( X \otimes (\varphi^{\Egh}_{\Eg})^{-1} \)} (20)
;
\end{tikzpicture}

%% file: diagramas/teo-iso-CG-1.tikz
\begin{tikzpicture}[xscale=3]
\path
    (1,2) node (12) {\( T_{g} T_{h} (X \otimes \1) \)}
        (12.west) node [left = -4pt] (11) {\( T_{g} T_{h} (X) \Iso \)}
    (2,2) node (13) {\( T_{g} (X \otimes T_{h} (\1)) \)}
    (3,2) node (14) {\( T_{g} (X \otimes \1)\)}
        (14.east) node [right = -4pt] {\( \Iso T_{g}(X)\)}
    (2,1) node (23) {\( X \otimes T_{g} T_{h} (\1) \)}
    (3,1) node (24) {\( X \otimes T_{g} (\1) \)}
    (1,0) node (32) {\( T_{gh} (X \otimes \1) \)}
        (32.west) node [left = -4pt] (31) {\( T_{gh} (X) \Iso \)}
    (2,0) node (33) {\( X \otimes T_{gh} (\1) \)}
    (3,0) node (34) {\( X \otimes \1\)}
        (34.east) node [right = -4pt] {\( \Iso X\)}
[mor]
    (11) edge node [left ] {\( (\gamma_{g,h})_{X} \)} (11 |- 32.north)
    (12) edge node [left ] {\( (\gamma_{g,h})_{X \otimes \1} \)} (32)
    (13) edge node [left ] {\( (\sigma_{g}^{X})_{T_{h}(\1)} \)} (23)
    (14) edge node [right] {\( (\sigma_{g}^{X})_{\1} \)} (24)
    (23) edge node [left ] {\( X \otimes (\gamma_{g,h})_{\1} \)} (33)
    (24) edge node [right] {\( X \otimes (J^{g0})^{-1} \)} (34)
    (12) edge node [above] {\( T_{g} ((\sigma_{h}^{X})_{\1}) \)} (13)
    (13) edge node [above] {\( T_{g} ( X \otimes (J^{h0})^{-1}) \)} (14)
    (23) edge node [above] {\( X \otimes T_{g}((J^{h0})^{-1}) \)} (24)
    (32) edge node [below] {\( (\sigma_{gh}^{X})_{\1} \)} (33)
    (33) edge node [below] {\( X \otimes ( J^{gh0} )^{-1} \)} (34)
;
\path
    (12) -- node {(I)} (33)
    (13) -- node {(II)} (24)
    (23) -- node {(III)} (34)
;
\end{tikzpicture}

%% file: diagramas/teo-iso-CG-2.tikz
\begin{tikzpicture}[xscale=2]
\path
    (2,1) node (12) {\( X \)}
    (3,1) node (13) {\( X \otimes \1 \)}
    (1,0) node (21) {\( T_{e} (X) \)}
    (2,0) node (22) {\( T_{e} (X \otimes \1) \)}
    (3,0) node (23) {\( X \otimes T_{e} (\1) \)}
    (4,0) node (24) {\( X \otimes \1 \)}
[mor]
    (12) edge node [above] {\( \Iso \)} (13)
    (21) edge node [below] {\( \Iso \)} (22)
    (22) edge node [below] {\( (\sigma^{X}_{e})_{\1} \)} (23)
    (23) edge node [below] {\( X \otimes (J^{e0})^{-1} \)} (24)
    (12) edge node [left =1ex] {\( u_{X} \)} (21)
    (13) edge node [left =1ex] {\( u_{X \otimes \1} \)} (22)
    (13) edge node [right] {\( X \otimes u_{\1} \)} (23)
[igual]
    (13) edge [bend left] (24)
;
\path
    (12) -- node {(I)} (22)
    (22) -- coordinate (_) (23)
    (_) -- node {(II)} (13)
    (13) -- node  {(III)} (24)
;
\end{tikzpicture}

%% file: diagramas/teo-iso-CG-3.tikz
\begin{tikzpicture}[xscale=2.0]
\path
    node (P0) at (0,0) {$T_g (X)$}
    node (P1) at (1,0) {$T_g (X\otimes \1)$}
    node (P2) at (2,0) {$X\otimes T_g (\1)$}
    node (P3) at ( 3,0) {$X\otimes \1$}
    node (P4) at ( 4,0) {$X$}
    node (P5) at (0,-1) {$T_g (Y)$}
    node (P6) at (1,-1) {$T_g (Y\otimes \1)$}
    node (P7) at (2,-1) {$Y\otimes T_g (\1)$}
    node (P8) at ( 3,-1) {$Y\otimes \1$}
    node (P9) at ( 4,-1) {$Y$}
[mor]
    (P0) edge node[above] {$T_g (r_X^{-1})$} (P1)
    (P1) edge node[above] {$(\sigma^X_g)_{\1}$} (P2)
    (P2) edge node[above] {$X\otimes (J^{g0})^{-1}$} (P3)
    (P3) edge node[above] {$r_X$} (P4)
    (P5) edge node[below] {$T_g (r_Y^{-1})$} (P6)
    (P6) edge node[below] {$(\sigma^Y_g)_{\1}$} (P7)
    (P7) edge node[below] {$Y\otimes (J^{g0})^{-1}$} (P8)
    (P8) edge node[above] {$r_Y$} (P9)
    (P0) edge node[left]
    {$T_g (\varphi)$} (P5)
    (P1) edge node[left]
    {$T_g (\varphi \otimes \1)$} (P6)
    (P2) edge node[left]
    {$\varphi \otimes T_g( \1)$} (P7)
    (P3) edge node[left]
    {$\varphi \otimes \1$} (P8)
    (P4) edge node[left]
    {$\varphi $} (P9)
;
\end{tikzpicture}

%% file: diagramas/teo-iso-CG-4.tikz
\begin{tikzpicture}[xscale=2]
\path
    (0,0) node (1) {\((X, \sigma^X )\)}
    (1,0) node (2) {\((X , \theta [\sigma ]^X )\)}
        (1) node [above=1ex] (1') {\(\C^{\underline{G}}\)}
        (1') node [left ] {\(\Phi \colon\;\)}
        (2) node [above=1ex] (2') {\(\C^{G}\)}
    (0,-1) node (3) {\((Y, \sigma^Y )\)}
    (1,-1) node (4) {\((Y , \theta [\sigma ]^Y )\)}
[mor]
    (1') edge (2')
    (1) edge [|-stealth] (2)
    (3) edge [|-stealth] (4)
    (1) edge
        node [left ] {\(\varphi\)}
        (3)
    (2) edge
        node [right] {\(\varphi\)}
        (4)
;
\end{tikzpicture}

%% file: diagramas/teo-iso-CG-5.tikz
\begin{tikzpicture}[xscale=2.0]
\path
    node (P0) at (0,0) {$T_g (X \otimes \underline{\;})$}
    node (P1) at (1,0) {$T_g (X)\otimes T_g (\underline{\;})$}
    node (P2) at (2,0) {$X\otimes T_g (\underline{\;})$}
    node (P3) at (0,-1) {$T_g (Y \otimes \underline{\;})$}
    node (P4) at (1,-1) {$T_g (Y)\otimes T_g (\underline{\;})$}
    node (P5) at (2,-1) {$Y\otimes T_g (\underline{\;})$}
[nat,font=\scriptsize]
    (P0) edge node[above] {$(J^g_{X,\underline{\;}})^{-1}$} (P1)
    (P1) edge node[above] {$\theta^X_g \otimes T_g (\underline{\;})$} (P2)
    (P3) edge node[below] {$(J^g_{Y,\underline{\;}})^{-1}$} (P4)
    (P4) edge node[below] {$\theta^Y_g \otimes T_g (\underline{\;})$} (P5)
    (P0) edge node[left]
    {$T_g (\varphi \otimes \underline{\;})$} (P3)
    (P1) edge node[left]
    {$T_g (\varphi ) \otimes T_g(\underline{\;})$} (P4)
    (P2) edge node[left]
    {$\varphi  \otimes T_g(\underline{\;})$} (P5)
;
\end{tikzpicture}

%% file: diagramas/teo-iso-CG-6.tikz
\begin{tikzpicture}[xscale=2]
\path
    (0,0) node (1) {\((X, \theta^X )\)}
    (1,0) node (2) {\((X , \sigma [\theta ]^X )\)}
        (1) node [above=1ex] (1') {\(\C^{G}\)}
        (1') node [left ] {\(\Psi \colon\;\)}
        (2) node [above=1ex] (2') {\(\C^{\underline{G}}\)}
    (0,-1) node (3) {\((Y, \theta^Y )\)}
    (1,-1) node (4) {\((Y , \sigma [\theta ]^Y )\)}
[mor]
    (1') edge (2')
    (1) edge [|-stealth] (2)
    (3) edge [|-stealth] (4)
    (1) edge
        node [left ] {\(\varphi\)}
        (3)
    (2) edge
        node [right] {\(\varphi\)}
        (4)
;
\end{tikzpicture}

%% file: diagramas/teo-iso-CG-7.tikz
\begin{tikzpicture}[xscale = 3]
\path
    (1, 0) node (11) {\(T_g(X \otimes \1)\)}
    (2, 0) node (12) {\(T_g(X) \otimes T_g(\1)\)}
    (3, 0) node (13) {\(X \otimes T_g(\1)\)}
    (2,-1) node (22) {\(T_g(X) \otimes \1\)}
    (3,-1) node (23) {\(X \otimes \1\)}
    (1,-2) node (31) {\(T_g(X)\)}
    (2,-2) node (32) {\(T_g(X)\)}
    (3,-2) node (33) {\(X\)}
    (31) -- coordinate (_) (11)
        (32) -- (_) node [pos=2/3] {(I)}
    (11) -- coordinate (_) (12)
        (32) -- (_) node [pos=2/3] {(II)}
    (12) -- node {(III)} (23)
    (22) -- node {(IV)} (33)
[mor]
    (11) edge
        node [above] {\((J_{X,\1}^g)^{-1}\)}
        (12)
    (12) edge
        node [above] {\(\theta_g^X \otimes T_g(\1)\)}
        (13)
    (11) edge
        node [fill=white] {\(T_g(r_X)\)}
        (32)
    (12) edge
        node [fill=white] {\(T_g(X) \otimes ({J^g}^0)^{-1}\)}
        (22)
    (13) edge
        node [right] {\(X \otimes ({J^g}^0)^{-1}\)}
        (23)
    (22) edge
        node [fill=white] {\(\theta_g^X \otimes \1\)}
        (23)
    (22) edge
        node [right] {\(r_{T_g(X)}\)}
        (32)
    (23) edge
        node [right] {\(r_X\)}
        (33)
    (31) edge
        node [left] {\(T_g(r_X^{-1})\)}
        (11)
    (32) edge
        node [below] {\(\theta_g^X\)}
        (33)
[igual]
    (31) edge (32)
;
\end{tikzpicture}

%% file: diagramas/teo-iso-CG-8.tikz
\begin{tikzpicture}[xscale=2.3,yscale=1.5]
\foreach \l/\c/\nome in {
    1/1/T_g(X) \otimes T_g(Y),
    1/2/T_g(X \otimes \1) \otimes T_g(Y),
    1/3/X \otimes T_g(\1) \otimes T_g(Y),
    1/4/X \otimes \1 \otimes T_g(Y),
    2/1/T_g(X \otimes Y),
    2/2/T_g(X \otimes \1 \otimes Y),
    2/3/X \otimes T_g(\1 \otimes Y),
    2/4/X \otimes T_g(Y),
    3/2/T_g(X \otimes Y),
    3/3/X \otimes T_g(Y)
}{
    \node (\l\c) at (\c-1,1-\l){\(\nome\)};
}
\path [mor,font=\tiny]
    (11) edge
        node [above=1ex] {\(T_g(r_X^{-1}) \otimes T_g(Y)\)}
        (12)
    (12) edge
        node [above=1ex] {\((\sigma_g^X)_{\1} \otimes T_g(Y)\)}
        (13)
    (13) edge
        node [above=1ex] {\(X \otimes ({J^g}^0)^{-1} \otimes T_g(Y)\)}
        (14)
    (21) edge
        node [above] {\(T_g(r_X^{-1} \otimes Y)\)}
        (22)
    (22) edge
        node [above] {\((\sigma_g^X)_ {\1 \otimes Y}\)}
        (23)
    (23) edge
        node [above] {\(X \otimes T_g(\ell_Y)\)}
        (24)
    (32) edge
        node [below] {\((\sigma_g^X )_ Y\)}
        (33)
    (11) edge [stealth-]
        node [left ] {\((J_{X,Y}^g)^{-1}\)}
        (21)
    (12) edge [stealth-]
        node [fill=white] {\((J_{X \otimes \1, Y}^g)^{-1}\)}
        (22)
    (13) edge
        node [fill=white] {\(X \otimes J_{\1,Y}^g\)}
        (23)
    (14) edge
        node [right] {
            \(\begin{aligned}
                r_X &\otimes T_g(Y) = \\
                &= X \otimes \ell_{T_g(Y)}
            \end{aligned}\)
        }
        (24)
    (22) edge
        node [fill=white] {\(T_g(X \otimes \ell_Y)\)}
        (32)
    (23) edge
        node [fill=white] {\(X \otimes T_g(\ell_Y)\)}
        (33)
    (11) edge
        node [above] {\(\)}
        (12)
[igual]
    (21) edge (32)
    (24) edge (33)
;
\path
    (11) -- node {(I)} (22)
    (12) -- node {(II)} (23)
    (13) -- node {(III)} (24)
    (21) -- coordinate (_) (32)
        (_) -- (22)
        node [pos=1/3] {(IV)}
    (22) -- node {(V)} (33)
    (24) -- coordinate (_) (33)
        (_) -- (23)
        node [pos=1/3] {(VI)}
;
\end{tikzpicture}

%% file: diagramas/teo-CG-mon-1.tikz
\begin{tikzpicture}[xscale=3]
\path
    node (P1) at (0, 0) {\(T_g(X  \otimes Y  \otimes \esp)\)}
    node (P2) at (0,-1) {\(X  \otimes Y  \otimes T_g(\esp)\)}
    node (P5) at (1, 0) {\(T_g(X' \otimes Y' \otimes \esp)\)}
    node (P6) at (1,-1) {\(X' \otimes Y' \otimes T_g(\esp)\)}
[nat]
    (P1) edge node [left ] {\( \sigma_g^{X \otimes Y} \)} (P2)
         edge node [above] {\( T_g ( \alpha \otimes \beta \otimes \esp ) \)} (P5)
    (P2) edge node [below] {\( \sigma_g^{X' \otimes Y'} \)} (P6)
    (P5) edge node [right] {\( \alpha \otimes \beta \otimes T_g ( \esp ) \)} (P6)
;
\end{tikzpicture}

%% file: diagramas/teo-CG-mon-2.tikz
\begin{tikzpicture}[xscale=2]
\path
    node (P1) at (  0, 0) {\(T_g(X  \otimes Y  \otimes \esp)\)}
    node (P2) at (  0,-2) {\(X  \otimes Y  \otimes T_g(\esp)\)}
    node (P3) at ( .8,-1) {\(X  \otimes T_g( Y  \otimes \esp )\)}
    node (P4) at (2.4,-1) {\(X' \otimes T_g( Y' \otimes \esp )\)}
    node (P5) at (3.2, 0) {\(T_g(X' \otimes Y' \otimes \esp)\)}
    node (P6) at (3.2,-2) {\(X' \otimes Y' \otimes T_g(\esp)\)}
[black]
    (P1.center) -- node {(I)  } (P4 -| P5)
    (P3) -- node {(II) } (P3 -| P1)
    (P4) -- node {(III)} (P4 -| P6)
    (P2.center) -- node {(IV) } (P4 -| P5)
[black,nat]
    (P1) edge node [left ] {\( \sigma_g^{X \otimes Y} \)} (P2)
         edge node [above] {\( T_g ( \alpha \otimes \beta \otimes \esp ) \)} (P5)
         edge node [right=1ex] {\( \sigma_g^X \)} (P3)
    (P2) edge node [below] {\( \alpha \otimes \beta \otimes T_g ( \esp ) \)} (P6)
    (P3) edge node [right=1ex] {\( X \otimes \sigma_g^Y \)} (P2)
         edge node [above] {\( \alpha \otimes T_g( Y \otimes \esp) \)} (P4)
    (P4) edge node [left =1ex] {\( \sigma_g^{Y'}\ \)} (P6)
    (P5) edge [edge label' = {\( \sigma_g^{X'}\ \)}] (P4)
         edge node [right] {\( \sigma_g^{X' \otimes Y'} \)} (P6)
;
\end{tikzpicture}

%% file: diagramas/teo-CG-mon-3.tikz
\begin{tikzpicture}[xscale=3]
\path
    (0,1) node (1) {\(T_g(\1 \otimes \esp)\)}
    (0,0) node (2) {\(T_g(\1_{g^{-1}} \otimes \esp) \)}
    (1,0) node (3) {\(T_g(\1_{g^{-1}}) \otimes T_g(\esp) \)}
    (2,0) node (4) {\(\1_g \otimes T_g(\esp) \),}
    (2,1) node (5) {\(\1 \otimes T_g(\esp) \)}
[nat]
    (1) edge node [left ] {\Iso} (2)
    (2) edge node [above] {\(J^g\)} (3)
    (3) edge node [above] {\((J^{g0})^{-1} \otimes T_g(\esp)\)} (4)
    (4) edge node [right] {\Iso} (5)
;
\path [nat,dashed]
    (1) edge node [above] {\(\sigma^{\1}\)} (5)
;
\end{tikzpicture}

%% file: diagramas/teo-iso-smash.tikz
\begin{tikzpicture}
\path
    (1,1) node (12) {\( \C \smp G \)}
    (0,1) node (21) {\( \C \)}
    (2,1) node (23) {\( \underline{G} \)}
    (1,0) node (32) {\( \mathcal{D} \)}
[mor]
    (21) edge node [above] {\( \phi_0 \)} (12)
    (21) edge node [left ] {\( \phi   \)} (32)
    (23) edge node [above] {\( \pi _0 \)} (12)
    (23) edge node [right] {\( \pi    \)} (32)
;
\path [dashed,mor]
    (12) edge node [right] {\( \Psi \)} (32)
;
\end{tikzpicture}